\documentclass[oneside,11pt, reqno]{amsart}

\usepackage{amsmath,amssymb,amsthm, mathrsfs, mathtools}
\usepackage{graphicx}
\usepackage{amscd,mathtools}
\usepackage{enumitem} \setlist{nosep} % Change enumeration labelling
\usepackage[utf8]{inputenc}
\usepackage[english]{babel}
\usepackage{color}
\usepackage[pagebackref,breaklinks,hidelinks,unicode]{hyperref}
\usepackage{floatrow}
\usepackage{float}
\usepackage{comment}
\usepackage{mathtools}
\usepackage{csquotes}
\usepackage[font=small,justification=centering]{caption}
\usepackage{subcaption}
\usepackage[dvipsnames]{xcolor}
\usepackage{mathabx} % Includes \sqsubset etc.
\usepackage{xfrac} % Nice quotients
\usepackage[normalem]{ulem} % Contains \sout
\usepackage{dirtytalk}
\usepackage{lipsum}

\usepackage[margin=3cm]{geometry}

\usepackage{tikz}
\usepackage{tikz-cd}
\usetikzlibrary{automata}
\usetikzlibrary{calc,decorations.pathreplacing}
\usetikzlibrary{arrows}
\usetikzlibrary{decorations.pathreplacing}
\usetikzlibrary{intersections}

\makeatletter
    \newtheorem*{rep@theorem}{\rep@title}
    \newcommand{\newreptheorem}[2]{%
    \newenvironment{rep#1}[1]{%
    \def\rep@title{#2 \ref{##1}}%
    \begin{rep@theorem}}%
    {\end{rep@theorem}}}
\makeatother

\newtheorem{theorem}{Theorem}[section]

\newreptheorem{proposition}{Proposition}

\newtheorem{lemma}[theorem]{Lemma}

\newtheorem{Set up}[theorem]{Set-up}

\theoremstyle{definition}

\newtheorem{definition}[theorem]{Definition}

\newtheorem{example}[theorem]{Example}

\newtheorem{remark}[theorem]{Remark}

\newtheorem*{answer*}{Answer}
\newtheorem*{application*}{Application}
\newtheorem*{warning}{Warning}

\newcommand{\id}{\text{id}}

\DeclarePairedDelimiterX{\Norm}[1]{\lVert}{\rVert}{#1}
\theoremstyle{definition}

  \newcommand{\calB}{\mathcal{B}}
  \newcommand{\calC}{\mathcal{C}}

  \newcommand{\calG}{\mathcal{G}}
  \newcommand{\calH}{\mathcal{H}}

  \newcommand{\calL}{\mathcal{L}}

  \newcommand{\calU}{\mathcal{U}}
  \newcommand{\calV}{\mathcal{V}}

\renewcommand*{\backrefalt}[4]{\ifcase #1 (Not cited).\or (Cited p.~#2).\else (Cited pp.~#2).\fi} % Cited on...

\newcounter{shcount}
 
\newcounter{enumlabelcount}
\makeatletter\newcommand\enumlabel[1][]{\item[#1]
    \refstepcounter{enumlabelcount}\def\@currentlabel{#1}}\makeatother

\definecolor{harrycomment}{rgb}{0.6,0,0.4}

\DeclareMathOperator{\Aut}{Aut}

\DeclareMathOperator{\dist}{\mathsf{d}}
\DeclareMathOperator{\Dist}{\mathsf{D}}
\DeclareMathOperator{\diam}{diam}

\DeclareMathOperator{\cay}{Cay}

\newcommand{\hull}{\mathrm{hull}}

\DeclareMathOperator{\isom}{Isom}

\newcommand*{\X}{X_{\Dist}}

\newcommand*{\cal}{\mathcal}

\newcounter{claimcount}

\newenvironment{claim*}[1]{\par\vspace{2mm}\noindent
    \underline{Claim:}\hspace{2mm}#1}{}

\title[Injectivity, cubical approximations and equivariant wall structures]{injectivity, cubical approximations and equivariant wall structures beyond CAT(0) cube complexes}
\hypersetup{pdftitle = Injectivity and cubical approximations}
\author{Abdul Zalloum}
\address{Mathematics department, 40 St. George Street, University of Toronto, Toronto, Canada}
\email{abdul.zalloum@utoronto.ca}
\begin{document}

\begin{abstract} This is an expository survey with two goals. 
\begin{enumerate}
    \item The primary goal is to discuss and highlight the impact of two recent influential ideas in geometric group theory. The first of which is the notion of an \emph{injective metric space} which is a rich class of spaces that was imported to geometric group theory by Lang \cite{LANG2013} and have shown to be of a great effect. The second is Behrstock-Hagen-Sisto's \emph{cubical approximation} theorem \cite{HHS_quasi} which provides a novel and particularly successful approach for studying mapping class groups (of finite type surfaces) and more generally, hierarchically hyperbolic groups.

%which a particularly powerful tool for studying mapping class groups and more generally hierarchichally hyperbolic groups which was established by Behrstock-Hagen-Sisto in \cite{HHS_quasi}. The second of which is the notion of \emph{injectivity}; a rich class of metric spaces that was imported to geometric group theory by Lang in \cite{LANG2013}.

\item Our second goal is to demonstrate how numerous geodesic metric spaces including hyperbolic spaces, CAT(0) spaces, and hierarchically hyperbolic spaces admit a strikingly rich equivariant wall structure: a discovery that was inspired by the aforementioned machines; the cubical approximation theorem as well as injective metric spaces. The presence of such a wall structure was first highlighted in \cite{PSZCAT} by Petyt, Spriano and the author.

\end{enumerate}
%Unlike hyperplanes in CAT(0) cube complexes, there is no upper bound on the number of pairwise intersecting curtains in any of the aforementioned spaces. Nonetheless, and even in the context of CAT(0) cube complexes, the above generality of curtains provides a wall structure which is invariant under the \emph{entire} isometry group of the CAT(0) space $X$, a property that cubical hyperplanes fail to enjoy as can be seen from rotations on $\mathbb{R}^2$. In fact, given a CAT(0) space $X$ and $G<\isom(X),$ the space $X$ admits a cube complex structure acted upon by $G$ if and only if there is a $G$-invariant discrete subcollection of curtains. 
\end{abstract}

\maketitle

    %Finally, and quite surprisingly, the lack of a bound on the number of pair-wise interesting curtains seems to be rarely relevant as such curtains still capture a great deal of the geometry of the underlying spaces and a wealth of cubical tools effortlessly push through to all of the above contexts. In the context of CAT(0) spaces, even when the CAT(0) the generality of curtains.
    
    %The is an expository survey that aims to explore the connections between the hyperplane structure of CAT(0) cube complexes on one hand and their hyperplane structure on one hand and 3 more recent notions of non-positive curvature: injective and hierarchichally hyperbolic spaces as well as curtains in CAT(0) spaces. The main theme of the survey is to show that cubical tools can be exported to stud, including the presence of a strikingly rich hyperplane structure
\setcounter{tocdepth}{1}\tableofcontents\setcounter{tocdepth}{2}

\part{Introduction and heuristic discussions}

    \section{Two influential ideas and some of their impact}\label{sec:two_ideas}

    Over the past few years, a bulk of ideas have been introduced to geometric group theory with great effect. Among these, two are most relevant to this survey:

    \vspace{2mm}

\noindent (1) \textbf{Injective metric spaces}. The first is Lang's work \cite{LANG2013} which brought \emph{injective metric spaces} to attention in geometric group theory. Recall that a metric space $X$ is said to be injective if it's geodesic and every collection of pair-wise intersecting closed balls $\{B_i\}_{i \in I}$ have a total intersection, that is, a point $x$, that lives in $\cap_{i \in I} B_i$. Equivalently, a metric space $X$ is injective if whenever $Y \subset Z$ and $f:Y \rightarrow X$ is a 1-Lipshitz map, then $f$ extends to a 1-Lipshitz map $\tilde{f}:Z \rightarrow X$ with $f(y)=\tilde{f}(y)$ for all $y \in Y,$ see  \cite{Aronszajn1956EXTENSIONOU}. It is worth noting that the notion of injectivity is \textbf{not} a coarse notion but rather a very fine one, for example, injective metric spaces are contractible and share many fundamental properties with CAT(0) spaces. One such common feature is that, in the same way CAT(0) cube complexes $X$ can be equipped with a CAT(0) distance, they can also be equipped with a distance $\dist_{\infty}$ making $(X, \dist_{\infty})$ an injective metric space. Despite their similarity with CAT(0) spaces, injective metric spaces seem to be much more tractable, for example, while the question of whether hyperbolic groups are CAT(0) is still wide open, in \cite{LANG2013}, Lang established the remarkable theorem that hyperbolic groups admit proper cocompact actions on injective metric spaces. %Similarly, while the question of whether infinite torsion groups exist in CAT(0) groups remains open, in \cite{Haettel-Osajda22}, Haettel and Osajda show that such subgroups don't exist in groups admitting a geometric action on locally finite injective graphs 

More recently, Haettel-Hoda-Petyt \cite{HHP} showed that mapping class groups of finite type surfaces $G$ and more generally hierarchically hyperbolic groups admit geometric (proper co-bounded) actions on certain injective metric spaces $X_\text{inj}$. Given the fine, CAT(0)-like nature of injective metric spaces, such a statement is quite surprising especially in light of the fact that mapping class groups of finite type surfaces can't admit proper actions (by semi-simple isometries) on CAT(0) spaces as shown by Kapovich and Leeb in \cite{Kapovich1996} (see also \cite{Bridson2010}), for the rest of the survey, when we say mapping class groups we mean mapping class groups of finite-type surfaces. A subsequent work of Sisto and the author \cite{Sisto-Zalloum-22} shows that Haettel-Hoda-Petyt's injective metric spaces $X_{\text{inj}}$ provide the first known \emph{geometric models} for mapping class groups $G$ where pseudo-Anosovs have strongly contracting quasi-axis; by a geometric model for $G$ we mean a space $X$ acted on properly co-boundedly by $G$; and hence equivariently quasi-isometric to any of its Cayley graph. In fact, and relying on work of Yang, \cite{Yang2018}, the presence of such strongly contracting quasi-axis in $X_\text{inj}$ confirms Thurston's conjecture on density of pseudo-Anosov orbits in the geometric model $X_{\text{inj}}$ of the underlying mapping class group (see also work of Choi \cite{Choi2021PseudoAnosovsAE} who showed that each mapping class group admit a finite generating set where the previous density conjecture of Thurston holds). Furthermore, combining \cite{HHP}, \cite{Sisto-Zalloum-22} with work of Arzhantseva, Cashen and Tao \cite{Arzhantseva2015} shows that the space $X_{\text{inj}}$ is the first known geometric model for mapping class groups where they admit a \emph{growth tight} action; see \cite{Arzhantseva2015} for a detailed discussion of the notion. Finally, although geodesics in Cayley graphs of mapping class groups are ill-behaved and can backtrack an arbitrary amount in the curve graph $\calC S$, \cite{Rafi2021}, geodesics in $X_{\text{inj}}$ all define uniform-quality reparametrized quasi-geodesics in $\calC S$; this is established in a forthcoming work of Rafi and the author. 
\vspace{2mm}

The aforementioned statements suggest that Haettel-Hoda-Petyt's injective space $X_{\text{inj}}$ is a more appropriate geometric model for mapping class groups compared to their Cayley graphs, and they encourage further investigation of injective metric spaces and groups. Aside from cubulated and hierarchically hyperbolic groups, the class of injective metric spaces and groups admitting geometric actions on them is broad and it includes finitely presented graphical C(4)-T(4) small cancellation groups \cite{CCGHO20}, weak Garside groups of finite type and FC-type \cite{Huang2021}, some Euclidean buildings and simplicial complexes \cite{Haettel2021} as well as many symmetric spaces \cite{Haettel2022}.

%the perhaps tell us a couple of things, first, it seems like the appropriate geometric model for mapping class groups $\{G\}$ (and HHGs more generally) is the HHP injective metric space $X_\text{inj}$ as opposed to their Cayley graphs. Second, injective metric spaces provide a very promising venue to understand both objects we have already cared about for decades such as surfaces and mapping class groups, as well as newly intrdouced objects and conects such as hierarchichally hyperbolic spaces and growth tight actions.
        %\vspace{2mm}

\noindent (2) \textbf{Cubical approximations}. The second influential machine relevant to the survey is work of Behrstock, Hagen and Sisto \cite{HHS_quasi} who built upon ideas of Bowditch \cite{Bowditch13} to show that the ``convex hull" of a finite set of points in a mapping class group (and more generally, an in an HHS) is ``median" quasi-isometric to a CAT(0) cube complex, we shall refer to this as the \emph{cubical approximation} theorem. Informally, this says that the smallest ``convex" set containing a finite set of points in a mapping class group ``looks like" a CAT(0) cube complex. Such a theorem provides a strong bridge between the worlds of mapping class groups and CAT(0) cube complexes allowing one to export a wealth of cubical tools to study the latter. This very philosophy has been successfully implemented in various recent works leading to the resolution of two long-standing open problems in the field including Farb's quasi-flats conjecture which was established in \cite{HHS_quasi} by Behrstock-Hagen-Sisto as well as semi-hyperbolicity of mapping class groups as shown independently by Durham-Minsky-Sisto in \cite{DMS20} and by Haettel-Hoda-Petyt in \cite{HHP}. Interestingly, the above two independent proofs emerged simultaneously three decades after semi-hyperbolicity of mapping class groups was conjectured, and despite their vastly different approaches, both of their proofs rely heavily on this cubical approximation machine. It's worth noting that more recently, Bowditch has given a different proof to Farb's quasi-flats conjecture under some surprisingly weak assumptions \cite{Bowditch2019QUASIFLATSIC}.

Finally, we discuss two recent generalizations of the cubical approximation machine. The first of which is work of Durham and the author in \cite{Durham-Zalloum22} were we extend the cubical approximation theorem allowing the finite set to include ``boundary points". Namely, we show that ``convex hulls" of finitely many points of the compactification of hierarchically hyperbolic groups (in the sense of \cite{Durham2017-ce}) are also median quasi-isometric to CAT(0) cube complexes. The utility of such a statement is that it allows one to use CAT(0) cube complex techniques to analyze not only the interior of an HHS, but also the various boundary notions it comes with, see \cite{Durham-Zalloum22} for more details as well as \cite{Abbott-Medici-23} by Abbott and Incerti-Medici for more applications. The second of which is Petyt's recent surprising theorem that mapping class groups (and more generally, colorable hierarchically hyperbolic groups) are ``median"-quasi-isometric to CAT(0) cube complexes providing a globalization of the cubical approximation machine \cite{Petyt21}. Unlike the cubical approximation theorem \cite{HHS_quasi} and its boundary extension \cite{Durham-Zalloum22}, Petyt's quasi-isometry is not equivariant. That being said, such a theorem is the best one can expect in this context as mapping class groups can't geometrically act on CAT(0) spaces \cite{Kapovich1996}, \cite{Bridson2010} or properly on CAT(0) cube complexes as shown by Genevois in \cite{Genevois2019MedianSO}.

      The purpose of this survey is twofold. First, to explore and highlight the recent influence of the above two statements. Second, to show that numerous spaces of non-positive curvature enjoy a rich equivariant wall structures; a discovery that was inspired by these two ideas, and first highlighted in \cite{PSZCAT} by Petyt, Spriano and the author.

      \section{What will you find in this survey? A detailed summary}\label{sec:what_will_you_find} In \underline{Part 1} of the survey, you will find the following.
      
\vspace{2mm}

\noindent \underline{\textbf{Section \ref{sec:two_ideas}}} discusses the two primary objects driving the rest of the survey, injective metric spaces as well as the cubical approximation theorem. It also highlights the recent impact of these two ideas on the field.

      \vspace{2mm}

      \noindent \underline{\textbf{Section \ref{sec:what_will_you_find}}} is this.

      \vspace{2mm}

     \noindent \underline{\textbf{Section \ref{sec:intro_CCC}}} briefly discusses CAT(0) cube complexes which on one hand are the primary examples of injective metric spaces and on the other, they provide the building blocks of the cubical approximation machine. In particular, the section describes the fundamental tools used to study CAT(0) cube complexes including their hyperplanes, median structures as well as their $\dist_\infty$-distance which makes them injective metric spaces.
     
      \vspace{2mm}

\noindent \underline{\textbf{Section \ref{sec:intr_coarse_median}}} discusses the class of coarse median spaces \cite{Bowditch13} which is both a pre-requisite to understanding the cubical approximation machine and a primary element to showing that mapping class groups as well as HHGs act geometrically on injective metric spaces.

\vspace{2mm}

\noindent \underline{\textbf{Section \ref{sec:Intro_Walls}}}       initiates the discussion regarding the existence of equivariant wall structures beyond CAT(0) cube complexes. 

\vspace{2mm}
     
 \noindent \underline{\textbf{Section \ref{sec:intro_their_connection}}} provides a brief sketch of Haettel-Hoda-Petyt's proof that hierarchically hyperbolic groups admit geometric actions on injective metric spaces which relies on the cubical approximation machine and is inspired by Bowditch's work on median spaces \cite{BOWDITCH2018}. It also discusses how their proof along with the cubical approximation theorem inspired Petyt, Spriano and the author's observation that many spaces of non-positive curvature admit equivariant wall structures. Finally, we discuss examples where such equivarint wall structures (called \emph{curtains}) are present.
 
\vspace{2mm}

\noindent \underline{\textbf{Section \ref{sec:why_care_intro}}} provides a glimpse of the utility of these curtains. For instance, it shows how curtains in HHGs can be used to describe their fundamental geometric components including their distance, median structures, convex sets and hierarchy paths.

\vspace{2mm}

\noindent \underline{\textbf{Section \ref{sec:Marrying_intro}}} summarizes some of the progress made over the past five years regarding the connections between mapping class groups, CAT(0) spaces and injective metric spaces.
\vspace{2mm}

In \underline{Part 2}, we provide detailed discussions of injective metric spaces, CAT(0) cube complexes, HHSes, CAT(0) spaces and we list many tools for studying them, some of which are well-known while others are more modern. More precisely:

\vspace{2mm}

\noindent \underline{\textbf{Section \ref{sec:Injective}}} defines injective metric spaces, explores aspects of their accommodating geometry and discusses the main philosophy used to prove statements about them. In particular, we demonstrate how the universal property of injective metric spaces has many immediate powerful consequences regarding their geometry. This includes the existence of a collection of fellow-travelling geodesics connecting every pair of points called a \emph{geodesic combing}, the existence of an isometrically embedded tripod for any given three points, and more generally, the presence of a center associated to any finite set of points. Finally, we discuss injective hulls and highlight the canonical nature of injective metric spaces, both from a general metric geometry perspective and from a geometric group theory one.

\vspace{2mm}

\noindent \underline{\textbf{Section \ref{sec:CCC and their hyperplanes}}} discusses CAT(0) cube complexes in more details and demonstrates how their geometry is fully described by their hyperplanes. This is most apparent by their ultrafilters and by Sageev's construction; both of which are discussed in this section.

\vspace{2mm}

\noindent \underline{\textbf{Section \ref{sec:HHSes}}} gives a brief overview of the theory of hierarchical hyperbolicity and demonstrates how their entire coarse geometry is recorded via the collection of hyperbolic spaces $\calU=\{U_i
\}_{i \in I}$ they come with. We also describe how the collection $\calU$ contains an exceptionally powerful hyperbolic spaces $\calC S$; which we refer to as \emph{the omniscient}. It might be good to already remark that while HHSes usually come with various HHS structures, the omniscient refers to a special hyperbolic space that comes with a specific HHS strcutrue introduced in \cite{ABD}. Perhaps one can say ``in the beginning there was nothing, and then Abbott, Behrstock and Durham came up with the omniscient" \cite{ABD}. We then provide many examples of useful constructions and statements in HHSes that utilize their entire collection of hyperbolic spaces $\calU$--such as the presence of a coarse median \cite{HHS2}--  in addition to ones that only utilize their omniscient $\calC S$, such as the Morse local-to-global property established by Russell, Spriano and Tran in \cite{Morse-local-to-global}.

\vspace{2mm}

\noindent \underline{\textbf{Section \ref{sec:from_HHS_to_curtains}}} discusses how starting with a CAT(0) cube complex, its collection of hyperplanes $\calH$ can be used to build hyperbolic spaces $\calU=\{U_i\}_{i \in I}$ (such as the contact graph of Hagen \cite{hagen:weak} and the separation space of Genevois \cite{genevois:hyperbolicities}) which often turn the cube complex into an HHS; namely, as in the following diagram: $$ X_{\text{CAT(0) cube complex }} \rightarrow	 \calH_{\text{hyperplanes }} \rightarrow 
\calU_{\text{hyperbolic spaces }} \rightarrow X_{\text{ HHS }}.$$

It then proposes the question: to which extent can the above diagram be reversed? Namely, starting with an arbitrary HHS (not necessarily a cube complex) and its collection of hyperbolic spaces $\calU=\{U_i\}_{i \in I}$, it asks whether these hyperbolic spaces can be used to provide ``hyperplanes" reversing the second arrow of the above diagram? And if so, whether these hyperplanes can be used to build a CAT(0) cube complex, reversing the first?

We then explain that while the answer to the second question is known to be negative, the answer to the first one is positive and in a very strong sense. In particular, we have the following diagram:

$$ \calH_{\text{curtains}} \leftarrow  \calU_{\text{hyperbolic spaces }}\leftarrow X_{\text{HHS} }.$$

\vspace{2mm}

\noindent \underline{\textbf{Section \ref{sec:CAT(0)_curtains}}} discusses curtains in the context of CAT(0) spaces \cite{PSZCAT} where we compare them with the ones present in the class of HHSes. Moreover, we describe how these curtains provide the building blocks for the \emph{curtain model}; a counter part of the curve graph which is a universal $\delta$-hyperbolic space for rank-one elements of the underlying CAT(0) space $X.$ Finally, we describe what it means for a geodesic in the CAT(0) space $X$ to make distance in the curtain model.

\vspace{2mm}

\noindent \underline{\textbf{Section \ref{sec:out}}} provides a summary of the ideas and objects discussed within the survey.

\vspace{2mm}

\noindent \underline{\textbf{Section \ref{sec:questions}}} contains some questions relating to the subjects discussed in the survey.

\vspace{2mm}

\noindent \underline{\textbf{Appendix A}}. This survey includes an idea and a few related lemmas that haven't already appeared in the literature. The new idea is the fact that hyperbolic spaces, HHSes and many other spaces with non-positive curvature admit equivariant wall structures that can be used to describe many interesting aspects of their geometry. The presence of these wall structures and every new lemma appearing in this survey (all of which revolve around these wall structures) were obtained in collaboration with Petyt and Spriano. The proof of these new lemmas is the content of Appendix A.

    \subsection*{Acknowledgments} First and foremost, I am very grateful to Anthony Genevois for an extremely thorough and helpful set of comments on the first draft of this survey, his comments have improved the exposition drastically. Furthermore:

    \begin{enumerate}
        \item I am very thankful to Jenny Beck for giving a very nice talk where I have learnt many of the constructions in Section \ref{sec:Injective} on injective spaces, in particular, the title of that section is inspired by the title of Jenny's talk.

        \item  I am quite grateful to Mathew Durham for patiently explaining to me the proof and the utility of the Behrstock-Hagen-Sisto's cubical approximation machine.

\item I would like to very much thank Thomas Haettel for some very useful remarks on an earlier draft of the survey and for coordinating with me regarding the connections between this survey and his forthcoming one on Helly groups and injective metric spaces.

\item  I am especially thankful to Mark Hagen for teaching me numerous cubical techniques and getting me to appreciate CAT(0) cube complexes and their hyperplanes, this survey would truly not have been possible without him.

\item I would like to greatly thank Urs Lang for some useful feedback regarding the first draft of this survey.

 \item  I am immensely grateful to Harry Petyt for some very helpful comments on an earlier draft of this survey, for thoroughly explaining his paper with Haettel and Hoda to me and for showing me many of the constructions that utilize injectivity.

        \item I am very thankful to Kasra Rafi for various discussions regarding the survey, for his continuous support and for always being exceptionally generous with his time.

        \item Much gratitude goes to Alessandro Sisto for useful comments regarding the first draft of this survey. 
        
        \item I am very grateful to Davide Spriano for teaching me many aspects of hierarchical hyperbolicity.

        \item  Warmest thanks go to my collaborators Harry Petyt and Davide Spriano for being kind enough allowing me to extract a few statements from our ongoing projects due to their relevance to the survey; all the new lemmas and proofs appearing in this survey were obtained in collaboration with them.

\item Finally, I would like to thank all members of the mathematics department at the University of Utah including Mladen Bestvina, Kenneth Bromberg, Elizabeth Field, Vivian He, Michael Kopreski, Priyam Patel, Alexander Rasmussen, and especially George Shaji for their hospitality during my visit between the 5th and 14th of April 2023 where the first draft of this survey was completed.

    \end{enumerate}

\vspace{4mm}
Since we intend for the survey to be self-contained, we start by discussing CAT(0) cube complexes.

\section{CAT(0) cube complexes, their median  structure and injective distance}\label{sec:intro_CCC} Putting the formal definition aside, a CAT(0) cube complex $X$ can be thought of as a contractible space which is built by gluing unit Euclidean cubes $[0,1]^k$ isometrically along faces. It's said to be \emph{finite dimensional} if the largest $k$ where $[0,1]^k \subset X$ is bounded above by an integer $n$; this integer is called \emph{the dimension} of $X.$ For instance, if the  3-cube and all the other 2-cubes in Figure \ref{fig:A hyperplane} are filled in, then the ambient space is a CAT(0) cube complex and otherwise it's not. Each cube $[0,1]^k$ comes with $k$ distinct \emph{midcubes} which are defined to be the restriction of one coordinate of $[0,1]^k$ to $\frac{1}{2}$. Finally, a \emph{hyperplane} is what results from starting with a midcube $c$, if $c$ touches a midcube $c'$ in an adjacent cube, we consider $c \cup c'$ and continue inductively, that is, if $c'$ touches a midcube $c''$ in another adjacent cube, we consider $c \cup c' \cup c''$, etc. See Figure \ref{fig:A hyperplane} where we start with a midcube given by a single blue edge, and extend it to the full blue hyperplane.

\begin{figure}[ht]
   \includegraphics[width=8cm, trim = .001cm 5cm 2cm 3cm]{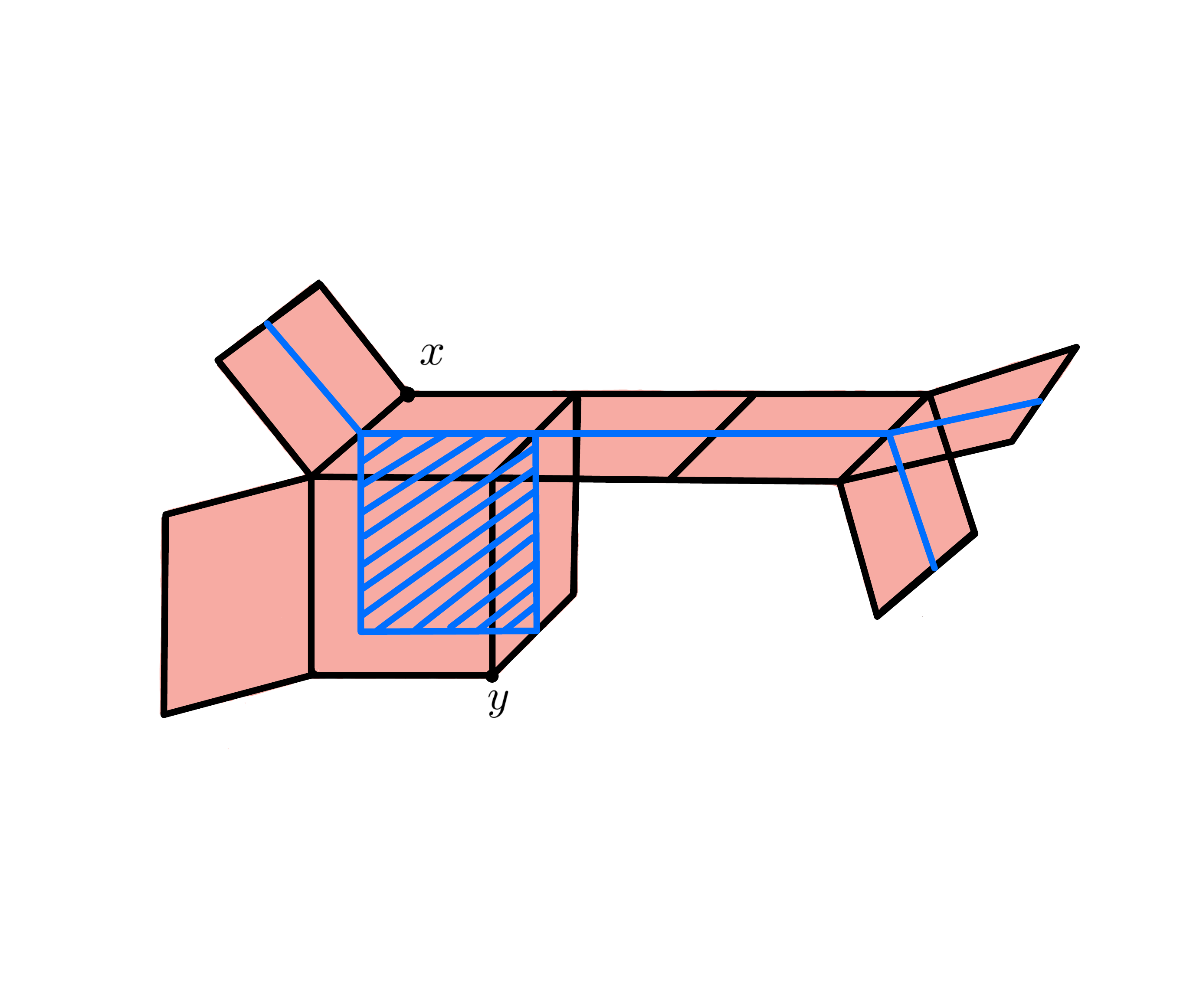}\centering
\caption{A hyperplane.} \label{fig:A hyperplane}
\end{figure}

As you can see from Figure \ref{fig:A hyperplane}, each hyperplane $h$ determines two disjoint subsets $h^+,h^-$ called \emph{half spaces} satisfying $X=h^+ \cup h^-$. We say that two points $x,y \in X$ (and more generally, two subsets $A,B \subset X$) are \emph{separated} by $h$ if one of them lives in $h^+$ and the other lives in $h^-.$ The standard metric to consider a CAT(0) cube complex with is the \emph{combinatorial} or the \emph{taxi cap} metric, denoted $\dist$, which for any pair of vertices $x,y \in X$ records the number of hyperplanes $h$ separating $x,y.$ Equivalently, its the metric that results from equipping each cube $[0,1]^k$ with the distance $$\dist((x_1,\cdots x_k), (y_1,\cdots y_k))=|x_1-y_1|+\cdots |x_k-y_k|,$$ and extending the metric $\dist$ to $X$ in the obvious way. By a \emph{combinatorial geodesic} or simply a \emph{geodesic} we mean a geodesic with respect to the combinatorial metric $\dist.$ 

Hyperplanes and their half spaces are all convex (combinatorial geodesics with end points on a hyperplane remain in that hyperplanes), and it's known that convex sets in a CAT(0) cube complex satisfy the \emph{Helly property}: a collection of pair-wise intersecting convex sets $C_i$ totally intersect, that is, if $C_i \cap C_j \neq \emptyset
$ for all $C_i,C_j \in \{C_i\}_{i \in I}$, then there is a point $x \in \cap_{i \in I} C_i.$

It's not hard to see that an intersecting pair of hyperplanes must intersect in a given cube. In particular, convexity of such hyperplanes combined with the Helly property shows that the cardinality of pair-wise intersecting hyperplanes is bounded above by the dimension of the CAT(0) cube complex $X$, assuming it's finite dimensional.

Each CAT(0) cube complex $X$ admits a metric $\dist_{\infty}$ so that $(X,\dist_\infty)$ is an injective metric space. The $\dist_{\infty}$-distance on a CAT(0) cube complex $X$ is defined by declaring each two distinct vertices $x,y$ of the same cube to be at distance 1, see Figure \ref{fig:infinite_distance}.

%Here are a couple of properties that are satisfied by the $\dist_{\infty}$ injective metric:
%\vspace{1mm}
%\begin{enumerate}
  %  \item (Forgetting the dimension) Independently of $n$, the injective distance $\dist_{\infty}(x,y)=1$ for any two distinct vertices $x,y$ in the same cube $[0,1]^n.$ That is, the $\dist_{\infty}$-distance \emph{forgets} the dimension of the CAT(0) cube complex as it does not record the number of intersecting hyperplanes in a given cube.
   % \vspace{1mm}

  %  \item (Chain distance) Since intersecting hyperplanes aren't recorded by the injective distance $\dist_{\infty}$, we have: $$\dist_{\infty}(x,y)= \text{sup}\{|c| \text{ such that } c \text{ is a chain of hyperplanes separating }x,y
%\}.$$ 

%\end{enumerate}

%That is, $\dist_{\infty}(x,y)=1$ for any two distinct vertices in $[0,1]^n$, \emph{independently of $n.$} Since intersecting hyperplanes are not recorded in the $\dist_{\infty}$-distance, its not hard to see that $$\dist_{\infty}(x,y)= \text{sup}\{|c| \text{ such that } c \text{ is a chain separating }x,y
\begin{figure}[ht]
   \includegraphics[width=10cm, trim =.001cm 7cm 2cm 6cm]{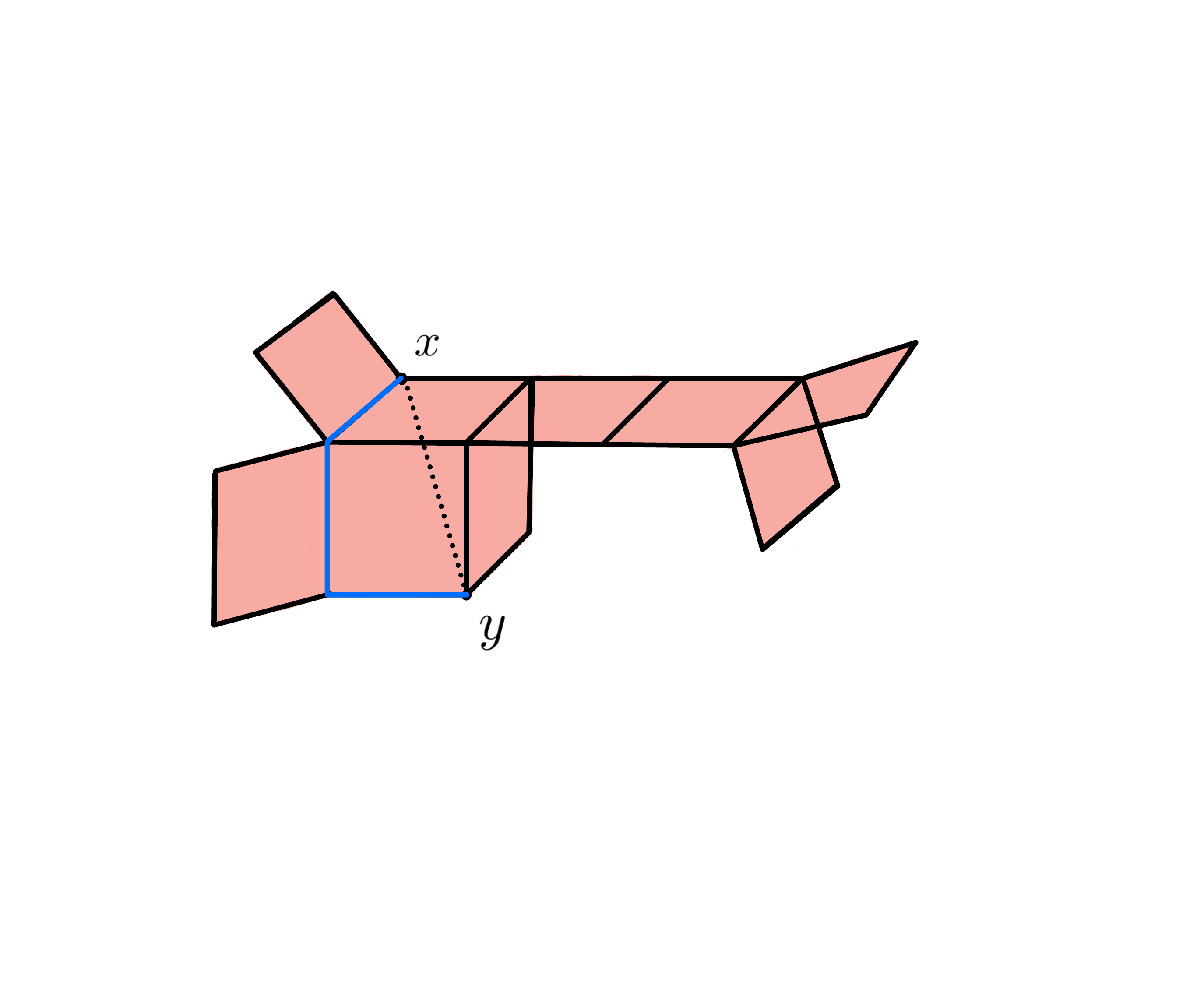}\centering
\caption{The blue edges represent a combinatorial geodesic showing that $\dist(x,y)=3$. On the other hand, $\dist_\infty(x,y)=1.$} \label{fig:infinite_distance}
\end{figure}

Recall that a \emph{chain} is a collection of hyperplanes $c=\{h_1,\cdots h_k\}$ such that $h_{i}$ separate $h_{i-1}$ from $h_{i+1}.$ In particular, elements of $c$ are pair-wise disjoint.

Here is a summary of the main properties the injective distance $\dist_{\infty}$ on a CAT(0) cube complex satisfy:

\vspace{2mm}

\noindent 1) \underline{Forgetting the dimension}: With respect to the combinatorial distance $\dist$, the diameter of a cube $[0,1]^n$ is exactly $n.$ Namely, the points $(0, \cdots, 0), (1,\cdots, 1)$ are at distance $n$ as they are separated by exactly $n$-hyperplanes. On the other hand, and independently of $n$, the injective distance $\dist_{\infty}(x,y)=1$ for any two distinct vertices $x,y$ in the same cube $[0,1]^n.$  Hence, the injective $\dist_{\infty}$-distance on a cube $[0,1]^n$ \emph{forgets} its dimension, or equivalently, it disregards the contribution of intersecting hyperplanes to the total distance.

\vspace{2mm}

\noindent 2) \underline{Chain distance}: Since intersecting hyperplanes aren't recorded by the $\dist_{\infty}$-injective distance on a cube, it follows that for any pair of vertices $x,y$ in a cube complex $X$, we have $$\dist_{\infty}(x,y)= \text{sup}\{|c|\,:\, c \text{ is a chain of hyperplanes separating }x,y
\}.$$ 

Finally, the combinatorial distance $\dist$ on a CAT(0) cube complex $X$ makes $X$ into a \emph{median metric space}. Recall that a metric space is said to be median if for any $x_1,x_2,x_3$, there is a unique point $m=m(x_1,x_2,x_3)$ such that $$\dist(x_i,x_j)=\dist(x_i,m)+\dist(m,x_j),$$ for all $i \neq j.$ In the language of hyperplanes, this median is exactly what results from orienting each hyperplane towards the majority of $\{x_1,x_2,x_3\}$ (choosing the half space that contains at least two points among $\{x_1,x_2,x_3\}$) and taking the intersection of all such half spaces.

\section{Coarse median spaces}\label{sec:intr_coarse_median}

A space is said to be \emph{coarse median} if every three points $x_1,x_2,x_3$ admit a point $\mu$ called the coarse median (or simply, the median) that satisfies certain coarse conditions similar to that of above inequality.

A subset $A$ of a coarse median space is said to be \emph{$K$-median-convex} if for any $x,y \in A,$ and any $z \in X$, the point $\mu(x,y,z)$ lies in the $K$-neighborhood of $A.$ For two coarse median spaces $X_1,X_2$, a map $f:X_1 \rightarrow X_2$ is said to be $K$-median if $\dist (f(\mu(c_1,c_2,c_3)), \mu(f(c_1),f(c_2),f(c_3))) \leq K$. A path $\gamma:[a,b] \rightarrow X$ is said to be a \emph{$K$-median path} if it is a $K$-median map in the aforementioned sense. For a reader who is familiar with HHSes, it is worth noting that $K$-median paths are exactly $K'$-hierarchy paths, for instance, see \cite{petyt:onlarge} or \cite{Durham-Zalloum22}. Finally, a \emph{$K$-median quasi-isometry} $f:X_1 \rightarrow X_2$ is a $(K,K)$-quasi-isometry which is $K$-median.

\section{Equivariant wall structures}\label{sec:Intro_Walls} As we shall thoroughly discuss in the upcoming sections, CAT(0) spaces, hyperbolic and hierarchically hyperbolic spaces admit equivariant wall structures that record many interesting aspects of their geometry. In particular, if $X$ is any of the aforementioned spaces, then there is a collection of $\Aut(X)$-invariant subspaces $\calH=\{h_i\}_{i \in I}$, called \emph{curtains} satisfying the following conditions (and more, but the precise definitions will be given later):

\vspace{2mm}

\noindent 1) \underline{Curtains separate:} Each $h$ is coarsely path-connected and $X-h \subseteq h^+ \sqcup h^+$ where $h^+,h^-$ are disjoint path-connected subspaces of $X.$

 \vspace{2mm}

 \noindent  2) \underline{Chain distance:} There exists a constant $K=K(X)$ such that for any $x,y \in X$, we have 
    
    $$\dist(x,y) \underset{K}{\sim} \text{sup}\{|c|\, :\, c \text{ is a chain of curtains separating }x,y
\}.$$

The notation  $\underset{K}{\sim}$ here means the quantities are equal up to an additive and multiplicative error depending only on $K.$

For instance, in the setting of CAT(0) spaces, a curtain in $X$ is defined to be the preimage of a unit interval $I$ along a geodesic $\alpha$ under the nearest point projection map $\pi_\alpha:X \rightarrow \alpha$. Unlike hyperplanes in finite-dimensional CAT(0) cube complexes, there is no upper bound on the number of pairwise intersecting curtains in any of the aforementioned spaces. However, and even in the context of CAT(0) cube complexes, the above generality of curtains provides a wall structure which is invariant under the \emph{entire} isometry group of the CAT(0) space $X$, a property that cubical hyperplanes fail to enjoy as can be seen from rotations on $\mathbb{R}^2$. In fact, given a CAT(0) space $X$ and $G<\isom(X),$ the space $X$ admits a cube complex structure acted upon by $G$ if and only if there is a $G$-invariant discrete subcollection of curtains.

\section{And their connections with injectivity and cubical approximations}\label{sec:intro_their_connection} As we just mentioned, the wall structures enjoyed by CAT(0) spaces, hyperbolic and hierarchically hyprbolic spaces fail to be \emph{discrete}: an arbitrary pair of points $x,y$ are generally separated by an infinite collection of walls. However, the cardinality of a maximal \emph{chain} of curtains separating $x,y$ is always finite, and in fact, it coarsely agrees with $\dist (x,y).$ This is reminiscent of the fact that the $\dist_{\infty}$-injective distance on a CAT(0) cube complex $X$ forgets its dimension as it disregards intersecting hyperplanes and only counts the maximal chain of hyperplanes separating $x,y \in X$; hinting at some connections between injectivity and these equivariant wall structures.

 An $E$-\emph{hierarchically hyperbolic space} is a quasi-geodesic metric space $(X,\dist)$ that comes with a collection of $E$-hyperbolic spaces $\calU=\{U_i\}_{i \in I}$ and $(E,E)$-coarsely-Lipshitz surjections $\{\pi_{U_i}:X \rightarrow U_i\}_{i \in I}$ satisfying certain conditions, see \cite{HHS2}. The constant $E$ is called the \emph{HHS constant}, it quantifies the ``coarseness" of the space and it bounds the error to which certain statements/inequalities fail to hold. As shown in \cite{HHS2}, HHSes are coarse median spaces.

 As we will describe in Section \ref{subsec:fundemental_HHSes}, the collection of hyperbolic spaces $\calU=\{U_i\}_{i \in I}$ and maps $\pi_{U_i}:X \rightarrow U_i$ describe the entire coarse geometry of an HHS. For instance, the celebrated \emph{distance formula} theorem \cite{masurminsky:geometry:1}, \cite{masurminsky:geometry:2}, \cite{HHS2} states that the distance between a pair of points $x,y \in X$ is roughly the sum of the distances between $\pi_U(x),\pi_U(y)$, where the sum is taken over all $U \in \calU$ where the $\dist(\pi_U(x), \pi_U(y))$ is substantial (at least $E$).

Since HHSes are coarse median spaces, we can talk about median convexity and median paths we have just described. Subsets $A \subseteq X$ which are $K$-median convex with $K=E$, the HHS constant, will simply be referred to as \emph{median convex} sets as the constant is uniform. Similarly, a \emph{median path} is an $E$-median path. Finally, up to increasing $E$ by a uniform amount depending only on $X$, the \emph{median hull} of a set $A$ can be defined to be the intersection of the $E$-neighborhoods all $E$-median-convex sets containing $A.$ It is worth noting that there are various other definitions for median hulls and convex hulls in HHSes and coarse median spaces, however, thanks to \cite{RST18} and \cite{Bowditch2019CONVEXITY}, all such notions are equivalent and one can simply operate with the previous definition.

%Finally, if $Y$ is a median convex set in an HHS, each point $x \in X$ determines a coarsely unique point in $Y$ called the \emph{gate} of $x$ to $Y$ and denoted by $P_Y(x)$. Using Lemma 2.57 in Durham Zalloum, the point $P_Y(x)$ is coarsely the unique point that lives in $\bigcap_{y \in Y}[x,y]$, hence, the reader might take that for the definition of a gate. The existence of the gate was originally established by Behrsotck-Hagen-Sisto \cite{HHS2} and was studied further in \cite{RST18}.

%The point $P_Y(x)$ is also coarsely the unique point satisfying $\dist(P_Y(x), \mu(x,y,P_Y(x)))<E$ for all $y \in Y.$

%the \emph{median hull} of a subset $A$, denoted $\hull(A),$ is defined as follows: first, define $A^1$ to be the union of all median paths starting and ending on $A,$ and define $A^n:=(A^{n-1})^1$. Thanks to \cite{RST18} and \cite{Bowditch2019CONVEXITY}, there exist constants $k,E'\geq E$ depending only on $X$ such that $A^k$ is median convex and $A^{k'} \subset N(A^k,E')$ for all $k ' \geq k.$ Finally, the median hull of $A$, denoted $\hull(A)$ is defined to be $A^k$ where $k$ is as above. Equivalently, $\hull(A)$ can be defined to be the intersection of all $E'$-median convex sets containing $A$

Now we are ready to state the cubical approximation theorem.

\begin{theorem} [{\cite[Theorem F]{HHS_quasi}}] \label{thm:cubical_approximation}
    Given an HHS $X$, there exists a constant $v$ such that the following holds. For any finite $F\subset X$, there exists a constant $K=K(|F|, X)$ such that $\hull(F)$ is $K$-median-quasiisometric to a CAT(0) cube complex $Q$ with $\dim(Q) \leq v.$
\end{theorem}

A particularly interesting interpretation of the previous theorem is that HHSes (whatever they are) can be thought of as higher dimensional hyperbolic spaces: in the same way median hulls of finite sets of points in hyperbolic spaces are quasi-isometric to trees, median hulls of HHSes are median quasi-isometric to ``higher dimensional trees", i.e. CAT(0) cube complexes. One reason CAT(0) cube complexes can be thought of as higher dimensional trees is that a CAT(0) cube complex is of dimension 1 if and only if it is tree. It is worth noting that in \cite{Durham-Zalloum22}, the authors extend the above theorem to include points from the ``boundary" of $X,$ that is, the finite set $F$ doesn't need to only contain points from $X$, but it can also include median rays.

In \cite{HHP}, Haettel, Hoda and Petyt exploit Theorem \ref{thm:cubical_approximation} to export the $\dist_{\infty}$-injective distance from the approximating cube complexes $Q$ to the underlying HHS $X.$ Since their construction is important for this survey, we shall record their theorem and briefly discuss the proof.

\begin{theorem} [{\cite[Theorem A]{HHP}}]\label{thm:HHP} Let $(X, \dist)$ be an HHS. There exists a coarsely injective metric $\rho_\infty$ on $X$ which is quasi-isometric to $\dist.$
\end{theorem}

Recall that a geodesic metric space is said to be \emph{coarsely injective} if there exists a constant $\delta$ such that for any collection of pair-wise intersecting balls $\{B(x_i,r_i)\}_{i \in I}$ there exists a point $x \in \cap_{i \in I} B(x_i, r_i+\delta)$. The primary example to be kept in mind for these are hyperbolic spaces, see Theorem 6.1 in \cite{Breuillard2021} and Proposition 1.1 in \cite{HHP}.

\vspace{2mm}

\subsubsection{A brief summary of the proof of Theorem \ref{thm:HHP}}

Let $(X, \dist)$ be an HHS with $\calU=\{U_i\}_{i \in I}$ as its collection of hyperbolic spaces. Inspired by Bowditch \cite{BOWDITCH2018}, they define $$\rho_\infty(x,y):=\text{sup}\{|f(x)-f(y)|\,:\, f:X \rightarrow \mathbb{R} \text{ is a coarse median contraction} \},$$ where a coarse median contraction is a map $f:X \rightarrow \mathbb{R}$ that coarsely decreases distances and it roughly preserves medians. They then rely on cubical approximations (Theorem \ref{thm:cubical_approximation}) to show that $(X,\rho_\infty)$ is quasi-isometric to $(X,\dist)$, and on the coarse Helly property for metric balls in the hyperbolic spaces $U \in \calU$ to establish an analogous Helly property for metric balls in $(X, \rho_\infty),$ concluding the proof.

The main ingredient of their new distance $\rho_\infty$ is the notion of a coarse median contractions $f:X \rightarrow \mathbb{R}.$ Observe that when $X$ is a CAT(0) cube complex and the maps $f:X \rightarrow \mathbb{R}$ in the definition of $\rho_\infty$ above are taken to be median contractions (i.e., with no coarseness), then we have $\dist_\infty=\rho_\infty$. To see this, recall that $\dist_\infty(x,y)$ on a CAT(0) cube complex is the cardinality of the maximal chain of hyperplanes separating $x,y$, and such hyperplanes are exactly the fibers $\{f^{-1}(m)\}_{m \in \mathbb{Z}}$ of the median contraction $f:X \rightarrow \mathbb{R}$ sending $x,y$ as far as possible to $\mathbb{R}$, see Figure \ref{fig:contractions}.  This remark that hyperplanes in CAT(0) cube complexes can be seen as pre-images of median contractions $f:X \rightarrow \mathbb{R}$ allows us to make sense of a hyperplane in any coarse median space $X$, namely, as a pre-image of a point (or an interval) under a coarse median contractions $f:X \rightarrow \mathbb{R};$ an observation that was made by Petyt, Spriano and the author. In fact, one can make the following more general remark:

\begin{figure}[ht]
    \centering
    \begin{minipage}{0.45\textwidth}
        \centering
        \includegraphics[width=7cm, trim = .01cm 5cm 5cm 5.7cm]{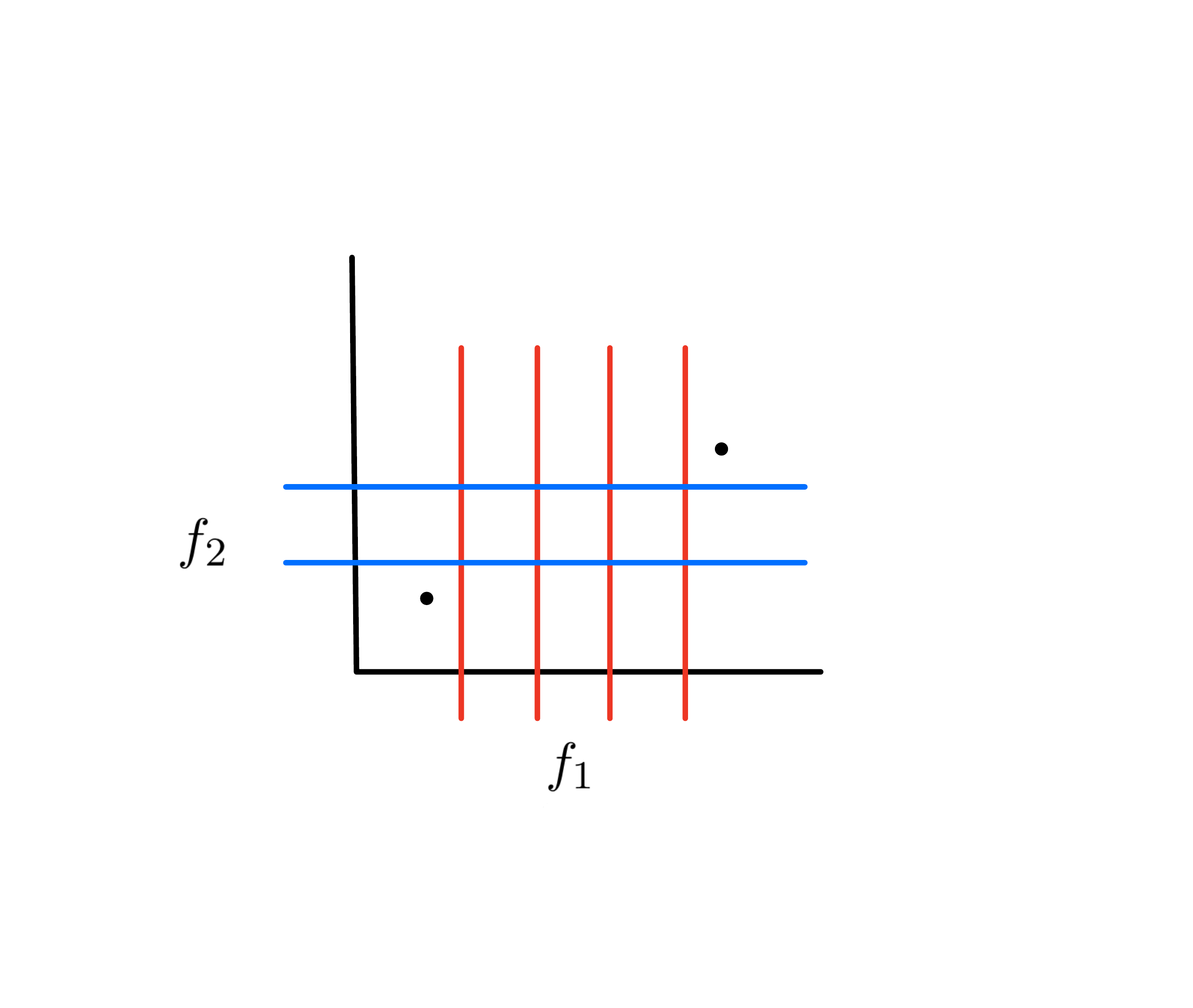} % first figure itself
        \caption{The two points in each of the pictures are $x=(1,1)$, $y=(5,3)$ which differ on 4 vertical hyperplanes and 2 horizontal ones. The picture on the left shows the two obvious median contractions $f_1,f_2$ which are nearest point projections to the $x$ and $y$ axes; $|f_1(x)-f_1(y)|=4$ while $|f_2(x)-f_2(y)|=2$. The picture on the right shows a third median contraction $f$ which is $f_1$ followed by compressing the $x$-axis so that $|f(x)-f(y)|=1$. Among such three contractions, the optimal one is $f_1$.}\label{fig:contractions}
    \end{minipage}\hfill
    \begin{minipage}{0.45\textwidth}
        \centering

        \includegraphics[width=7cm, trim = .01cm 5cm 5cm 5.7cm]{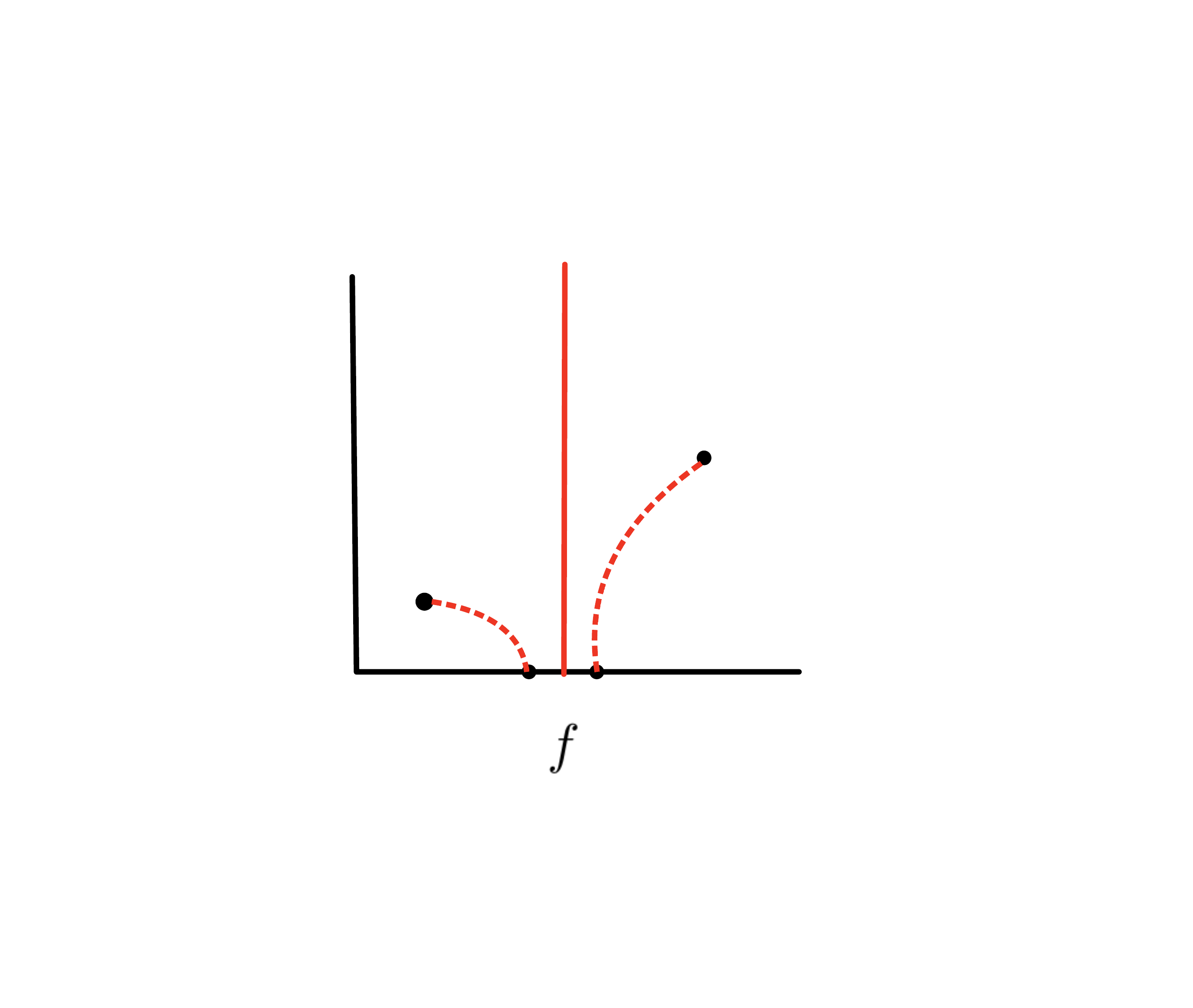} % second figure itself
    \end{minipage}
\end{figure}

\vspace{2mm}

\subsubsection{Main observation}
Inspired by the above, Petyt, Spriano and the author observed the following. Whenever you have a ``well-behaved" contraction $f$ from an arbitrary space $X$ onto $\mathbb{R}$, or more generally, onto a CAT(0) cube complex $Q$, you can attempt to define ``curtains" or ``hyperplanes" in $X$ as pre-images of hyperplanes in $Q.$ The more ``well-behaved" the surjection $f$ is, the stronger properties these hyperplanes enjoy. Namely, the quality of the surjection $f$ dictates the quality of the resulting fibers. For instance, if $X$ is a CAT(0) cube complex, and $f:X \rightarrow \mathbb{R}$ is a \emph{median} map, then fibers of hyperplanes $f^{-1}(p_i)$ --with $p_i$ a point, i.e., a hyperplane in $\mathbb{R}$-- are themselves convex.

\subsubsection{Examples where this observation applies}\label{subsub:examples_where_observation_applies}
Here are a few examples of the above phenomenal:

\begin{example} \label{ex:First_Example}(Hyperbolic and hierarchically hyperbolic spaces) Motivated by the above, if $X$ is an arbitrary metric space and $f_0:X \rightarrow Y$ is a coarse surjection to a hyperbolic space $Y$, then, one obtains a coarse surjection to a cube complex of dimension 1 as follows: compose $f_0$ with the nearest point projection map $\pi_l:Y \rightarrow l$ to a geodesic line $l$ in $Y$ (doesn't need to be a line, but it's easier to imagine) to obtain $f_l=\pi_l \circ f_0:X \rightarrow l \cong \mathbb{R}$. One can now define the ``$l$-hyperplanes" as the fibers $\{h^l_n=f_l^{-1}([n, n+10\delta])\}_{n \in \mathbb{Z}},$ and ``hyperplanes" are simply $l-$hyperplanes for some geodesic line $l \subset Y.$

\end{example}

The following remark highlights further connections between CAT(0) cube complexes and metric spaces that come with surjections onto hyperbolic spaces.

\begin{remark} (Hyperplanes versus unparameterized quasi-geodesics) Here is a fun and easy to see fact. Continuing with Example \ref{ex:First_Example} above, it is immediate that a path $\alpha \subset X$ projects to an unparameterized quasi-geodesic in $Y$ if and only if $\alpha$ never crosses the same hyperplane in $X$ twice, modulo some constants. Namely, since hyperplanes were defined to be pre-images under $f_l:X \rightarrow l$ of ``thick" intervals $I=[n,n+10\delta] \subset l \subset Y,$ crossing a hyperplane $h^l_n=f_l^{-1}([n, n+10\delta])$ twice (by which we mean starting with a half space $(h^l_n)^+$, entering the other half space $(h^l_n)^-$ and returning to $(h^l_n)^+$) implies backtracking along $l$ for a long time; the converse is slightly trickier but it can be easily established by median considerations, see the proof to item 2 of Lemma \ref{lem:sample_append}. For a reader who is familiar with the language of HHSes, the previous observation will be used to show that a quasi-geodesic in an HHS is a hierarchy path if and only if it never crosses a hyperplane twice (modulo some constants), where hyperplanes will be defined in essentially the same way as in the previous example (we will take every geodesic of every hyperbolic space $U \in \calU$ and define hyperplanes analogously).

\end{remark}

One class where surjections to CAT(0) cube complexes are particularly abundant is the class of CAT(0) spaces.

\begin{example}(CAT(0) spaces) A fundamental and almost characteristic property of a CAT(0) space $X$ is the existence and uniqueness of nearest point projections to convex sets. In particular, for a geodesic $\alpha \subset X$, there is 1-Lipshitz map $\pi_\alpha:X \rightarrow \alpha$ assigning to each $x \in X$ its closest point along $\alpha.$ This provides a natural collection of ``hyperplanes" called \emph{curtains} which are simply fibers $\pi_\alpha^{-1}(I)$ of unit intervals $I \subset \alpha$. 

%As we mentioned before, and unlike hyperplanes in CAT(0) cube complexes, there is no upper bound on the number of pairwise intersecting curtains in CAT(0) spaces (notice: this is reminiscent of the fact that the $\dist_{\infty}$ injective distance on a CAT(0) cube complex forgets its dimension, yet another connection with injectivity). However, and even in the context of CAT(0) cube complexes, the generality of curtains provides a wall structure which is invariant under the \emph{entire} isometry group of the CAT(0) space $X$, a property that cubical hyperplanes fail to enjoy as can be seen from rotations on $\mathbb{R}^2$. In fact, given a CAT(0) space $X$ and $G<\isom(X),$ the space $X$ admits a cube complex structure acted upon by $G$ if and only if there is a $G$-invariant discrete subcollection of curtains. 

%In \cite{PSZCAT}, Petyt, Spriano and the author show that such curtains describe a great deal of the underlying geometry of the CAT(0) space. For instance, using such curtains, one can associate a $\delta$-hyperbolic space for any CAT(0) space $X$, called the \emph{curtain model}, where \emph{all} rank-one elements of $\isom(X)$ act loxodromically. More details regarding the curtain model will be given in Sections \ref{sec:why_care_intro} and \ref{sec:CAT(0)_curtains}.

\end{example}

  A good class of examples to keep in mind where the above observation applies is the collection of spaces admitting a \emph{strongly contracting} geodesic line. 
    
    \begin{definition}\label{def:Strongly_contracting} A geodesic line $l$ is said to be $D$-strongly contracting if the nearest point projection of any ball $B$ disjoint from $l$, denoted $\pi_l(B),$ has diameter at most $D.$ 

        \end{definition}

Essentially by the contra-positive of the above definition, one can easily prove that if $l$ is $D$-strongly contracting and $x, y\in X$ with $\dist(\pi_l(x), \pi_l(y))>5D$, then the subsegment of $l$ lying between $\pi_l(x), \pi_l(y)$ is contained in in the $6D$-neighborhood of any geodesic $\alpha$ connecting $x,y.$

    \begin{example} (Contracting geodesics)
    Given a strongly contracting geodesic line $l$, the nearest point projection provides a coarse surjection to a cube complex $\pi_l:X \rightarrow \mathbb{R}$ whose fibers $\{h_i:=\pi_l^{-1}(n)\}_{n \in 20D\mathbb{Z}}$ can be thought of ``hyperplanes" in $X.$ These ``hyperplanes" dissect the space in an efficient and convex manner, namely, observe that by the remark following Definition \ref{def:Strongly_contracting} above, if $\alpha$ crosses $h_i,h_{i+1}$, then it must spend its time between $h_i,h_{i+1}$ near $l$. In particular, the fact that $\alpha$ is a geodesic implies that $\alpha$ can't go back to cross $h_{i+1}, h_i.$ An alternative way of stating this is: every geodesic $\alpha \subset X$ projects to some $(D',D')$-unparameterized-quasi-geodesic $\pi_l(\alpha)$ in $\mathbb{R},$ where $D'$ depends only on $D.$

\end{example}

\section{Why care about general curtains? A sample lemma}\label{sec:why_care_intro} We have discussed three fundamental examples where curtains exist, namely, spaces with a strongly contracting geodesic, CAT(0) spaces, hyperbolic and hierarchically hyperbolic spaces. In this section, we explore the utility of such curtains in the context of hyperbolic and hierarchically hyperbolic spaces (although, notably, most of our constructions and statements solely involve the coarse median structures on these spaces \cite{Bowditch13}). For CAT(0) spaces, their utility is described in \cite{PSZCAT}, in particular, these curtains and their \emph{curtain model} (a $\delta$-hyperbolic space constructed via curtains) allow for vast extensions of theorems known in mapping class groups and CAT(0) cube complexes to that of CAT(0) spaces including an Ivanov-stlyle rigidity theorem, a dichotomy of a rank-rigidity flavor and the presence of a universal hyperbolic space for rank-one elements; a statement that wasn't known even for proper cocompact cube complexes.

Finally, these curtains also seem to provide some  powerful constructions in the context of spaces containing a strongly contracting geodesic (e.g., injective metric spaces \cite{Sisto-Zalloum-22} and Garside Groups \cite{Garsidestrong}) but the full picture is still being investigated by Petyt, Spriano and the author in some work-in-progress.

%Each HHS comes with a universal constant $E$. This constant describes the ``coarseness" of the space, that is, it bounds the error to which certain statements/equalities fail to hold. The reader who is familiar with HHSes and coarse median spaces should be warned that for the sake of this survey, we are taking $E$ to be a large enough constant (yet uniform for $X)$ so that numerous statements hold simultaneously, such a constant will be referred to as the \emph{HHS constant}. 

%Since HHSes are coarse median spaces, we can talk about median convexity and median paths we have just described in Section bla. Subsets $A \subseteq X$ which are $K$-median convex with $K=E$, the HHS constant, will simply be referred to as \emph{median convex} as the constant is uniform. Similarly, a \emph{median path} is an $E$-median path. Finally, the \emph{median hull} of a subset $A$, denoted $\hull(A),$ is defined to be the intersection of all median convex sets containing $A$. It is worth noting that there are various other definitions for median hulls, however, thanks to \cite{RST18} and \cite{Bowditch2019CONVEXITY}, all such notions are equivalent.

In the previous section, we have hinted at how one can define curtains in a geodesic metric space $X$ which come with a surjection to a CAT(0) cube complex $\pi_Y:X \rightarrow Y,$ namely, as preimages of hyperplanes in $Y$ under $\pi_Y$.

We have also discussed how this applies to spaces that come with surjections to hyperbolic spaces such as HHSes. Namely, if $X$ is an HHS with a collection of hyperbolic spaces $\calU=\{U_i\}_{i \in I}$ and maps $\{\pi_{U_i}: X \rightarrow U_i\}_{i \in I}$, then a reasonable definition of a curtain could be $(\pi_U \circ \pi_l)^{-1}(I)$ where $U \in \calU$ is a hyperbolic space, $l \subset U$ is a geodesic line, and $I \subset l$ is a (perhaps thick) interval. Furthermore, if we would like for these ``curtains" to be median convex, one can instead take the median hull of the set $(\pi_U \circ \pi_l)^{-1}(I)$.

The above subsets will be our primary source of curtains, however, in order to keep the definition as compact and flexible as possible, we choose to define HHS curtains as follows.
\begin{definition}(HHS curtains) Let $X$ be an HHS and let $E$ be the HHS constant. An HHS \emph{curtain} is defined to be a median convex set $h$ such that $X-h \subseteq h^+ \sqcup h^-$ where $h^+, h^-$ are disjoint, median-convex sets with $\dist(h^+, h^-) \geq 10E.$ See Figure \ref{fig:curtain}. 
\end{definition}

  \begin{figure}[ht]
   \includegraphics[width=8cm, trim = 1cm 7cm 1cm 8cm]{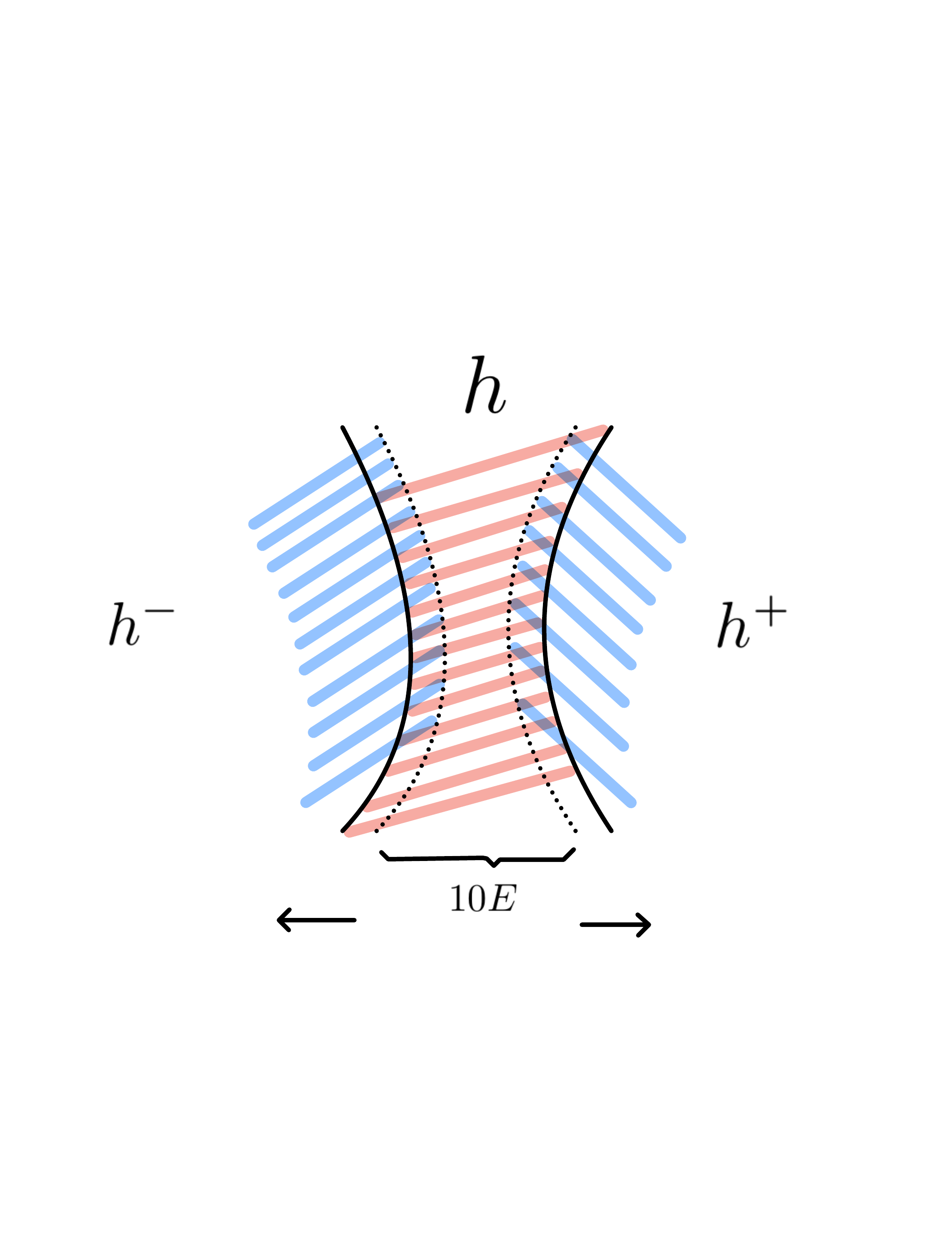}\centering
\caption{A curtain along with its half spaces.} 
\label{fig:curtain}
\end{figure}

It's immediate from the definition that any action of a group $G$ on $X$ preserving medians must take curtains to curtains. The sets $h^+,h^-$ are called \emph{half spaces} for $h$ and a path $\alpha$ is said to \emph{cross} $h$ if the sets $\gamma \cap (h^+-h), \gamma \cap (h^--h)$ are both non-empty. Our primary examples of HHS curtains will come from the maps $\pi_U: X \rightarrow U$ where $U \in \calU$. Namely, it is immediate to check that the set $\hull[(\pi_U \circ \pi_l)^{-1}(I)]$ is an HHS curtain where $U, I, l$ as in the previous paragraph.

Recall that median convex sets in HHSes satisfy the \emph{coarse Helly property} \cite{HHP}. That is, if a family of median convex sets $\{C_i\}_{i \in I}$ satisfy $C_i \cap C_j \neq \emptyset$ where at least one $C_i$ is bounded or $|I|<\infty$, then there is a point $x \in X$ that lives in the intersection $\cap_{i \in I} N(C_i,E)$, where $N(C_i,E)$ denotes the $E$-neighborhood of $C_i$. In particular, this property applies to the above curtains and their half spaces when $I$ is finite. If $I$ is not finite, then the property still holds, but one needs to assure that there is a constant $K$ such that for all $i \neq j$ the diameter of $ C_i \cap C_j$ is bounded above by $K.$

The following lemma is extracted from some forthcoming work of Petyt, Spriano and the author. It is worth noting that our framework there is more general than HHSes and relies only the median structure of the space $X.$ However, for the purpose of completeness and self containment of the survey, in Appendix A, we will provide a different proof for the statement relying on the hierarchical structure of $X.$

\begin{lemma}[Petyt-Spriano-Zalloum]\label{lem:intro_lemma} Let $X$ be an HHS, up to increasing the HHS constant $E$ by a uniform amount, the following hold for any $x,y,z \in X.$

\begin{enumerate}

\item The distance $\dist(x,y)$ coarsely agrees with the size of a maximal chain of curtains separating $x,y.$ That is, $\dist(x,y) \underset{E}{\asymp} \text{max}\{|c|\, :\, c \text{ is a chain of curtains separating }x,y
\}.$

    \item A quasi-geodesic $\alpha$ is a hierarchy path (equivalently, a median path) if and only if it never crosses the same curtain twice.

    \item The median $\mu(x,y,z)$ is coarsely the unique point that lies in the intersection of the $E$-neighborhoods of all half spaces containing the majority of $x,y,z.$

    \item For any set $A$, the median hull of $A$ is coarsely the intersection of the $E$-neighborhoods of all half spaces properly containing $A.$

\end{enumerate}
    
\end{lemma}
Perhaps serving as an evidence of the canonical nature of such curtains, the proof of each item of the above relatively long statement doesn't exceed a few lines (see Appendix A for the proof). For a reader who is familiar with HHSes, it is worth noting that item 1 of Lemma \ref{lem:intro_lemma} above is merely a reflection of the distance formula combined with the Behrstock's inequality. Namely, making a large distance in a hyperbolic space $U \in \calU$ is equivalent to crossing many curtains coming from that hyperbolic space, and for two generic (transverse) hyperbolic spaces $U,V \in \calU,$ the Behrstock's inequality controls how curtains coming from $U$ interact with the ones coming from $V,$ namely, it assures that they form a chain.

   %After deleting at most $6$-curtains from $c,$ we obtain a subchain $c' \subseteq c$ with $|c'| \geq |c|-6$ and no curtain from $c'$ contains any of the points $x,y,z,$ in particular, for each $h \in c',$ at least one half space, say $h^+$, lies in $\calH.$ Further, since $c'$ is a chain, there is a pair of points in $\{x,y,z\}$, say $x,y$ which are contained in $h^+$ for each $h \in c'.$ If $k \in c'$ is the curtain closest to $\{x,y\}$ among curtains in $c'$, then $\mu=\mu(x,y,z)$ lies in the $E$-neighborhood of $\hull(k^+).$

\section{Marrying mapping class groups, CAT(0) groups and injective spaces}\label{sec:Marrying_intro}

\vspace{3mm}

While the tools we had for studying CAT(0) groups and mapping class groups as of five years ago were quite abundant, they were essentially mutually exclusive. 

\subsubsection{Tools that applied to CAT(0) groups but not to mapping class groups}

Unlike CAT(0) groups, a mapping class groups can't geometrically (properly coboundedly) act on CAT(0) spaces, and the only contractible space with fine (i.e., non-coarse) geometry that was known to admit an action by mapping class groups is Teichmüller space; however, such an action is not co-bounded. In short, and until very recently, we weren't aware of any fine, contractible geometric model (a space on which $G$ act geometrically) for mapping class groups. 

Furthermore, and unlike CAT(0) groups, mapping class groups weren't known to admit a geometric model where Morse (equivalently, pseudo-Anosov \cite{Behrstock_pA_are_Morse}) elements have strongly contracting quasi-axis. Having a strongly contracting quasi-axis is an extremely desirable property that is known to inform a great deal about the acting group, for instance, regarding genericity of pseudo-Anosov elements \cite{Yang2018}, growth tightness of the action \cite{Cashen2019}, acylindrical hyperbolicity, regularity of Morse geodesics \cite{EikeZalloum}, \cite{Morse-local-to-global} and many more. In fact, Rafi and Verberne show that in the obvious geometric models for mapping class groups, i.e., their Cayley graphs, axes of Morse elements can fail to be strongly contracting \cite{Rafi2021}.

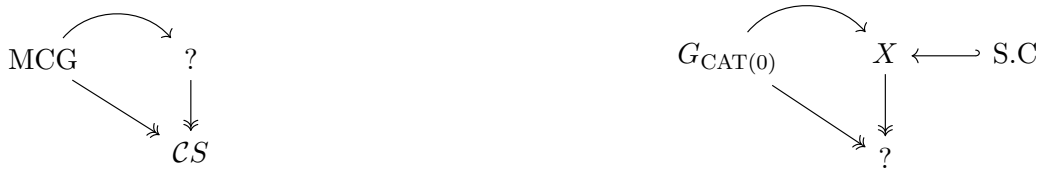
\begin{figure}[ht]
    \centering
\begin{tikzcd}  
\text{MCG} \arrow[r,bend left=50] \arrow[dr, two heads] & ?  \arrow[d, two heads]\\
 & \calC S\\
\end{tikzcd}%
    \qquad \hspace{50mm}
\begin{tikzcd}  
G_{\text{CAT(0)}} \arrow[r,bend left=50] \arrow[dr, two heads] & X\arrow[d, two heads]& \text{S.C}  \arrow[l, hook']\\
 & ?\\
\end{tikzcd}    \caption{No CAT(0)-like space for mapping class groups and no curve graph-like space for CAT(0) spaces}%
    \label{fig:before_marriage}%
\end{figure}

\subsubsection{Tools that applied to mapping class groups but not to CAT(0) groups}
On the other hand, unlike mapping class groups with their curve graphs, CAT(0) groups - including the cubulated ones - weren't known to admit an action on a hyperbolic space which is universal with respect to its Morse (equivalently, rank-one \cite{charneysultan:contracting}, \cite{Bestvina2009}) elements. Cubulated groups were known to act on many hyperbolic spaces with a great effect, including the \emph{contact graph} of Hagen and the \emph{separation space} of Genevois, however, neither of these is universal with respect to \emph{all} Morse elements of $G$, see Figure \ref{fig:before_marriage}.

\subsubsection{Marrying the tools} Over the past few years, it was discovered by \cite{HHP} that mapping class groups (and more generally HHGs) enjoy geometric actions on fine contractible metric spaces which are very close analogues of CAT(0) spaces; namely, injective metric spaces. Their theorem allows for a wealth of statements known in the context of CAT(0) groups to export to mapping class groups (and HHGs) including semi-hyperbolicity, the equivalence between Morse and strongly contracting geodesics \cite{Sisto-Zalloum-22} as well as the presence of only finitely many conjugacy classes of finite subgroup. See Figure \ref{fig:after_marriage}.

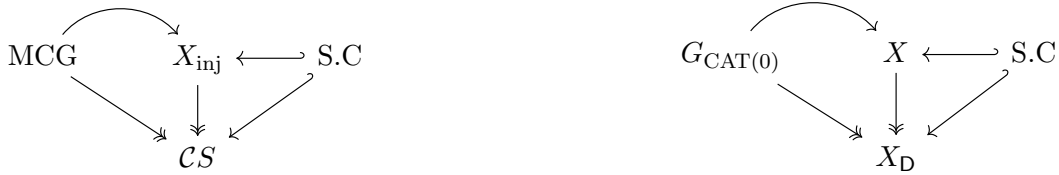
\begin{figure}[ht]
    \centering
\begin{tikzcd}  
\text{MCG} \arrow[r,bend left=50] \arrow[dr, two heads] & X_{\text{inj}}  \arrow[d, two heads]& \text{S.C}  \arrow[l, hook'] \arrow [dl, hook']\\
 & \calC S\\
\end{tikzcd}%
    \qquad \hspace{30mm}
\begin{tikzcd}  
G_{\text{CAT(0)}} \arrow[r,bend left=50] \arrow[dr, two heads] & X\arrow[d, two heads]& \text{S.C}  \arrow[l, hook'] \arrow [dl, hook']\\
 & X_{\Dist} \\
\end{tikzcd}    \caption{Mapping class groups admit a geometric action on an injective space where Morse elements have strongly contracting quasi axes, and CAT(0) groups admit a universal W.P.D action on the curtain model $X_{\Dist}$. In particular, all strongly contracting geodesics quasi-isometrically embed in $\calC S$ and $X_{\Dist}.$ }%
    \label{fig:after_marriage}%
\end{figure}

On the other hand, recent work of Petyt, Spriano and the author shows that each CAT(0) group $G$ admits an action on a very close relative of the curve graph: a $\delta$-hyperbolic space called the curtain model and denoted $\X$. Further, in analogy with mapping class groups actions on their curve graphs, the action of $G$ on the curtain model $\X$ is W.P.D \emph{universal} meaning that all Morse elements of $G$ act loxodromially on $\X$ as W.P.D isometries.

\part{Detailed discussions}

%For a CAT(0) cube complex $X$, a geodesic ray $b$ is $\kappa$-sublinearly Morse (where $\kappa$ is a sublinear function) if and only if:

%\begin{enumerate}
 %   \item $b$ crosses infinitely many bottlenecks ($L$-separated hyperplanes) where the bottlneck width (the separation constant $L$) increases at a $\kappa(t)$-rate
%\end{enumerate}

\section{ Injective metric spaces and their exceptionally tolerating geometry} \label{sec:Injective}

\subsection{Their definition and aspects of their accommodating geometry}\label{subsec:their_accomodating_geometry}

An object $X$ in a category $\calC$ is said to be \emph{injective} if if for any morphism $f:Y \rightarrow X$ and any monomorphism $i:Y \rightarrow Z,$ there is a morphism $\Tilde{f}:Z \rightarrow X$ that restricts to $f$ over $Y,$ namely, we have $\Tilde{f}\circ i=f.$ An \emph{injective metric space} is a metric space which is an injective object in the category of metric spaces with 1-Lipshitz maps. More concretely, a metric space $X$ is said to be injective if the following holds: given a metric space $Y$ and a 1-Lipshitz map $f:Y \rightarrow X$, the map $f$ extends to every metric space $Z$ where $Y$ isometrically embeds. See Figure \ref{fig:injective_def}.

\begin{figure}[ht]
    \centering
    \caption{The object $X$ is injective.}
    \label{fig:injective_def}

\begin{tikzpicture}[commutative diagrams/every diagram]
\matrix[matrix of math nodes, name=m, commutative diagrams/every cell] {
Y & Z \\
X &  \\};
\path[commutative diagrams/.cd, every arrow, every label]
(m-1-1) edge[commutative diagrams/hook] node{$i$} (m-1-2)
(m-1-1) edge node[swap] {$f$} (m-2-1)
(m-1-2) edge[commutative diagrams/dashed]node{$\tilde{f}$} (m-2-1);
\end{tikzpicture}

\end{figure}
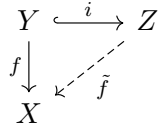

Based on the definition, injective objects of a certain category are objects where it's ``easy to map things into". In particular, for every subspace $Y$ of an injective space $X$, if $Y$ isometrically embeds into another metric space $Z,$ then $Z$ can also be mapped into $X$ via a 1-Lpshitz map. Such a universal property for $X$ allows us to perform certain constructions in an abstract metric space $Z$ containing a certain subspace $Y$ of $X$ and carry them over to our space; this is the main theme in injective spaces.

%Based on the definition, an injective metric space $X$ can be thought of as a ``large and highly packed", namely, a special application of the definition states that if you can map a small object into $X$, and your small object embeds in a bigger object, then the bigger object also maps into $X$). 

Said differently, injective metric spaces have plenty of space on the inside making their geometry very tolerating. We demonstrate this by some examples:
\vspace{2mm}

\noindent 1) \textbf{Embedding intervals}: Given two points $x,y$ in an injective metric space $(X, \dist)$, if our claim that $X$ has an abundance of space on the inside is true, there better exist a geodesic connecting $x,y$. Well, there is, and the proof follows immediately from the definition. Let $Y$ be the two points metric space $Y=\{x,y\}$ with $\dist_Y(x,y)=\dist(x,y)$. We know that $Y$ isometrically embeds in both $X$ and the interval $Z=[0, \dist(x,y)]$. Using the definition of injectivity, the isometric embedding $f:\{x,y\} \rightarrow X$ extends to $Z=[0, \dist(x,y)]$ providing a map $\Tilde{f}:[0, \dist(x,y)] \rightarrow X$ that restricts to $f$ on the end points. Since the map $\tilde{f}$ is 1-Lipshitz and restricts to $f$ on the end points $\{0, \dist(x,y)\}$, the triangle inequality implies that $\tilde{f}$ is an isometric embedding.

\vspace{2mm}

    \begin{figure}[ht]
   \includegraphics[width=14cm, trim = 1cm 9cm 1cm 9cm]{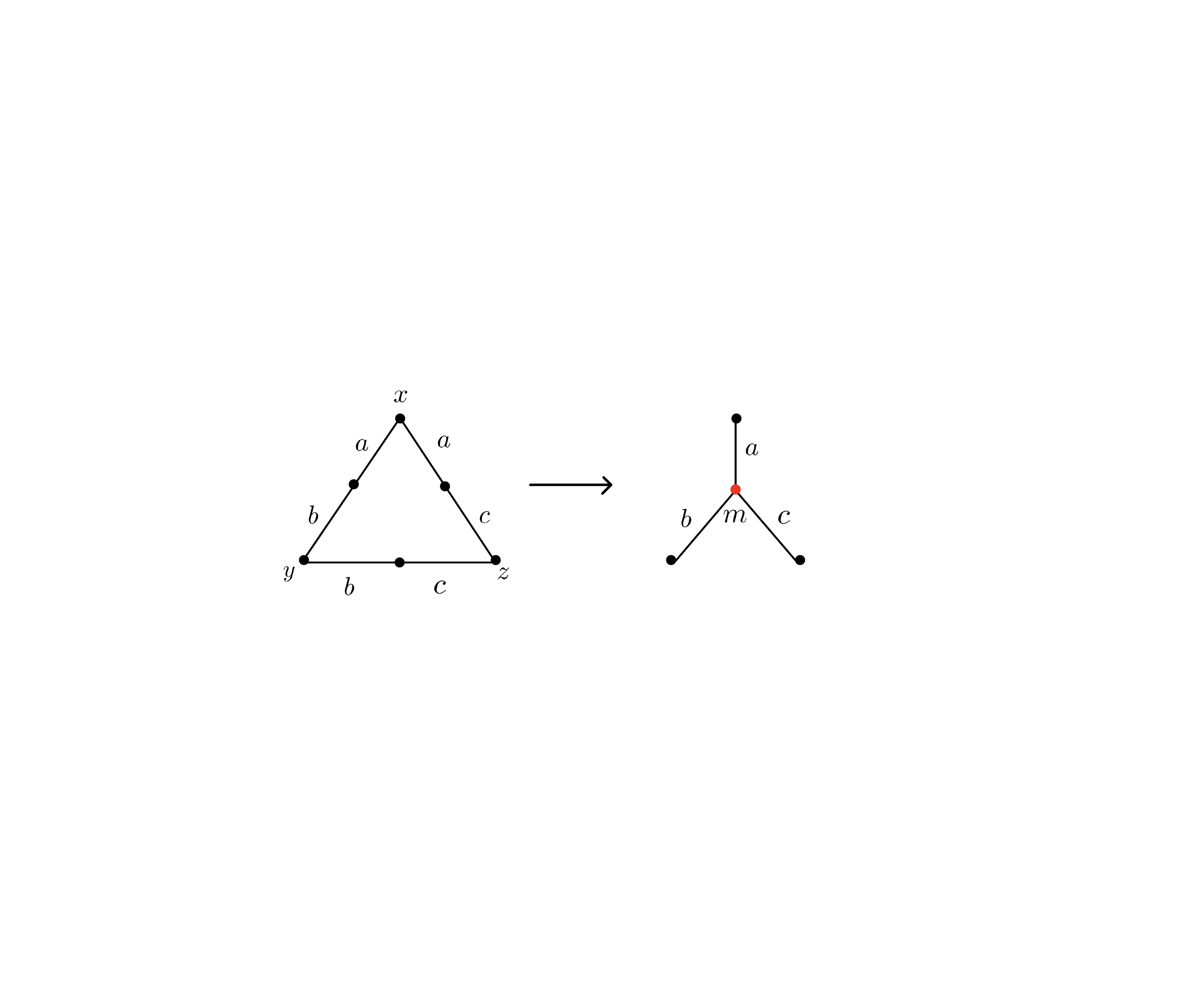}\centering
\caption{Abstract tripod associated to a geodesic triangle in an arbitrary metric space.} \label{fig:Tripod}
\end{figure}
\noindent 2) \textbf{Tripods:} Given a metric space $X$, for simplicity, let's say $X$ is geodesic (but this is not strictly required). For any three points $x,y,z \in X$, one can always find non-negative real numbers $a,b,c$ such that: 

\begin{align*}
\dist(x,y) & = a+b\\
 \dist (x,z)&=a+c \\
 \dist(y,z)&=b+c. 
\end{align*}

Namely, in any geodesic metric space $X$ and three points $x,y,z \in X$, there is a natural map from a given geodesic triangle $\Delta$ connecting $x,y,z$ to a tripod $T$ with vertices $x,y,z,m$ such that $\dist_T(x,m)=a$, $\dist_T(y,m)=b$ and $\dist_T(z,m)=c$ with $a,b,c$ as above, see Figure \ref{fig:Tripod}. However, the tripod $T$ doesn't usually isometrically embed in $X$, it's merely an abstract tripod that is mapped onto by the geodesic triangle $\Delta$. Serving as a further evidence of the tolerating nature of injective metric spaces, if $X$ is injective, then the tripod indeed isometrically embeds in $X$ (for instance, see \cite{LANG2013}). To see this, simply apply the same idea from the previous example, namely, consider the 3-point metric space $Y=\{x,y,z\}$ making the injection $f:Y \rightarrow X$ an isometric embedding, and notice that $Y$ isometrically embeds in a tripod $T$ as above. Injectivity of $X$ assures that the map $f$ extends to 1-Lipshitz map $\tilde{f}:T \rightarrow X$ such that the restriction on $\{x,y,z\}$ is an isometric embedding concluding the proof. This implies that $\tilde{f}$ is also an isometric embedding. Observe that the existence of such a tripod can be thought of as a ``higher dimensional" analogue of the property of being a geodesic metric space. Namely, a metric space $X$ is \emph{geodesic} if for any pair of points $x,y,$ there is an isometrically embedded $I$ with end points $x,y$ in $X.$ On the other hand, a metric space is, let's say \emph{tripodal}, if for any three points $x,y,z$ there is an isometrically embedded tripod $T$ with end points $x,y,z$ in $X.$

Given three points $x,y,z$ in an injective metric space, we have shown that there is an isometrically embedded tripod $T$ with a median-like point $m$ at the center of such a tripod. It is worth noting that such a tripod (and its median point $m$) is not unique, even for the same $x,y,z.$ Namely, it's possible to find two isometrically embedded tripods $T_1,T_2$ with end points $\{x,y,z\}$ whose median points $m_1 \neq m_2$ are different. This is not surprising given the fact that injective metric spaces have plenty of space on the inside making their geometry very tolerant (maybe too tolerant here?).

That being said, as Lang points out in his article  \cite{LANG2013}, the existence of such an isometrically embedded tripod is usually a good first test to whether the metric space in question is injective.

\vspace{2mm}

\noindent 3) \textbf{Centers:} Given a finite collection of points  $\{x_1,\cdots x_n\}$ in a metric space $X$, it's often desirable to associate a ``center" or an ``average of mass" to these points. For instance, if $G$ acts geometrically on a metric space $X$ where centers exist, one can usually use such centers to prove that each finite subgroup $F$ has to fix a point (the center of an $F$-orbit) which in turns implies that there are finitely many conjugacy classes of finite subgroups in $G.$

Recall that for a given metric space $X,$ the space $\mathbb{R}^X$ is simply the vector space consisting of all maps $f:X \rightarrow \mathbb{R}.$ As Petyt \cite{petyt:onlarge} describes in his thesis (building on ideas from \cite{descombes:asymptotic} and \cite{navas:L1}), the vector space structure on $\mathbb{R}^X$ makes it easy to assign centers for a finite collection of points $\{f_1,\cdots f_n\} \subset \mathbb{R}^X$, simply, by adding up points and dividing the sum by their number. That is to say, the \emph{center} of $\{f_1,\cdots f_n\}$ is defined by $$c=c(f_1,\cdots f_n)= \frac{1}{n}(f_1+f_2 \cdots f_n).$$

    %\begin{figure}[ht]
   %\includegraphics[width=14cm, trim = .001cm 6cm 4cm 4cm]{Center.pdf}\centering
%\caption{Pushing centers down} \label{fig:center}
%\end{figure}

    \begin{figure}[ht]
    \includegraphics[width=13cm, trim = 3cm 5cm 0cm 5.7cm]{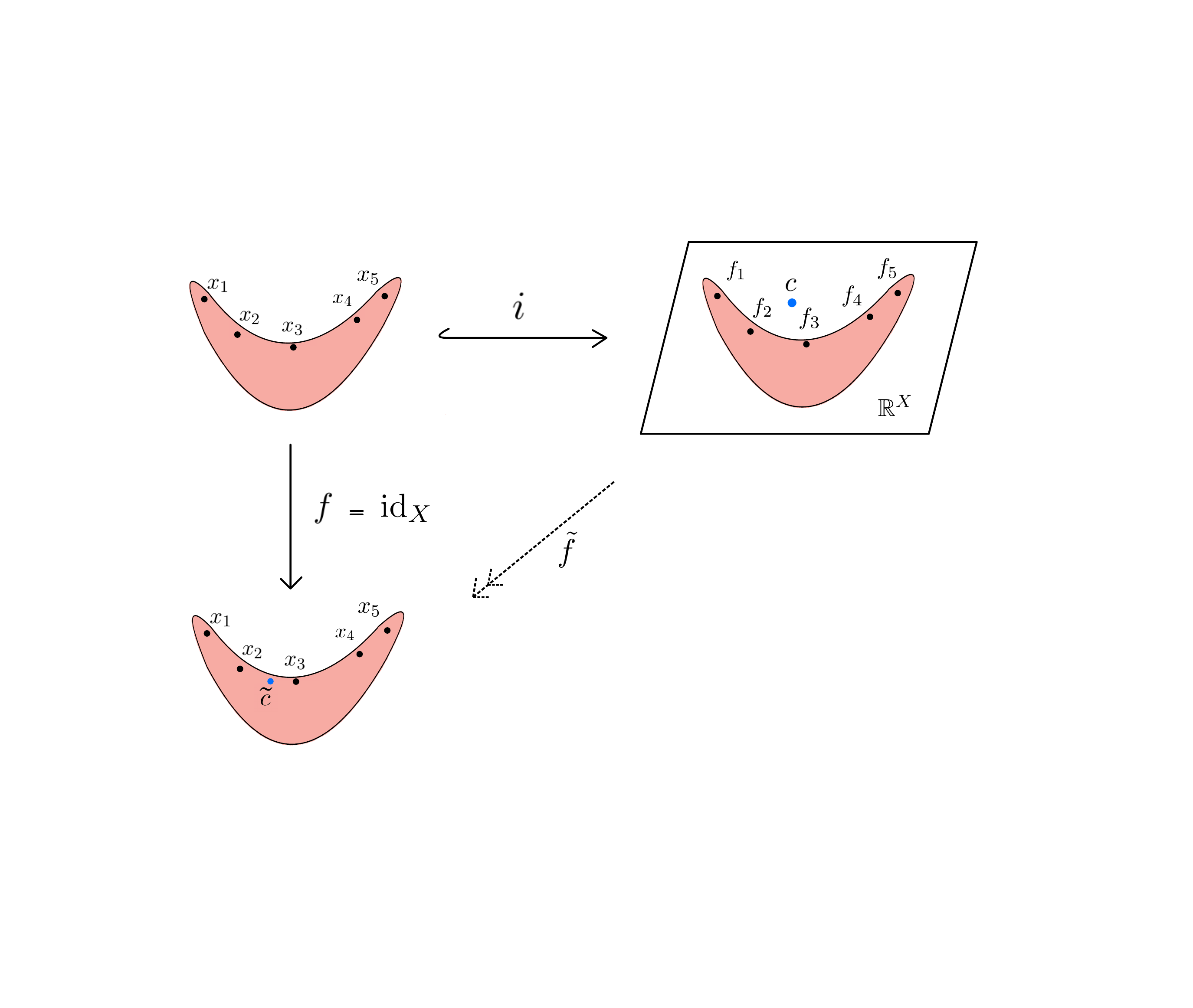}\centering
\caption{For a finite collection of points in $X$, we constructs the center of their canonical image $i(X)$ in $\mathbb{R}^X$ and the universal property lets us push it down to $X.$ } \label{fig:center}
\end{figure}

Given a finite collection of points $\{x_1,\cdots x_n\}$ in an injective metric space $X,$ we will use the fact that centers exist in $\mathbb{R}^X$ to associate a center to $\{x_1, \cdots x_n\},$ namely, by isometrically embedding $\{x_1,\cdots x_n\}$ in $\mathbb{R}^X$, finding a their center in $\mathbb{R}^X$ and pushing it down to $X$ via the universal property.

More precisely, consider the isometric embedding injection $i:X \rightarrow \mathbb{R}^X$ given by $x \mapsto \dist(x,-)$ where $\mathbb{R}^X$ is given with respect to the distance $\dist_\infty(f,g)=\text{sup}\{|f(x)-g(x)| \text{ where } x \in X\}.$ The images of the points $\{x_1,\cdots x_n\}$ to which we would like to assign a center are given by $\{f_1,\cdots f_n\}$ where $f_i:X \rightarrow \mathbb{R}$ is given by $f(z)=\dist(x_i,z)$. As we just discussed, the points $\{f_1,\cdots f_n\}$ admit a center $c \in \mathbb{R}^X.$ It is worth noting that this center $c$ doesn't generally live in $i(X) \subset \mathbb{R}^X,$ see Figure \ref{fig:center}, however, injectivity of $X$ assures that the identity map $f=\id_X:X \rightarrow X$ extends to a map $\tilde{f}:\mathbb{R}^X \rightarrow X,$ in particular, the point $c \in \mathbb{R}^X$ has a natural image in $X.$ Further, the fact that the map $\tilde{f}$ is 1-Lipshitz assures that $\tilde{f}(c)$ indeed behaves like a center for $\{x_1,\cdots x_n\},$ that is, the point $\tilde{f}(c)$ is at least as close to each $x_i$ as $c$ is to $f_i.$

\begin{remark}
    There is an important distinction to be made between the examples given in items 1 and 2 above versus the one given in item 3. In examples 1 and 2, we started with a finite collection of points $Y$ which we wanted to satisfy a certain property, and in order to establish such a property we isometrically embedded $Y$ into into a metric space $Z$ where the property holds. Injectivity of $X$ then kicked in allowing us to push such a property down to $X$ producing the desired statement. Importantly, the existence of such a property for $Z$ relied merely on the metric structure for $Z$ and nothing else.

    In example 3, we have also started with a finite collection of points $Y \subset X$ which we wanted to satisfy a certain property; the existence of a center, and in order to do so, we isometrically embedded $X$ into a metric space where centers exist, namely, the \emph{vector space} $\mathbb{R}^X$. Injectivity of $X$ then allowed us to push this center down to $X.$ Unlike items 1 and 2, the fact that centers exist in $\mathbb{R}^X$ didn't only use the metric structure on $\mathbb{R}^X,$ but it also relied on $\mathbb{R}^X$ being a vector space (addition was used to establish centers in $\mathbb{R}^X$). This shows that constructions that utilize vector space structures can often be exported to injective metric spaces; despite the fact that the latter aren't necessarily vector spaces.

\end{remark}

\begin{figure}[ht]
  \includegraphics[width=13cm, trim = 3cm 5cm 1cm 5.7cm]{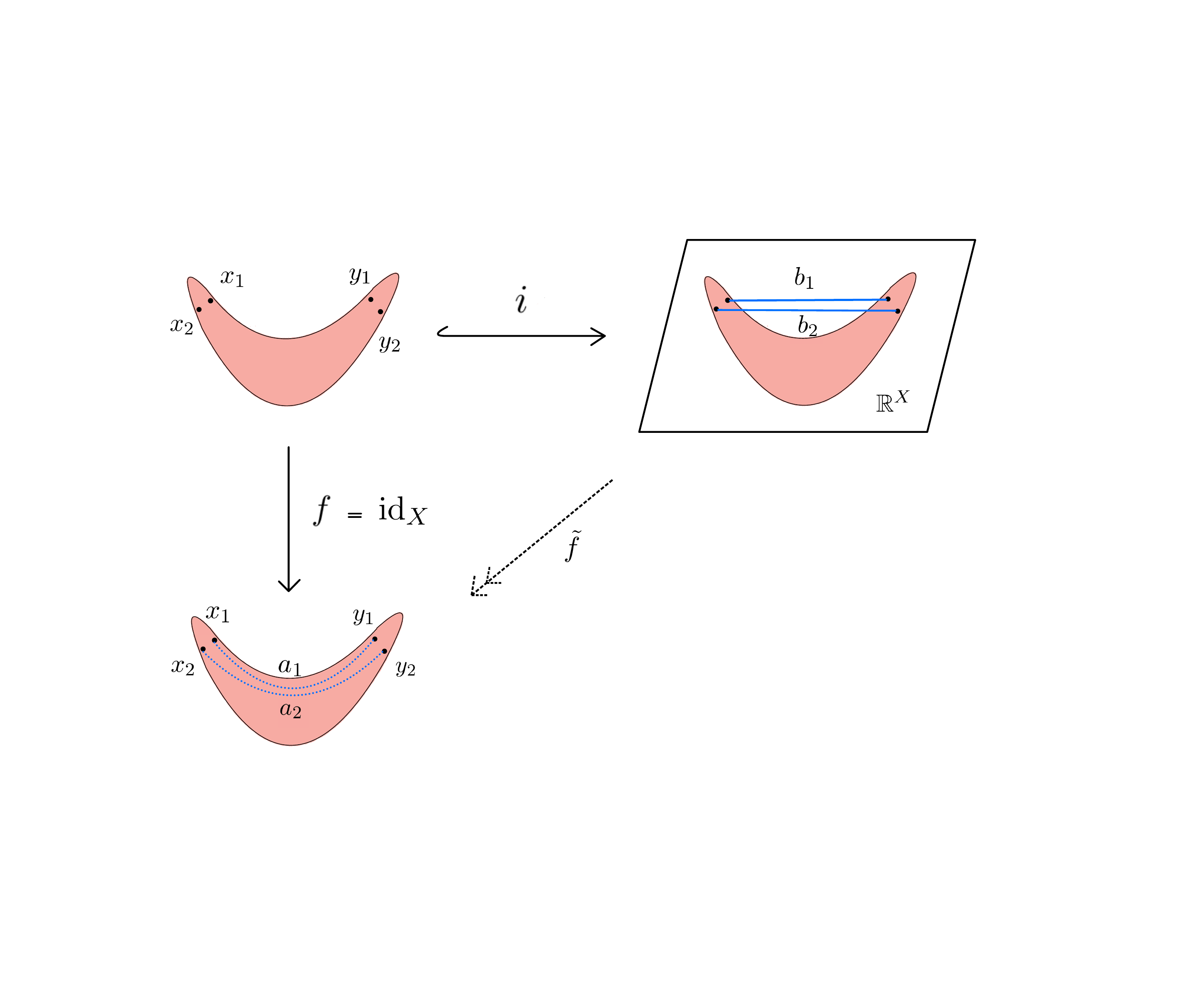}\centering
\caption{For two pairs of points $x_1,x_2,y_1,y_2 \in X$, we connect their canonical image in $i(X)$ by the affine geodesics $b_1,b_2 \subset \mathbb{R}^X$ and map these geodesics down to $X$; relying on the universal property for injectivity.} \label{fig:Combing}
\end{figure}

\noindent 4) \textbf{Fellow-travelling geodesics}: We will show that each injective metric space $X$ admits a collection of geodesics $\calG$ connecting every pair of points which fellow travel each other. As has been the philosophy of our constructions so far, all we need to do is find a metric space $Z$ where our desired property is present, include $X$ (or a finite set of points in $X$) in $Z$, and use injectivity of $X$ to push such a property down to $X$.

We will do just that. Namely, given two pairs of points $\{x_1,y_1\}$, $\{x_2,y_2\}$ where $\dist(x_1,x_2)$ and $\dist (y_1,y_2)$ are small, say they are both equal to  1, we wish to build a pair of geodesics $\{a_j\}_{j=1}^2$ connecting $\{x_j,y_j\}$ which remain within 1 of each other. Again, we isometrically embed $X$ into $\mathbb{R}^X$ via $i:X \rightarrow \mathbb{R}^X$ and consider the affine geodesics connecting $i(x_j),i(y_j)$ given by $b_j=(1-t)x_j+ty_j.$ Since we are in $\mathbb{R}^X,$ these affine geodesics $b_1,b_2$ stay within 1 of each other as they begin and end at points $i(x_j), i(y_j)$ with $\dist(x_1,x_2)=\dist(i(x_1),i(x_2))=1,$ and $\dist(y_1,y_2)=\dist(i(y_1),i(y_2))=1.$  By definition of injectivity, the identity map $f:X \rightarrow X$ extends to $\tilde{f}: \mathbb{R}^X \rightarrow X$ and since $\tilde{f}$ is 1-Lipshitz, the paths $a_j=\tilde{f}(b_j)$ stay within 1 of each other. Finally, the triangle inequality assures that $\{a_j\}_{j=1}^2$ are also geodesics, see Figure \ref{fig:Combing}.

In the study of geometric group theory, one of the most desirable properties of a metric space $X$ is the presence of a \emph{geodesic bi-combing}: an $\isom(X)$-invariant collection of geodesics $\calG$ connecting every pair of points $x,y \in X$ satisfying some fellow-travelling condition similar to what we obtained above. While the previous description provides a collection of fellow-travelling geodesics connecting every pair of points $x,y \in X$, there is no guarantee that such a collection is $\isom(X)$-invariant (thanks to Thomas Haettel for pointing this out to me). That being said, the proof of the existence of an invariant geodesic bi-combing follows very similar lines, but the ``push-down" map must be constructed explicitly to assure $\isom(X)$-invariance, see \cite{LANG2013}. The same comment applies to the center construction in item 3 above, to assure $\isom(X)$-invariance of the center, one should use such a map.

\vspace{2mm}

\noindent 4) \textbf{Helly}: The existence of the tripod in item 3 above is a reflection of the more general \emph{Helly property} for balls which characterizes injective metric spaces. Namely, a metric space is injective if and only if it is geodesic and every collection of pair-wise intersecting closed balls $\{B(x_i,r_i)\}_{i \in I
}$ admit a total intersection. The fact that the Helly property holds for injective spaces rhymes well with the mantra that injective spaces have exceptionally accommodating geometries. Namely, to have a collection of balls pair-wise intersect can be thought of as the minimal pre-requisite for them to have a total intersection. So the Helly property here says, once a collection of metric balls meets the minimal pre-requisite of having a total intersection, the accommodating nature of injective metric spaces (formally seen by its universal property) assures that they do. The proof of the (forward direction of the) statement also follows the mantra of ``find a metric space with the desired property, and push it down to $X$", we will do just that, see Figure \ref{fig:center_and_combing}. Let $\{B(x_i,r_i)\}_{i \in I
}$ be a collection of pair-wise intersecting balls, that is, $\dist(x_i, x_j) \leq r_i+r_j$ for all $i,j \in I$ and let $Y=\{x_i\}_{i \in I}.$ Consider the (possibly infinite) weighted graph $Z$ with vertex set $\{x_i\}_{i \in I} \cup \{\star\}$, each two vertices $x_i,x_j \in \{x_i\}_{i \in I}$ are connected by an edge with weight $\dist(x_i,x_j)$, and $\star$ is connected to each $x_i$ by an edge weighted $r_i.$ It is immediate by the triangle inequality that we have an isometric embedding $i:Y \rightarrow Z$. Since $Y$ also isometrically embeds in $X$ via the inclusion map $f:Y \rightarrow X$,  the universal property of injectivity produces a map $\tilde{f}: Z \rightarrow X$. As $\Tilde{f}$ is 1-Lipshitz, the point $\tilde{f}(\star)$ satisfies $\dist(\tilde{f}(\star), \tilde{f}(x_i)) \leq \dist(\star, x_i)=r_i$ for each $i.$ This concludes the proof.

\begin{figure}[ht]
    \centering
    \begin{minipage}{0.45\textwidth}
        \centering
        \includegraphics[width=7cm, trim = .01cm 5cm 5cm 5.7cm]{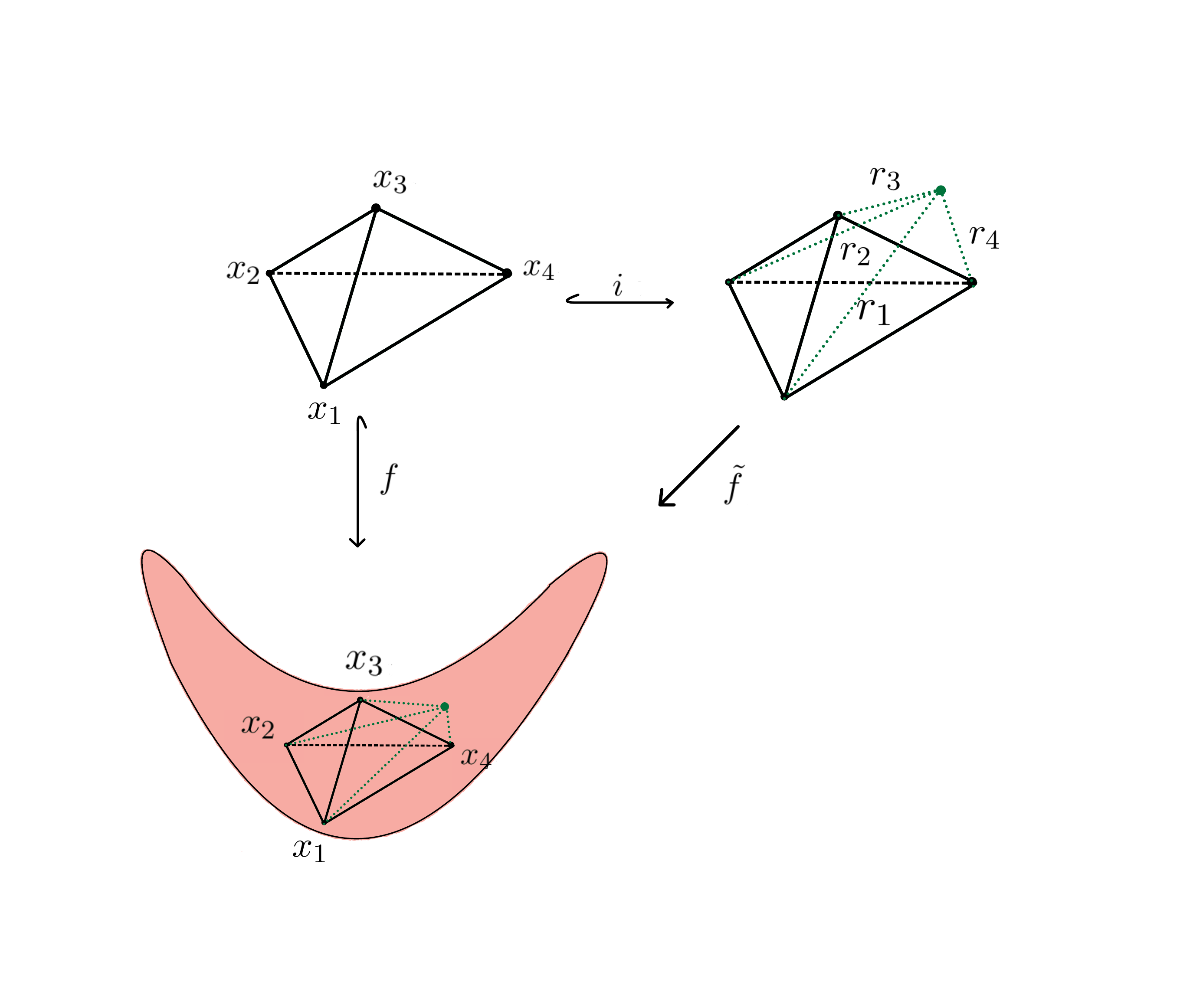} % first figure itself
        \caption{The picture on the right shows the failure of the Helly property in $\mathbb{R}^2$ with respect to the usual Euclidean metric. The picture on the left shows how injectivity implies the Helly property.}\label{fig:center_and_combing}
    \end{minipage}\hfill
    \begin{minipage}{0.45\textwidth}
        \centering

        \includegraphics[width=8.5cm, trim = .01cm 5cm 5cm 5.7cm]{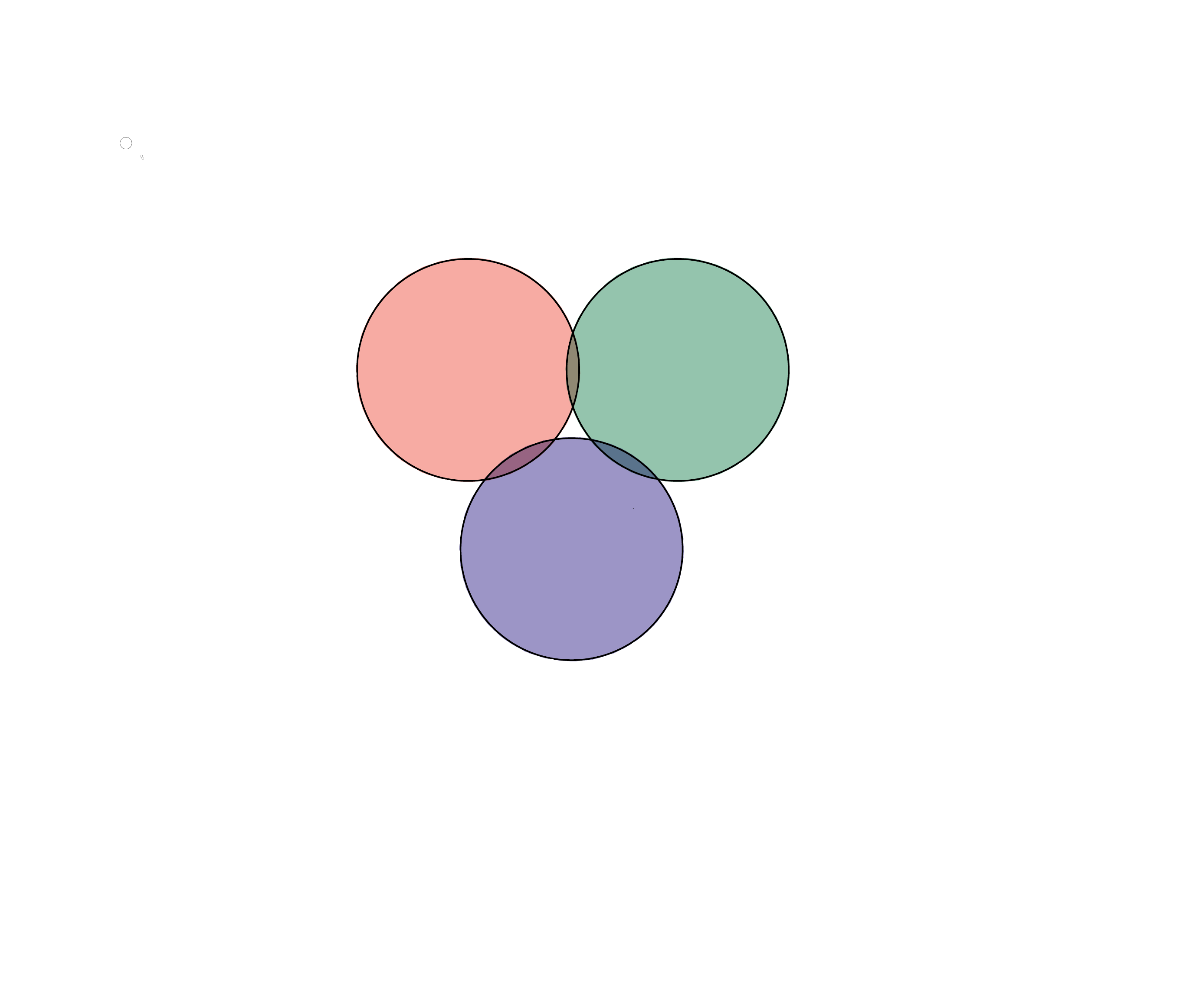} % second figure itself
    \end{minipage}
\end{figure}

\noindent 5) \textbf{Injective spaces are complete}: This is perhaps a good exercise at this point as its proof is very similar to the ones given in items 1-4 above.

%construct a finite graph whose vertices $Y=\{x_1,\cdots x_n\}$ with a weighted edge connecting $x_i,x_j$ whose weight is $\dist_X(x_i,x_j)$ and consider the resulting metric space, by abusing notation, we shall use $Y$ to denote the graph.

%As  Jenny Beck said in a recent talk of hers, an injective metric space can be thought of as a space that has a lot of space for things to happen.

%\textcolor{brown}{Give two montras: You would like to have a median, and median is a 1-dimensional statement of having no halls. Well, you do have medians and they are usually the first check of being an injective metric space as Lang said one about the vector space $\mathbb{R^n}$ not using the dimension. If something can be done in a CAT(0) cube complex, or less generally $\mathbb{R^n},$ that doesn't rely on the finite dimensionality, then it can be done. Simply as we forget the dimension, and also, compare with too many curtains separating two points. Add diagram for the product}.

\subsection{The canonical nature of injective spaces:}\label{subsec:the_cannonical_nature} In the introduction, we have discussed examples of injective metric spaces as well as the recent impact of injective metric spaces on the field of geometry group theory. In this section, we mention a few reasons that make injective spaces very canonical objects, both in general metric geometry and particularly in geometric group theory. Their canonical nature should not be surprising, however, given the fact that they are defined via a universal property. We start by noting that every metric space admits an isometric embedding in a ``smallest" injective metric space.

\subsubsection{Injective envelopes}

Suppose that $X$ is a metric space which is not injective, this means that $X$ doesn't have enough points to make its geometry accommodating. A natural thing to try and do is add more points to $X$, as many points as needed to make $Y=X \cup \{\text{extra points}\}$ accommodating enough, i.e., to make it injective. On the other hand, if we desire the resulting bigger space $Y=X \cup \{\text{extra points}\}$ to be reflective of the original geometry of $X$, it's important that we don't add more points than necessary, i.e., we want $Y=X \cup \{\text{extra points}\}$ to be injective $Y-X$ but also as small as possible. The tension between these is what defines the \emph{injective hull} of a metric space $X$.

Given a metric space $X$, consider the set $\Delta:=\{f \in \mathbb{R}^X | f(x)+f(y) \geq \dist(x,y) \,\, \forall x,y \in X\}.$ An element $f \in \Delta$ is said to be of a \emph{metric form}. Define the set $\Delta^1:=\Delta \cap \text{Lip}^{1}(X, \mathbb{R})$ where $\text{Lip}^{1}(X, \mathbb{R})$ denotes the space of all $1$-Lipshitz maps $f:X \rightarrow \mathbb{R}$.  We equip the set $\Delta^1$ with the distance $\dist_\infty(f,g):=\text{sup}\{|f(x)-g(x)| \text{ where } \,\, x \in X\}$.

The space $\Delta$ has a natural poset structure where $f \leq g$ if and only if $f(x) \leq g(x)$ for all $x \in X.$ The \emph{injective hull} (sometimes also called the \emph{injective envelope} or the \emph{tight span}) of a metric space $X$, denoted $E(X)$, is defined to be the collection of all minimal elements in the aforementioned poset on $\Delta$:
$$E(X):=\{f \in \Delta| \,\,\text{if}\,\, g \in \Delta, \text{we have}\,\, g \leq f \implies g=f\}.$$
 
As a metric space, the injective hull is the set $E(X)$ equipped with the $\dist_\infty$-distance. The injective hull was originally introduced by Isbell \cite{Isbell1964SixTA} and was independently rediscovered twice since, first by Dress in \cite{Dress1984TreesTE} and second by Chrobak and Larmore in \cite{Chrobak1994}. We now summarize some properties of $E(X),$ for more details, see Section 3 of \cite{LANG2013}.

\begin{lemma}[{\cite[Section 3]{LANG2013}}]\label{lem:summary} For any metric space $X$, the injective hull $E(X)$ is a geodesic metric space with respect to $\dist_\infty$-distance. Furthermore, we have the following.

\begin{enumerate}
    \item There is an isometric embedding $e:(X,\dist) \hookrightarrow (E(X),\dist_\infty)$ given by $x \mapsto \dist(x,-).$
    \item For any $f \in \Delta^1,$ we have $f(x)=\dist_\infty(e(x),f)$ for all $x \in X.$
    
    \item  For $f \in \mathbb{R}^X$, we have $f \in E(X)$ if and only if  $f \in \Delta$ and for any $\epsilon >0$ and any $x \in X,$ there is some $y \in X$ with $f(x)+f(y) \leq \dist(x,y)+\epsilon.$ 
    
    \item $E(X) \subseteq \Delta^1.$
In fact, the embedding $e:X \hookrightarrow E(X) $ is $\isom(X)$-equivariant.
    \end{enumerate}

\end{lemma}

One perspective on the injective hull of a metric space is that it's a convexifying process in the sense that we are adding points to $X$ that had a ``canonical reason" of being in $X.$ Namely, all of the newly added points $f:X\rightarrow \mathbb{R}$ were of a metric form (i.e., they satisfy $f(x)+f(y) \geq \dist(x,y))$ which essentially makes them indistinguishable from the points $\dist(x,-) \in e(X) \cong X$ as both satisfy the same triangle-like inequality.

%Namely, the injective hull $E(X)$ is built by first thinking of points in $X$ as distance functions $\dist(x,-) \in \mathbb{R}^X,$ and then we add the (minimal) points in $\mathbb{R}^X$ which are ``very similar" to $\dist(x,-).$ More precisely, all the newly added points are ``distance-like" functions in the sense that they all satisfy the triangle-like inequality $f(x)+f(y) \geq \dist(x,y)$. That is to say, all the newly added points had a cannonical reason of being part of the space as they all satsfied the same triangle inequality enjoyed by the distance function.

%Another perspective on the injective hull of a metric space is that it's a ``convexifying process" in the sense that we are adding points to $X$ that had a ``cannonical reason" of being in $X.$ Namely, the injective hull $E(X)$ is built by first thinking of points in $X$ as distance functions $\dist(x,-) \in \mathbb{R}^X,$ and then we add the (minimal) points in $\mathbb{R}^X$ which are ``very similar" to $\dist(x,-).$ More precisely, all the newly added points are ``distance-like" functions in the sense that they all satisfy the triangle-like inequality $f(x)+f(y) \geq \dist(x,y)$. That is to say, all the newly added points had a cannonical reason of being part of the space as they all satsfied the same triangle inequality enjoyed by the distance function.

\subsubsection{Natural metric on products} Given two objects $X_1,X_2$ in a category $\calC$, the \emph{product} of $X_1$ and $X_2$ is an object $X$, usually denoted as $X_1 \times X_2$, equipped with a pair of morphisms $\{\pi_i: X \rightarrow X_i\}_{i=1}^2$ such that for any object $Y$ and morphisms $\{f_i: Y \rightarrow X_i\}_{i=1}^2$, there is a unique morphism $f:Y \rightarrow X_1 \times X_2$ so that the diagram in Figure \ref{fig:products} commute. For instance, in the category of sets, topological spaces and groups, the universal notion of a product above gives the usual product structures on such objects, in particular, for topological spaces it gives the usual product topology on $X_1 \times X_2$ and for groups, the product is exactly the direct product of the two groups. Interestingly, if we work with the category of metric spaces and 1-Lipshitz maps, then the product of two metric spaces $X_1 \times X_2$ resulted by the above universal property is the metric space $(X_1 \times X_2, \dist_\infty)$, that is, $\dist((x_1,x_2), (x_1',x_2'))= \text{max}\{|x_1-x_1'|, |x_2-x_2'|\}$. In particular, the natural metric to consider on the Euclidean space $\mathbb{R}^n$ is the $\dist_\infty$-distance and not the Euclidean or the combinatorial (taxi-cap) one. Said differently, the ``correct" metric structure on $\mathbb{R}^n$ (when morphisms are chosen to be 1-Lipshitz maps) is precisely the one making it an injective metric space.

\begin{figure}[ht]
    \caption{Products in any Category.}
    \label{fig:products}

\begin{tikzpicture}[commutative diagrams/every diagram]
\matrix[matrix of math nodes, name=m, commutative diagrams/every cell] {
 & Y &  \\
X_1 & X_1 \times X_2 & X_2 \\};
\path[commutative diagrams/.cd, every arrow, every label]
(m-1-2) edge [commutative diagrams/dashed] node {$f$} (m-2-2)
(m-1-2) edge node[swap] {$f_1$} (m-2-1)
(m-2-2) edge node {$\pi_1$} (m-2-1)
(m-2-2) edge node[swap] {$\pi_2$} (m-2-3)

(m-1-2) edge node {$f_2$} (m-2-3);

\end{tikzpicture}
\end{figure}

So far, we have seen two reasons why injective metric spaces are very canonical objects. The first of which is the fact that every metric space naturally comes with its unique injective envelope, or injective hull. Secondly, the canonical metric to impose on a product of two metric spaces is the $\dist_\infty$-metric which is intimately connected with injectivity. Now, we shall discuss reasons that make injective metric spaces appealing from a geometric group theory point of view.

\subsubsection{Coarsely injective metric spaces} A metric space $X$ is said to be \emph{coarsely injective} if there exists a constant $C$ such that every point of $E(X)$ is within $C$ of a point from $e(X),$ where $e$ is the map $e:X \rightarrow E(X)$. Equivalently, a geodesic metric space is coarsely injective if it satisfies the \emph{coarse Helly property} for its balls: there is a constant $C$ such that for every collection of pair-wise intersecting balls $\{B(x_i,r_i)\}_{i \in I}$, there is a point $x \in \cap_{i \in I} B(x_i, r_i+C)$; for equivalent definitions of coarsely injective metric spaces, see Proposition 5.2 in \cite{Urs2022}. The primary example to be kept in mind is that of a Gromov hyperbolic space \cite{LANG2013}. Further, in \cite{HHP}, the authors show that every HHS $(X, \dist)$ admits an injective distance $\rho_\infty$ that is equivariently quasi-isometric to the original $\dist$ on $X.$ Coarsely injective metric spaces can be thought of as spaces that are already very close to being injective so that injectifying them via the map $e:X \rightarrow E(X)$ doesn't amount to adding too many points: every newly added point is within a uniform constant of a point from $X.$ Finally, thinking of the injective hull of $X$ as a convexifying process as we described above, it is not much of a surprise that hyperbolic spaces are coarsely dense in their injective hull, given the well-known strong convexity properties they enjoy.

\subsubsection{Strongly-contracting geodesic} There are various competing way of quantifying what it means for a space $X$ to contain a ``hyperbolic-like" geodesic. These are usually geodesics that enjoy certain convexity properties similar to geodesics in hyperbolic spaces. The strongest of these ``hyperbolic-like" notions is that of \emph{$D$-strongly-contracting} geodesic. A geodesic $\gamma$ is said to be $D$-strongly-contracting if every ball $B$ with $B \cap \gamma=\emptyset$ satisfies that 
$\diam(\pi_\gamma(B)) \leq D,$ where $\pi_\gamma:X \rightarrow \gamma$ is the nearest point projection.

Strongly-contracting geodesics are characterized by having a strong ``bottleneck-like" behaviour. Namely, it is not hard to check that a geodesic $\gamma$ is $D$-strongly-contracting if and only if for any two points $a,b \in X$ that project far along $\gamma$ (more than $6D$), each geodesic connecting $a,b$ must travel $4D$-close to every point $p$ lying between $\pi_\gamma(a), \pi_\gamma(b)$ along $\gamma.$

A geodesic metric space is $\delta$-hyperbolic if and only if there exists a constant $D \geq 0$ such that every geodesic is $D$-strongly contracting; in this sense, these geodesics can be thought of as ``hyperbolic-like"-geodesics. The presence of (infinite) strongly contracting geodesics in a group $G$ or a space $X$ on which $G$ acts has numerous desirable consequences regarding its growth, genericity of certain elements, acylindrical hyperbolicity and many more.

On the other hand, the weakest of these ``hyperbolic-like" notions is that of an \emph{$M$-Morse geodesic}. For a function $M: \mathbb{R}^+ \times \mathbb{R}^+ \rightarrow \mathbb{R}^+ $, a geodesic $\gamma$ is said to be $M$-Morse if every $(q,Q)$-quasi-geodesic $\beta$ with end points on $\gamma$ remains the $M(q,Q)$-neighborhood of $\gamma.$ Such $M$-Morse geodesics also characterize hyperbolicity in the sense that a geodesic metric space is hyperbolic if and only if there is a map $M: \mathbb{R}^+ \times \mathbb{R}^+ \rightarrow \mathbb{R}^+ $ such that every geodesic is $M$-Morse.

For any geodesic metric space, strongly contracting geodesics are Morse, but the converse is far from true (see \cite{Fink}).

In this sense, the presence of Morse geodesics in a space $X$ can be thought of as a very minimal pre-requisite to the existence of strongly contracting geodesics. The exceptionally tolerating nature of injective spaces assures that such a minimal pre-requisite suffices:

\begin{theorem}[{\cite[Theorem A]{Sisto-Zalloum-22}}] A geodesic in an injective metric space is $M$-Morse if and only if it is $D$-strongly contracting where $M$ and $D$ determine each other.
    
\end{theorem}

If $\alpha$ is a strongly contracting geodesic in $X$, there is no reason to expect that it remains strongly contracting in a space $Y$ where $X$ isometrically embeds, even if $Y$ is injective. For instance, the line $X=\mathbb{R}$ is strongly contracting in $X$, however, it admits an isometric embedding into the injective metric space $(\mathbb{R}^2, \dist_\infty)$ where no geodesic line is strongly contracting. On the other hand, the injective hull $E(X)$ of a metric space $X$ was defined to be the \emph{minimal} injective metric space containing $X$, hence, one would hope that such a minimality assures that the geometries of $X$ and $E(X)$ are similar with respect to strongly contracting geodesics, and indeed, they are.

\begin{theorem}[{\cite[Theorem B]{Sisto-Zalloum-22}}] \label{thm:hull_iff}
A geodesic $\gamma$ in a metric space $X$ is $D$-strongly contracting if and only if its image in $E(X)$ is $D'$-strongly contracting, where $D,D'$ determine each other.
\end{theorem}

In principle, the above theorem provides a test to whether a geodesic line $l$ in an arbitrary metric space $X$ is strongly contracting, namely, by considering its image $i(l)$ in the injective hull $E(X)$ and checking whether its strongly contracting there. 

However, in practice, this requires understanding what $E(X)$ looks like in advance, and that's usually pretty hard (precisely because of how accommodating and big injective spaces are). The upshot of the above theorem is indeed the opposite; it sheds a light on the shape of the injective hull $E(X)$ of a geodesic metric space $X$ containing a strongly contracting geodesic $l \subset X$. That is, the fact that $l$ remains strongly contracting in $E(X)$ means that the strong bottleneck property met by $l$ in $X$ persists in its injective hull $E(X),$ namely, if $a,b \in E(X)$ project far along $e(l)$, then every geodesic connecting $a,b$ must pass close to the geodesic $e(l)$. The fact that such a bottleneck property persists means that not too many extra points where added to $X$ \emph{along the geodesic} $l$; which can be viewed as saying, the region around $l$ was already very close to being injective. This is not very surprising given the fact that strongly contracting geodesics characterize hyperbolicity (a space $X$ is hyperbolic if and only if there exists some $D \geq 0$ such that every geodesic is $D$-strongly contracting) and hyperbolic spaces are known to be dense in their injective hulls.

Another interpretation of the above theorem is the following: since strongly contracting geodesics $l \subset X$ already possess very strong convexity-like properties similar to these present in hyperbolic spaces, convexifying the space $X$ (embedding it in its injective hull $E(X))$ doesn't amount to adding too many points along such geodesics. In fact, as pointed out in \cite{Sisto-Zalloum-22}, the above Theorem \ref{thm:hull_iff} recovers Lang's theorem \cite{LANG2013} that hyperbolic spaces are coarsely dense in their injective hull.

%Namely, while its hull $E(X)$ can be mysterious, the theorem above says that $E(X)$ must admit certain regions with bottleneck-like properties, for instance, whatever the region the image of $l$ in $E(X)$ lives in looks-like, we know that it can't be something like an $\mathbb{R}^2.$

Prior to \cite{Sisto-Zalloum-22}, the primary two classes of examples where the notions of Morse and strongly contracting were known to agree are CAT(0) spaces with their CAT(0) distance \cite{charneysultan:contracting}, and CAT(0) cube complexes  with their combinatorial distance \cite{Genevois2020}.

The mere existence of Morse geodesics in a metric space $X$ implies that $X$ is ``hyperbolic-like"; at least in the regions where these geodesics are present. When $X$ is also CAT(0) or injective, the additional non-positive curvature coming from $X$ combine with the Morse property upgrading it to a strongly contracting one.

In particular, both the CAT(0) and injective notions are strong enough to upgrade Morse geodesics to strongly contracting ones; a natural question to ask is, which among the two notions is a stronger notion of non-positive curvature, the CAT(0) notion or the injective one?

\subsubsection{Are injective spaces more or less hyperbolic than CAT(0) spaces? Two contrasting perspectives} If we agree to place Gromov hyperbolic spaces as a reference to how ``negatively-curved" a geodesic metric space is, then one can make the argument that injective metric spaces are more non-positively curved than CAT(0) spaces. Let's demonstrate this with an example. Suppose that $X$ is a CAT(0) cube complex and $c=[0,1]^n \subset X$ is a cube of dimension $n\geq 3$. Such a cube is not $1$-hyperbolic since there are (combinatorial) geodesics $\alpha,\beta$ connecting its end points $(0,0 \cdots, 0), (1,1 \cdots, 1)$ that get quite far from each other, depending on $n$. However, the cube $c$ --and every other cube regardless of its dimension-- is $1$-hyperbolic with respect to the injective $\dist_\infty$-distance, since every pair of vertices in $c$ are at distance $1.$ Hence, the injective $\dist_{\infty}$-distance can be viewed as being more hyperbolic at a local level than the combinatorial distance is.

In fact, suppose that we start with a CAT(0) cube complex $X$ which is not $\delta$-hyperbolic, $\dim(X) \geq 100\delta$ and assume that we wish to cone off certain pieces of $X$ so that the resulting space $Y=X/\sim$ is $\delta$-hyperbolic (one can keep in mind a simple example, for instance, a tree of flats $\mathbb{R}^k$ with $k \geq 100\delta$). This conning off can be viewed as a two-step process, each of which brings $X$ closer to being hyperbolic. The first step is to address the local failure of hyperbolicity by conning off every $n$-cube for all $n \geq 1$, and this can be done by equipping $X$ with its injective $\dist_\infty$-metric. The second step then is to address to global failure of hyperbolicity and this can be done by conning off each flat or each product region. This two-step process description suggests that the $\dist_{\infty}$-distance can be viewed as being ``one step closer" to hyperbolicity compared to the usual combinatorial distance.

On the other hand, one advantage of CAT(0) spaces over injective metric spaces is that CAT(0) spaces come with their \emph{Euclidean} or CAT(0) distance $\dist_{\mathbb{E}}$ that makes $(X,\dist_{\mathbb{E}})$ into a CAT(0) space, in particular CAT(0) cube complexes admit such a distance. Unlike the injective distance on a CAT(0) cube complex $X$, the Euclidean distance produces unique geodesics, and triangles in $(X, \dist_{\mathbb{E}})$ are at most as fat as triangles in $\mathbb{R}^2$, while the ones in 
$(X, \dist_{\infty})$ can be very fat. Finally, hyperplanes in CAT(0) cube complexes are convex with respect to both the CAT(0) and combinatorial distances, but not with respect to the $\dist_\infty$-geodesics. The existence of fat triangles, non-uniqueness of geodesics, failure of convexity with respect to hyperplanes can be viewed as pushing us a step further from hyperbolicity, contrasting the claim in the previous paragraph. We shall leave it to the reader to choose a perspective on this. One interpretation of the aforementioned could be that the universal property of an injective metric space $X$ forces it to include more things than necessary leading to too many geodesics in $X$ which in turns leads to the existence of fat of triangles and lack of convex-type properties.

\begin{remark} After finishing the first draft of this survey, Anthony Genevois pointed out the following phenomenal that makes it clear how injectivity is more hyperbolic at a local level supporting the two-step process described above. First, a \emph{grid} of hyperplanes is the data of two chains $c=\{h_1,\cdots h_m\}, c'=\{k_1, \cdots k_n\}$ such that $h_i \cap h_j \neq \emptyset$ for all $1 \leq i \leq m $ and $1 \leq j \leq n.$ A grid is said to be \emph{$L$-thin} if $\text{min}\{m,n\} \leq L.$ We state two theorems characterizng hyperbolicity of a CAT(0) cube complex in terms of thin grids.

\begin{theorem}[{\cite[Theorem 3.1 and 3.6]{genevois:hyperbolicities}}]  Let $X$ be an arbitrary CAT(0) cube complex and let $\dist, \dist_\infty$ denote the combinatorial and injective metrics, respectively. We have the following:

\begin{enumerate}
    \item $(X,\dist)$ is hyperbolic if and only if its grids are uniformly thin and $X$ is finite dimensional.

 \item $(X,\dist_\infty)$ is hyperbolic if and only if its grids are uniformly thin.
 
 \end{enumerate}
    
\end{theorem}

\end{remark}

    \section{CAT(0) cube complexes and their hyperplanes}\label{sec:CCC and their hyperplanes}
        In part due to their fundamental role in the resolution of the virtual Haken's conjecture \cite{wise:structure,agol:virtual}, CAT(0) cube complexes have been a central object of study in geometric group theory. When thinking about a CAT(0) cube complex, one is meant to imagine a simply connected space obtained by gluing unit Euclidean cubes $[0,1]^i$ isometrically along faces, with no ``empty cubes", that is: if the faces of a cube are present, then so is the cube. The standard metric a CAT(0) cube complex is typically equipped with is the \emph{combinatorial} or \emph{taxi cap} metric denoted $\dist$: You equip each cube $[0,1]^i \subset X$ with the metric given by $$\dist((x_1,\cdots x_i),(y_1,\cdots y_i))= \underset{1 \leq j \leq i}{\sum}|x_j-y_j|,$$ and extend the metric to the cube complex $X$ in the obvious way. Given three vertices $x_1,x_2,x_3$ in a CAT(0) cube complex, one can assign a fourth point $m=m(x_1,x_2,x_3)$ called the \emph{median} of $x_1,x_2,x_3$ and it's defined to be the (unique) vertex satisfying $\dist(x_i,m)+\dist(m,x_j)=\dist(x_i,x_j)$ for all $i \neq j.$

    A \emph{midcube} is defined by restricting one of the coordinates in a cube $[0,1]^i$ to $\frac{1}{2}$. A \emph{hyperplane} is what you obtain from the following process: given a midcube $c$, if $c$ touches a midcube $c'$ in an adjacent cube, you consider $c \cup c'$ and you continue inductively, that is, if $c'$ touches a midcube $c''$ in another adjacent cube, we consider $c \cup c' \cup c'',$ etc. More formally, a hyperplane $h$ is a subspace of $X$ such that for any (closed) cube $c,$ the intersection $h\cap c$ is either a midcube or it's empty. See Figure \ref{fig:A hyperplane}. Each hyperplane $h$ defines two disjoint path-connected \emph{half spaces} $h^+,h^-$ with $X=h^+ \sqcup h^-$. If $x \in h^+, y \in h^-$, we say that $h$ \emph{separates} $x,y.$ Similarly, $h$ separates $A,B \subset X$ if $A \subseteq h^+, B \subseteq h^-.$ Finally a collection of hyperplanes $\{h_1,\cdots h_n\}$ is said to be a \emph{chain} if each $h_i$ separates $h_{i-1}, h_{i+1}.$

    \subsection{Fundamental geometric notions of a cube complex via hyperplanes}\label{subsec:fundemental_CCC}

    A groundbreaking work of Sageev \cite{sageev:ends,sageev:codimension} shows that the geometry of a CAT(0) cube complex is fully recorded in their hyperplanes and the way they interact with another. For example, although distances, geodesics, convex hulls, medians and projections are all defined purely via geometric terms, they admit the following descriptions which only involve hyperplanes:
    
    %The survey is organized as follows. In the first section, we discuss the main class of interest in this survey CCC and their hyperplanes. We demonstrate how the entire geometry of a CAT(0) cube complex is encoded via their hyperplanes and the way they interact with one another, to do so, we show how all the fundemental geometric components of a CCC have characterizations that only involve hyperplane including the distance, geodesics, medians, convex hulls, and projections to convex sets. Namely, the following holds true for any finite deimensional CAT(0) cube complex.

\vspace{2mm}

\begin{itemize}
    \item \textbf{Distance:} $\dist(x,y)=|\text{hyperplanes that separate }x,y|,$ for any vertices $x,y \in X$.

    \item \textbf{Geodesics:} An edge path is a geodesic if and only if it never crosses a hyperplane twice.

        \item \textbf{Medians:} The median of any $x,y,z \in X$ is the intersection of all half spaces containing the majority of $x,y,z$.

    \item \textbf{Convex hulls:} The convex hull of a set $A$ (which recall is the smallest convex set containing $A)$ is the intersection of all half spaces that properly contain $A$.

    \item \textbf{Projections/gates}: The vertex $P_C(x)$ minimizing the distance from a vertex $x$ to a convex set $C$ is characterized by being the unique vertex of $C$ such that every hyperplane $h$ separating $x,P_C(x)$ separates $x,C.$

    \end{itemize}

%Recall that a cubical isometry $g$ is said to \emph{skewer} a pair of hyperplanes $h,k$ if $h^+ \subset k^+$ and $gk^+ \subset h^+$ for some half spaces $h^+,k^+$ corresponding to $h,k.$ Hyperplanes also provide full descriptions of the nature of cubical isometries in CAT(0) cube complexes. First, regarding the order of an isometry and second, regarding ``how hyperbolic" a given cubical isometry is:

%\vspace{3mm}
%\begin{itemize}
 %   \item A cubical isometry $g$ has an infinite order if and only if $g$ skewers a pair of hyperplanes.
%\item A cubical isometry is Morse (or contracting) if and only if it skewers a pair of separated hyperplanes.

%\end{itemize}

%\vspace{3mm}
%Recall that a cubical isometry $g$ is said to be \emph{Morse} if $(q,q)$-quasi-geodesics begining and ending on an axis $l_g$ for $g$ remain in some $M=M(q)$-neighborhood of $l_g$. A collection of hyperplanes $c=\{h_1,\cdots h_n\}$ is said to form a \emph{chain} if each $h_i$ separates $h_{i-1}$ from $h_{i+1}$ and two hyperplanes $h,k$ are said to be \emph{$L$-separated} if every chain of hyperplanes intersecting them has caridnality at most $L$.

\subsection{What makes hyperplanes so appealing?} With the goal of understanding a certain object in math, a very typical approach is to decompose it into smaller simpler pieces, understand these pieces and assemble them back together to get an understanding of the underlying object. The idea is to decompose an object into \emph{simple building blocks} and by understanding these simple building blocks and how they interact with one another, you hope to get a grasp of the original object. The simple building blocks for CAT(0) cube complexes are hyperplanes. They are simple since they are purely combinatorial objects that record binary information: each point $x \in X$ is either in the ``left" half space $h^-$ or the ``right" half space $h^+$ of a hyperplane $h$. We provide various examples of the powerful structures and constructions that yield from hyperplanes and the binary nature of their half spaces.
\vspace{2mm}

\noindent 1) \textbf{Medians:} Given three vertices $x,y,z \in X$, suppose that you wish to find the ``center" of such three points. Since $X=h^+ \cup h^-$, at least two of the three points must live in the same half space. Before we move on, let's emphasize that the conclusion that at least two of $x,y,z$ are in the same half space was only possible due to the binary nature of our simple building blocks. Let's continue, having established that two of our points $x,y,z$ must be in the same half space of $h$, it seems reasonable that, whatever the center is, it should lie in the half space that contains the majority of our three points, say $h^+.$ Hence, in order to obtain the center, we first choose for each hyperplane the half space containing the majority of $x,y,z$ and take the intersection of all such half spaces (it's not trivial that such an intersection is non-empty and unique, but it is). Such a center is called the \emph{median} for $x,y,z$ and is denoted $m(x,y,z).$

\vspace{2mm}

\noindent 2) \textbf{Hyperbolicity.} Gromov hyperbolicity of a geodesic metric space $X$ is also a purely geometric property, and in the context of CAT(0) cube complexes, hyperbolicity admits the following characterization based on hyperplanes. First, recall that a pair of disjoint hyperplanes $h,k$ is said to be \emph{$L$-separated} if each chain of hyperplanes $c$ meeting both $h,k$ is of cardinality at most $L.$ This property can be thought of as saying that there is a bottleneck of width $ \leq L$ between the two hyperplanes $h,k$. A finite-dimensional CAT(0) cube complex is $\delta$-hyperbolic if and only if every pair of hyperplanes $h,k$ with $\dist(h,k)> \delta'$ is $L$-separated, where the constants $L,\delta, \delta'$ determine each other. Although this is an easy good exercise, the proof of the statement is given in Lemma 6.51 of \cite{genevois:hyperbolicities}.
    
\vspace{2mm}

 \noindent 3) \textbf{Distances via colored walls.} If you fix any subcollection $\calC$ of hyperplanes in a CAT(0) cube complex $X$, then, you can define a new distance $\dist_\calC$ on $X$ by declaring $$\dist_\calC(a,b)=1+|\text{ hyperplanes in }\calC \text{ that separate } a,b|,$$ for any distinct vertices $a,b \in X$ and $\dist_\calC(x,x)=0$ for any vertex $x \in X.$ Namely, instead of counting all the hyperplanes separating $x,y$--which would give their usual combinatorial distance-- we count ``colored" hyperplanes separating $x,y,$ that is, we only count the hyperplanes coming from our special collection $\calC.$

The condition that is usually trickiest to check when establishing that a certain function defines a metric is the triangle inequality, however, the binary nature of our hyperplanes provides a very simple proof of the triangle inequality for $\dist_\calC$. Let $x,y,z$ be vertices of $X$ and let $h \in \calC$ be a hyperplane separating $x,z$, in particular, $h$ contributes to $\dist_\calC(x,z)$. The third vertex $y$ must live on exactly one side of $h$, the side containing $x$ or the one containing $z$. In other words, this hyperplane must either separate $x,y$, in which which case it contributes to $\dist_\calC (x,y)$ or it separates $y,z$ in which case it contributes to $\dist_\calC (y,z)$; this establishes the triangle inequality. The reason for adding the plus $1$ here is that two points $x,y$ may not be separated by any hyperplane from our colored collection $\calC$.

%Lets again emphasize that what made building such a metric ``easy" is the binary nature of our hyperplanes, particularly, the fact that if $h$ separates two points $a,c$ then every other point $b$ must live on exactly on side of $h$. 

This idea was exploited in work of Genevois \cite{genevois:hyperbolicities} where, for any CAT(0) cube complex $X,$ the author defines a distance $\dist_L$ by $$\dist_L(x,y)=\text{max}\{|c| \,:\, c \text{ is a chain whose elements are pair-wise } L-\text{separated}\}.$$

 The metric space $(X,\dist_L)$ is hyperbolic with constants depending only on $L,$ and for each rank-one element (a semi-simple isometry none of whose axes bounds a half flat) $g$, there exists an $L$ such that $g$ acts loxodromically on $X_L.$ The utility of this construction is that it provides an ``increasing" collection of hyperbolic spaces $(X,d_L)$ that collectively captures the hyperbolic aspects of the underlying CAT(0) cube complex. The term increasing comes form the fact that the identity map $i:(X,\dist_{L+1}) \rightarrow (X,\dist_L)$ is 1-Lipshitz and if $g$ acts loxodromically on $X_L$, then it acts loxodromaiclly on all $X_{L'}$ with $L' \geq L$. 

\vspace{2mm}

 \noindent 4) \textbf{Orientations.} Our claim that hyperplanes describe the entire geometry of a CAT(0) is perhaps most evident in this example. First, observe that a vertex $x$ of a CAT(0) cube complex is determined uniquely by orienting each hyperplane towards the half space that contains that vertex $x$, said differently, a vertex $x \in X$ is exactly the intersection of all half spaces that contain that vertex.

     \begin{figure}[ht]
   \includegraphics[width=14cm, trim = 1cm 12cm 6cm 3cm]{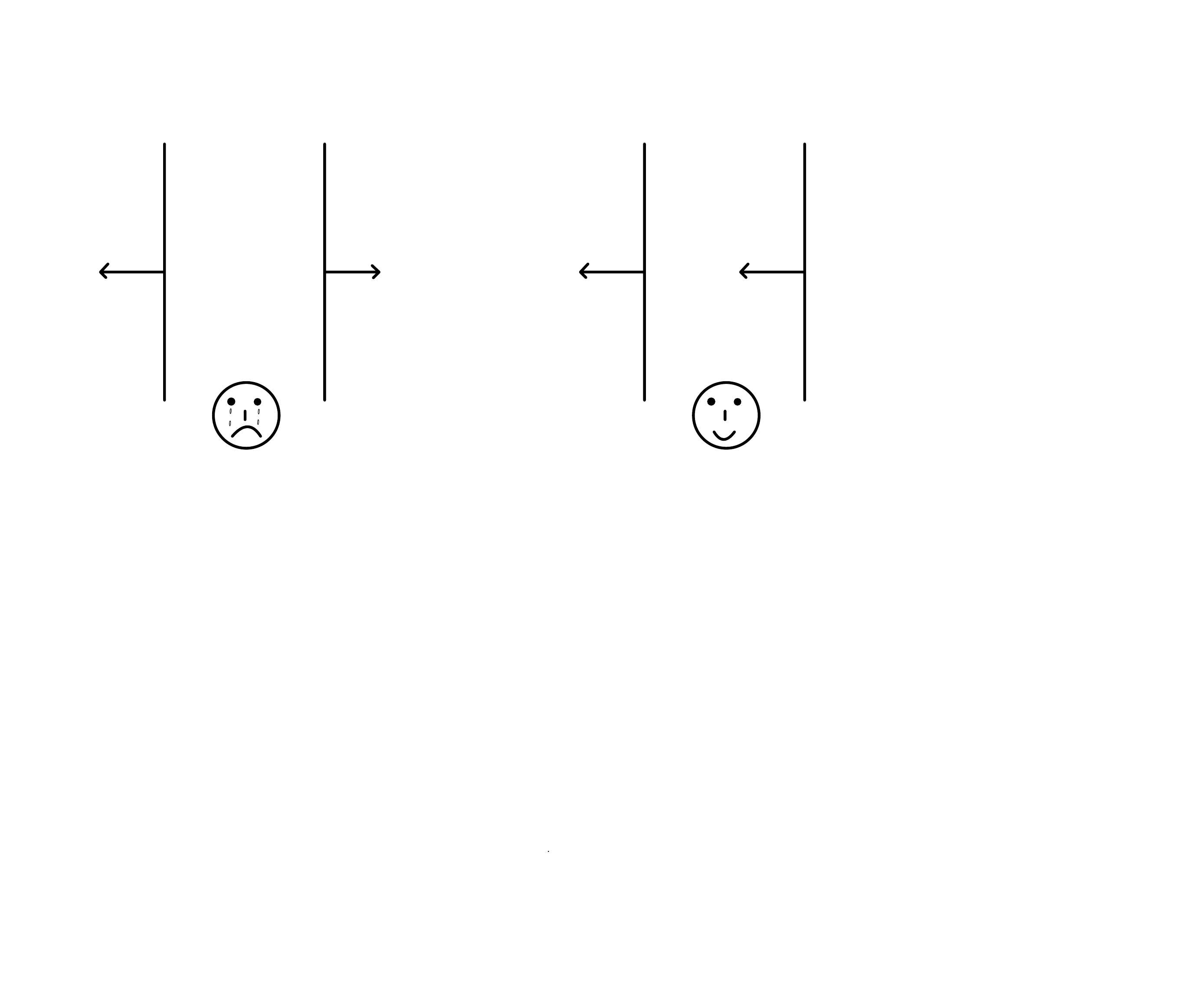}\centering
\caption{The two half spaces chosen in the left picture can't be part of a coherent orientation, while the ones on the right can.} \label{fig:Happy}
\end{figure}

Now, if $\calH$ is the collection of all hyperplanes in $X$, then each vertex $x \in X$ determines a unique point in $\{0,1\}^\calH$, by choosing for each hyperplane $h$ the half space containing $x$ and denoting $\{h^+,h^-\}$ by $\{0,1\}.$ It is natural to ask whether every point of $\{0,1\}^\calH$ defines a vertex in $X$, but it's immediate that the answer to this is negative, for instance, if $h,k$ are two disjoint hyperplanes pointing at opposite directions, then they can't be pointing at the same vertex, in other words, they can't define a vertex $x \in X,$ see Figure \ref{fig:Happy}.

This motivates the definition of a \emph{coherent orientation} which is a special point in $\{0,1\}^\calH$. More precisely, a coherent orientation is a collection of half spaces $\calC$ such that:

\noindent 1) \underline{Exhaustion:} For each hyperplane $h \in \calH$, exactly one half space of $\{h^+,h^-\}$ is in $\calC$.
\vspace{1mm}

 \noindent 2) \underline{Upwards closure:} If $h^+ \in \calC$ and $h^+ \subset k^+,$ then $k^+ \in \calC$.
\vspace{1mm}

 \noindent 3) \underline{Termination:} Every descending chain $\cdots h_1^+ \supseteq h_2^+ \supseteq \cdots $ of $\calC$ has a minimal element.

 \vspace{1mm}

It is immediate that there is a 1-1 correspondence between coherent orientations on hyperplanes $\calH$ and vertices of a finite dimensional CAT(0) cube complex $X.$ Further, an edge $e$ in a cube complex corresponds to two coherent orientations $\calC_1, \calC_2$ differing on exactly one hyperplane.

\vspace{2mm}

 \noindent 5) \textbf{Sageev's construction.} The above example shows the intimate connection between a CAT(0) cube complex and its set of hyperplanes, but the connection is indeed deeper. Choose a collection of bi-partitions $\calH$ of $\mathbb{N}$, that is, each element of $\calH$ is a set $W=\{W^+,W^-\}$ with $\mathbb{N}=W^+ \cup W^-$ and $W^+,W^-$ both non-empty. Suppose that $\calH$ satisfies the property that for any distinct $x,y \in \mathbb{N}$, there are only finitely many elements $W=\{W^+,W^-\}$ of $\calH$ with $x \in W^+$ and $y \in W^-.$ Coherent orientations are defined exactly as in the previous item 6, with $\mathbb{N}$ replacing $X$. Now, we can define a CAT(0) cube complex $X$ via the collection $\calH$ as follows:

\vspace{2mm}

\noindent 1) \underline{Vertices:} The vertex set of $X$ is defined to be all possible coherent orientations on $\calH$.

    \vspace{2mm}
    
   \noindent 2) \underline{Edges:} Two coherent orientations are connected by an edge if they differ on exactly one element of $\calH.$

   \vspace{2mm}

    \noindent 3) \underline{Cubes:} We glue a cube in whenever the faces of the cube are present.

    \vspace{2mm}

It is not immediately clear that the object you get from the above construction is a CAT(0) cube complex, but it is, see \cite{CHATTERJI2005} for an excellent description of the construction.

In short, starting with a collection of bi-partitions $\calH$ on $\mathbb{N}$ satisfying the consistency condition that each pair of points in $\mathbb{N}$ differ on only finitely many $W =\{W^+,W^-\} \in \calH$, there is canonical CAT(0) cube complex associated to $\calH.$ In fact, every CAT(0) cube complex $X$ arises in such a fashion, first, by fixing a bijection $f:\mathbb{N} \rightarrow X^0$ (or a surjection if $X^0$ is finite) between the vertex set of $X$ and the natural numbers $\mathbb{N}$, the collection of $X$-hyperplanes naturally define a family of bi-partitions on $\mathbb{N}$ via the map $f.$ Applying Sageev's construction to such a family of bi-partitions recovers the original CAT(0) cube complex $X.$ This exact fact is what I find most appealing about CAT(0) cube complexes: each canonically arises from a collection of bi-partitions of $\mathbb{N}$ satisfying a simple consistency condition.

\subsection{The median perspective} In addition to hyperplanes, the median structure on a CAT(0) cube complex also records various interesting aspects of its geometry. Namely, half spaces, geodesics, convex hulls and combinatorial projections in CAT(0) cube complexes also admit the following median descriptions.

\vspace{2mm}

\noindent 1) \underline{Half spaces via medians:} For an edge $e$ connecting two vertices $a_1,a_2 \in X$, if $h$ is a hyperplane dual to $e$ (meaning that $h$ separates $a_1,a_2)$, the two half spaces corresponding to $h$ are exactly the sets $\{x \in X| m(a_1,a_2,x)=a_1\}$ and $\{x \in X| m(a_1,a_2,x)=a_2\}$.
\vspace{2mm}

\noindent 2) \underline{Geodesics via medians:} The collection of geodesics connecting a pair of vertices $x,y \in X$ coincides with their \emph{median interval} defined as $$[x,y]=\{m(x,y,z)| z \in X\}.$$

\vspace{2mm}

\noindent 3) \underline{Convex hulls via medians:} Let $A^0=A,$ $A^1=\cup_{x,y \in A}[x,y]$ and define $A^i= (A^{i-1})^1$. It is well-known that (for instance, by \cite{Bowditch2019CONVEXITY}) if the CAT(0) cube complex $X$ is of dimension $n$, then we have $A^n=A^{n+1}$ and in particular, $\hull(A)=A^n$.

\vspace{2mm}

\noindent 4) \underline{Combinatorial projections:} For any (combinatorially) convex set $Y \subset X,$ and any vertex $x \in X,$ there is a unique vertex $P_Y(x) \in X$, called the (combinatorial) \emph{projection} or \emph{gate}, which realizes the distance $\dist(x,Y)$, that is, $\dist(x, P_Y(x))=\dist(x,Y).$ This gate is characterized by hyperplanes as we stated above, but it also admits a description via medians, namely, we have $$P_Y(x)=\cap_{y \in Y}[x,y].$$ One reference for this statement is \cite{Durham-Zalloum22}, but the statement and proof were certainly well-known and likely appear in other places, see \cite{Haglund2007}.
\vspace{2mm}

\noindent 5) \underline{CAT(0) cube complexes via medians:} Recall that a graph $\Gamma$ is said to be a \emph{median graph} if for any three vertices $x_1,x_2,x_3 \in \Gamma$, there is a unique vertex $m=m(x_1,x_2,x_3)$ called the \emph{median} of $x_1,x_2,x_3$ satisfying $\dist(x_i,m)+\dist(m,x_j)=\dist(x_i,x_j)$ for all $i \neq j.$ The intimate connection between CAT(0) cube complex and their medians is perhaps most apparent by the following theorem.

\begin{theorem}[{\cite{Chepoi2000GraphsOS}}] A graph $\Gamma$ is median if and only if it is the 1-skeleton of a (uniquely determined) CAT(0) cube complex.
\end{theorem}

%One nice and characteristic property of CAT(0) cube complexes is that they can be fully described via their hyperplanes and the way they interact with one another. Namely, every countable CAT(0) cube complex arises from a collection of biparitions $\calH$ of the natural number $\mathbb{N}$ satsfying certain consistency conditions. Conversely, hyperplanes in a given CAT(0) cube complex $X$ naturally define a collection of biparitions of the natural numbers satsfying the same consistency conditions. In other words, we have the following diagram.

\subsection{Summary}\label{subsec:Summary_CCC} If $\{X_i\}_{i \in I}$ is the collection of all countable CAT(0) cube complexes, and $\{\calH_i\}_{i \in I} $ is the collection of all bi-partitions of $\mathbb{N}$ satisfying certain consistency conditions, then we have the following 1-1 correspondence:

    $$\{\calH_i\}_{i \in I} \longleftrightarrow \{X_i\}_{i \in I},$$ and the geometry of an underlying CAT(0) cube complex $X$ is fully described by its collection of hyperplanes: $$\calH \longleftrightarrow X.$$

    As we mentioned, CAT(0) cube complexes can also seen via their medians, but hyperplanes are what's most relevant for the purpose of this survey.

\section{Hierarchical hyperbolicity from a median perspective}\label{sec:HHSes}

\subsection{What is an HHS?}\label{subsec:the_Def_HHS}   Motivated by the influential Masur-Minsky hierarchical machinery \cite{masurminsky:geometry:1}, \cite{masurminsky:geometry:2}, Behrstock, Hagen and Sisto introduced the notion of a \emph{hierarchically hyperbolic space}, abbreviated to HHS, as a general framework for studying spaces and groups that can be \textbf{disassembled} into a \textbf{coherent} collection of hyperbolic spaces \cite{HHS1}, \cite{HHS2}. Each HHS $X$ comes with a constant $E$ called the \emph{HHS constant} that describes the ``coarseness" of the space, that is, it bounds the error to which certain inequalities fail to hold. The following two items summarize the essence of an HHS:

    %and  $(E,E)$-coarsely-Lipshitz maps $\{\pi_{U}:X \rightarrow U\}_{U \in \calU}$ to a collection of $E$-hyperbolic spaces $U \in \calU$. These hyperbolic spaces interact with one another in a very specific and prescribed manner, in fact, one can think of HHSes as spaces that satisfy the following conditions:

        \vspace{2mm}

\noindent (1) \underline{Disassembling:} Each HHS $X$ comes with a (typically infinite) family of $(E,E)$-Lipshitz surjections $\{\pi_{U_i}:X \rightarrow U_i\}_{i \in I}$ onto a collection of $E$-hyperbolic spaces $\calU=\{U_i\}_{i\in I}$ called \emph{domains}, see Example \ref{ex:HHS_example} and Figure \ref{fig:First_HHS_Example}.

        %\item (Disassembling) Each HHS $X$ comes with a collection of $E$-hyperbolic spaces $\calU=\{U\}$ and a collection of $(E,E)$-coarsely-Lipshitz maps $\pi_U: X \rightarrow U$, for each $U \in \calU.$

        \vspace{2mm}

\noindent (2) \underline{Coherence:} The maps $\{\pi_{U_i}: X \rightarrow U_i\}_{i \in I}$ interact with one another in a very specific and prescribed manner satisfying certain ``consistency" conditions. For instance, although the family of hyperbolic spaces $\calU=\{U_i\}_{i\in I}$ is typically infinite, it satisfies the following finiteness conditions:
\vspace{2mm}

    \begin{itemize}
\item For any $x,y \in X$ and any $K \geq E$, only finitely many domains $U \in \calU$ satisfy that $\dist(\pi_U(x), \pi_U(y))>K$. In words, $x,y$ have large projections only to finitely many domains, however, such a finite number is not uniform as it depends on the choices of $x,y \in X.$ In fact, if $\text{Rel}_{K}(x,y)$ denotes the collection of domains $U$ with $\dist(\pi_U(x), \pi_U(y))>K$, then the quantities $\dist(x,y)$ and $\sum_{U \in \text{Rel}_{K}(x,y)}\dist(\pi_U(x), \pi_U(y))$ are the same, up to an additive and multiplicative error depending only $K.$ See the second picture of Figure \ref{fig:First_HHS_Example}.

\item If $\alpha:[a,b] \rightarrow X$ is a coarsely continuous path in $X$ connecting $x,y$, then, the number of domains $U \in \calU$ where the projection of $\alpha(t)$ is simultaneously changing is uniformly bounded, see Example \ref{ex:HHS_example} and the picture at the bottom of Figure \ref{fig:First_HHS_Example}. Namely, as $\alpha(t)$ is moving in $X,$ if its projection is substantially changing in a domain $U$, then it must be coarsely still in all but uniformly finitely many domains $V \in \calU.$
    \end{itemize}

For an excellent overview of the theory of an HHS, see Section 3.1 in \cite{russell2020convexity} by Russell, Example 1.2.2 in \cite{SprianoTHesis} by Spriano and \cite{SistoWhatIs} by Sisto.

    %\begin{figure}[ht]
   %\includegraphics[width=10cm, trim = .001cm 8cm 4cm 6cm]{Primary HHS example.pdf}\centering
%\caption{An HHS example} \label{fig:First_HHS_Example}
%\end{figure}

    %\begin{figure}[ht]
   %\includegraphics[width=10cm, trim = .001cm 8cm 4cm 6cm]{Finite complexity.pdf}\centering
%\caption{Increasing distance simultaneously} \label{fig:First_HHS_Example}
%\end{figure}

    %\begin{figure}[ht]
   %\includegraphics[width=10cm, trim = .001cm 8cm 4cm 6cm]{Finitely many domains.pdf}\centering
%\caption{Finite distance} \label{fig:First_HHS_Example}
%\end{figure}

    \begin{figure}[ht]
   \includegraphics[width=8cm, trim = .001cm 3cm 4cm 3cm]{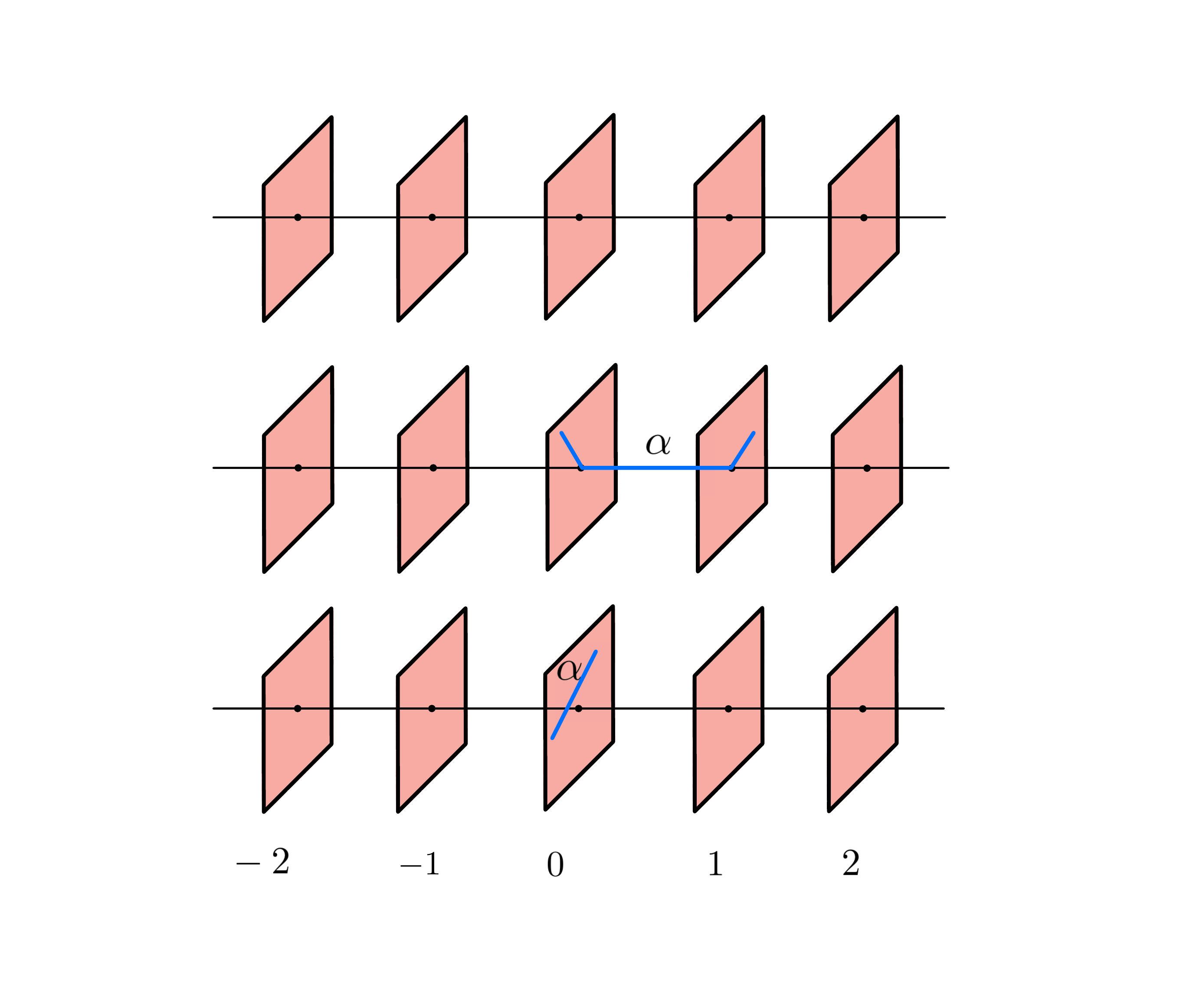}\centering
\caption{The picture on top is an infinite HHS consisting of a line $\calC S=\mathbb{R}$ and a copy of $\mathbb{R}^2$ glued at each integer point as shown. The path $\alpha$ in the second picture has non-trivial projection to only 5 domains, namely, the horizontal lines $U^h_0, U^h_1$, the vertical lines $U^v_0, U^v_1$ and the line $S.$ Finally, the picture in the bottom shows how the projections of a path $\alpha$ can be simultaneously (non-trivially) increasing along only uniformly finitely many domains; the number in this example is 2; the number 2 here is uniform and independent of the chosen path $\alpha.$} \label{fig:First_HHS_Example}
\end{figure}

\begin{example}\label{ex:HHS_example}
    Consider the space $X$ given in Figure \ref{fig:First_HHS_Example} which consists of a line $\mathbb{R}$ and a copy of $\mathbb{R}^2$ glued at each integer point $n \in \mathbb{Z}.$ The collection of hyperbolic spaces associated to $X$ is as follows. First, there is the line $\mathbb{R}$, denoted $\calC S$, where copies of $\mathbb{R}^2$ are all glued. Second, for each $\mathbb{R}^2$ glued at an integer $n \in \mathbb{Z}$, we associate two hyperbolic spaces which are the obvious the horizontal and vertical copies of $\mathbb{R}$ that define $\mathbb{R}^2$, we denote these by $U_n^h, U_n^v$; see Figure \ref{fig:First_HHS_Example}. The maps $\pi_U:X \rightarrow U$ are the usual nearest point projections.
\end{example}

 Each HHS is a \emph{coarse median space} \cite{HHS2}. These can be thought of as metric spaces where for each three points $x_1,x_2,x_3 \in X$, one can assign a point $\mu(x_1,x_2,x_3)$ called the \emph{coarse median}, or the \emph{median} for short, which behaves like a coarse center for $x_1,x_2,x_3$, see \cite{Bowditch13} for more details.
 
 A good example of a coarse median space to keep in mind is an $E$-Gromov hyperbolic space $X$. Namely, given three points $x_1,x_2,x_3 \in X$, as geodesics connecting $x_i,x_j$ all meet a ball $B$ of radius $2E,$ one can choose any point in $B$ and declare that point to be $\mu(x_1,x_2,x_3).$ Clearly, the choice of this point is not canonical (although, it's coarsely canonical).

 A subset $A$ of a coarse median space, an HHS for instance, is said to be \emph{$K$-median-convex} if for any $x,y \in A,$ and any $z \in X$, the point $\mu(x,y,z)$ lies in the $K$-neighborhood of $A.$ For two coarse median spaces $X_1,X_2$, a map $f:X_1 \rightarrow X_2$ is said to be $K$-median if $\dist (f(\mu(c_1,c_2,c_3)), \mu(f(c_1),f(c_2),f(c_3))) \leq K$. A path $\gamma:[a,b] \rightarrow X$ is said to be a \emph{$K$-median path} if it is a $K$-median map in the aforementioned sense. A subset $A$ of an HHS which is $K$-median convex with $K=E$, the HHS constant, will simply be referred to as a \emph{median convex} set, as  the constant is uniform. Similarly, a \emph{median path} is an $E$-median path. Every pair of points $x,y$ in an HHS can be connected by a collection of especially well-behaved paths called \emph{median paths} or \emph{hierarchy paths}. A path $\alpha:[a,b] \rightarrow X$ is said to be a median path if for any $\{t_i\}_{i=1}^3$ with $t_1<t_2<t_3$, we have $\dist(\alpha(t_2), \mu(\alpha(t_1), \alpha(t_2), \alpha(t_3)) \leq E,$ where $E$ is the HHS constant. These are the ``canonical" paths in an HHS, for instance, if a CAT(0) cube complex $X$ is also an HHS, then the collection of $0$-median paths connecting $x,y \in X$ is exactly the set of all geodesics that connect them.

Up to increasing $E$ by a uniform amount depending only on $X$, the \emph{median hull} of a set $A$ can be defined to be the intersection of all $E$-median-convex sets containing $A.$

 Finally, if $Y$ is a median convex set in an HHS, then each point $x \in X$ determines a coarsely unique point in $Y$ called the \emph{gate} of $x$ to $Y$, denoted by $P_Y(x)$ which is coarsely the unique point with $\dist(P_Y(x), \mu(x,y,P_Y(x)))<E$ for all $y \in Y$ (see Lemma 2.57 in \cite{Durham-Zalloum22}). This provides a coarse map $P_Y: X \rightarrow Y$ called the \emph{gate map}. The existence of the gate was originally established by Behrsotck-Hagen-Sisto \cite{HHS2} and was studied further in \cite{RST18} by Russell, Spriano and Tran.

    \subsection{Fundamental geometric notions of an HHS via domains}\label{subsec:fundemental_HHSes}

In Subsection \ref{subsec:fundemental_CCC}, we saw that the fundamental geometric notions of CAT(0) cube complexes including distances, medians, projections and convex hulls can all be described purely in terms of hyperplanes; the CAT(0) cube complex building blocks. Similarly, in HHSes, the fundamental geometric notions including distances, medians, gates and median hulls are all recorded via the associated collection of hyperbolic spaces $\calU=\{U_i\}_{i \in I}$; the HHS building blocks. Namely, we have the following.

 \begin{itemize}
    \item \textbf{Distance:} For each $K \geq E,$ the quantities $\dist(x,y)$ and $\sum_{U \in \text{Rel}_K(x,y)}\dist(\pi_U(x), \pi_U(y))$ coarsely agree up an additive and multiplicative constants depending only on $K$.

    \item \textbf{Median paths:} A path $\gamma \subset X$ is $K_1$-median if and only if $\pi_U(\gamma)$ is $K_2$-median for all $U \in \calU$, where $K_1,K_2$ determine each other.
\item \textbf{Medians:} The median of any $x,y,z \in X$ is coarsely the unique point that maps to the median of $\pi_U(x), \pi_U(y), \pi_U(z)$ for each $U \in \calU$.

    \item \textbf{Median hulls:} A point $x$ lies in the median hull of $A$ in $X$ if and only if $\pi_U(x)$ coarsely lies in the median hull of $\pi_U(A)$ for each $U \in \calU$.

    \item \textbf{Projections/gates}: The gate of a point $x$ to a median convex set $Y$ is coarsely the unique point $z \in X$ such that $\pi_U(z)$ coarsely agrees with the nearest point projection of $\pi_U(x)$ to $\pi_U(Y)$.

    \end{itemize}

For items 1,2,3 and 5 above, see \cite{HHS2} and for item 4 see \cite{RST18}.
 
 %Now, we shall export the coarse median construction from our simple building blocks (i.e., the Gromov hyperbolic spaces $\calU=\{U\}$) to the ambient HHS $X.$ Let $x_1,x_2,x_3 \in X,$ we wish to assign a point $\mu=\mu(x_1,x_2,x_3)$ that behaves like a center for $x_1,x_2,x_3.$ Here is how it's done: one first maps $x_1,x_2,x_3$ to each hyperbolic space $U \in \calU$, then you find a center $m_U$ of the thin triangle $\pi_U(x_1), \pi_U(x_2), \pi_U(x_3)$, and finally (and this is the non-trivial part \cite{HHS2}) one finds a point $p \in X$ that projects under $\pi_U$ to the respective centers $m_U$ for each $U \in \calU$. The median $\mu(x_1,x_2,x_3)$ is then defined to be this point $p$, once the point $p$ is assured to exist, it is easy to show that it's unique, up to an error of $E'$ depending only on $E.$ It's worth noting that such a construction of medians assures that the maps $\pi_U:X \rightarrow U$ are \emph{median maps}: they take medians to medians, up to an error depending only on $X$.

    \subsection{What makes HHSes nice to work with?}\label{subsec:What_makes_them_nice} As we just mentioned, hyperbolic spaces are the simple building blocks for HHSes. They are simple as Gromov hyperbolicity is relatively well understood, and we are familiar with numerous powerful constructions and statements for these. For instance, we know that hyperbolic spaces are coarse median, their median hulls are quasi-isometric to trees, their median-convex sets satisfy a coarse Helly property, their Gromov boundary is particularly well-behaved and their quasi-geodesics satisfy local-to-global properties yielding regular languages for their geodesics. As we shall describe below, each of the aforementioned property for Gromov hyperbolic spaces will provide us with an analogous powerful statement for HHSes, and although the actual constructions are more involved, the core idea that lets us obtain such statements for HHSes is simple: import such a property from our building blocks $\calU=\{U_i\}_{i \in I}$ to the ambient HHS $X$ via the maps $\pi_{U_i}: X \rightarrow U_i$. The goal of the next two subsections is to discuss examples of such constructions and how one obtains them from the collection of Gromov hyperbolic spaces $\calU=\{U_i\}_{i \in I}.$

 The examples we provide come in two types, the first set of examples are constructions that utilize the \textbf{entire} collection of hyperbolic spaces $\calU=\{U_i\}_{i \in I}$ while the second only relies on an \textbf{exceptionally useful} hyperbolic space called the \emph{omniscient}, see \cite{ABD}. The reason this later hyperbolic space is of a particular power is that it simultaneously records all of the hyperbolicity of the ambient HHS, this will be discussed in more depth in the respective Subsection \ref{subsec:omniscient}.

\begin{warning} Although we haven't discussed this, but an HHS can often be equipped with different \emph{HHS structures}, namely, it's possible to have two different collections of hyperbolic spaces $\calU=\{U_i\}_{i \in I}, \calV=\{V_i\}_{i \in I}$ so that $X$ is an HHS with respect to either collection. We will be working with a particularly well-behaved HHS structure $\calU=\{U_i\}_{i \in I}$ introduced by Abbott, Behrstock and Durham in \cite{ABD}. The existence of such a structure requires the HHS to satisfy an extra condition called the \emph{bounded domain dichotomy}, see \cite{ABD}. This assumption is very minor and excludes none of the major examples, in fact, it's satisfied for all \emph{hierarchically hyperbolic groups} which are HHSes that come with an appropriate action of a group $G$ satisfying certain conditions. The name \emph{omniscient} refers to a specific hyperbolic space $\calC S \in \calU$ in their structure, more details regarding the reasons behind its exceptional status will be given in Subsection \ref{subsec:omniscient}.
\end{warning}

\subsection{Constructions that utilize all of their building blocks}\label{subsec:constructions_using_all} In this section, we discuss three fundamental constructions and statements regarding HHSes that utilize all of their building blocks $\calU=\{U_i\}_{i \in I}$. These are the \emph{Helly property} for HHSes, the \emph{boundary} of an HHS as well as the cubical approximation theorem (which we have already discussed, but we shall shed more light on the construction). We start with the Helly property for HHSes.

\vspace{2mm}

\noindent (1) \textbf{The Helly property for HHSes.} Median-convex sets in $E$-hyperbolic spaces satisfy a coarse \emph{Helly property}, that is, if $\{C_i\}_{i \in I}$ is a collection of pairwise intersecting $K$-median-convex sets with $|I|< \infty$ or $\diam(C_i)< \infty$ for some $i \in I,$ then there is a point that is $K'$-close to all of them, for some $K'=K'(K, \delta)$, for instance, see Theorem 6.1 in \cite{Breuillard2021}. Such a Helly property for the HHS building blocks can be imported to the ambient HHS and it results an identical Helly property for median convex sets. Namely, in \cite{HHP}, the authors show the following. Let $X$ be an HHS with constant $E$. If $\{C_i\}_{i \in I}$ is a collection of pair-wise intersecting $K$-median-convex sets satisfying either that $|I|<\infty$ or $\diam(C_i)<\infty$ for some $i$, then there exists a point $x \in \cap_{i \in I} N(C_i,E')$, for some $E'$ depending only on $E$ and $K.$

%Namely, given a collection of median-convex sets $Y_i$, using the previous example, $\pi_U(Y_i)$ is quasi-convex in each $U.$ Now, the Helly property of hyperbolic spaces kicks in resulting a point $x_U$ that is simultaneously close to each $\pi_U(Y_i)$. One can then find a point $x \in X$ (and this is the nontrivial part) that projects to $x_U$ in each $U.$ One can then show that this $x$ is simultaneously close to each $Y_i$ in $X.$

%The claim that the building blocks hyperbolic spaces describe the entire (coarse) geometry of the ambient HHS is perhaps most evident by the celebrated \emph{distance formula}. First, define $\text{Rel}_\theta(x,y)$ to be the collection of all hyperbolic spaces $U$ with $\dist(\pi_U(x), \pi_U(y))>\theta.$ In words, this is the set of all hyperbolic spaces where the projections of $x,y$ are at distance at least $\theta.$

%\begin{example} (Distance formula) For any constant $\theta,$ there exists a constant $K=K(\theta, E)$ such that the following holds. For any $x, y \in X,$ we have 
%$$\dist(x,y) \underset{K}{\asymp} \underset{U \in \text{Rel}_\theta(x,y)}{\sum} \dist(\pi_U(x), \pi_U(y)),$$ where the notation $\underset{K}{\asymp}$ means that the two quantities are equal up to an additive and a multiplicative constant of $K.$ In words, this simply means that the distance between two points $x,y \in X$ is roughly the sum of the distance between their projection to the hyperbolic spaces where $x,y$ have large projections.
%\end{example}

\vspace{2mm}
\noindent (2) \textbf{From hyperbolic trees to hierarchical cubes}. Although the original motivation for introducing HHSes was the discovery that the influential hierarchical machinery for mapping class groups \cite{masurminsky:geometry:1}, \cite{masurminsky:geometry:2} also applies to numerous CAT(0) cube complexes \cite{HHS1}, studying both mapping class groups and cube complexes from the same HHS perspective led Behrstock-Hagen-Sisto to realize that the connection between these objects is much deeper: not only can one export mapping class groups techniques to study CAT(0) cube complexes, but one can also import cubical techniques to study mapping class groups and more generally HHSes. Namely, using the fact that median hulls in hyperbolic spaces are quasi-isometric to trees, Behrstock, Hagen and Sisto proved the following \cite{HHS_quasi}. For any finite set of points in an HHS $F \subset X$, there exists a constant $K$ depending only on $|F|$ such that $\hull(F)$ is $K$-median $(K,K)$-quasi-isometric to a CAT(0) cube complex $Q$. As mentioned in Subsection~\ref{subsec:fundemental_HHSes}, the median hull of a finite set $F$ of an HHS is characterized by being the collection of all points $x \in X$ whose image $\pi_U(x)$ coarsely lies in $\hull(\pi_U(F)) \subset U$ for each $U \in \calU.$ The idea then is that, since median hulls of hyperbolic spaces are known to be quasi-isometric to trees, each  $\hull(\pi_U(F))$ is quasi-isometric to a tree $T_U.$ The CAT(0) cube complex $Q$ in the above theorem is then obtained as a certain (typically not convex) subspace of $\prod_{U \in \text{Rel}(F)} T_U$. Here, $\text{Rel}(F)$ is the collection of $U \in \calU$ where $\pi_U(F)$ has a large enough diameter, depending only on $E.$

   %While aware of my personal bias, this is perhaps among the most powerful and revolutionary recent theorems in the field as it opens the door to exporting numerous cubical tools and constructions from the world of CAT(0) cube complexes to that of mapping class groups and HHSes. This very philosophy has been successfully implemented in various recent works leading to the resolution of two long standing open problems in the field as highlighted in Section \ref{intro:impact}.
   \vspace{2mm}

\noindent (3) \textbf{The HHS boundary}. Utilizing the well-studied Gromov boundaries of hyperbolic spaces, in \cite{Durham2017-ce}, Durham, Hagen and Sisto show that each HHS $X$ admits a boundary notion $\partial X$ that provides a compactification for $X$. They also show that $\partial X$ enjoy many desirable properties similar to that of the Gromov boundary of a hyperbolic space.

\subsection{Constructions that utilize the omniscient}\label{subsec:omniscient} Let $X$ be an HHS and let $\calU=\{U_i\}_{i \in I}$ be the associated collection of hyperbolic spaces (once more, here, we are working with the specific HHS structure given in \cite{ABD} which requires $X$ to satisfy some very minor extra assumptions). Among elements of $\calU$, one hyperbolic space, denoted $\calC S$, is of a particular importance. Up to quasi-isometry, the space $\calC S$ is exactly what results when collapsing all the non-trivial ``product-regions" of $X$. We shall refer to $\calC S$ as \emph{the omniscient}. The special status of $\calC S$ come from the fact that it simultaneously records  \textbf{the entirety} of the hyperbolic aspects of the ambient HHS, which makes it exceptionally useful. For example, the \emph{Morse} and \emph{sublinearly Morse} boundaries -- which are topological spaces that record the ``hyperbolic aspects" of a proper geodesic metric space \cite{charneysultan:contracting}, \cite{cordes:morse}, \cite{Cashen2019}, \cite{QRT19}, \cite{QRT20} -- both continuously inject in the Gromov boundary of $\calC S$ as shown in \cite{ABD} for the Morse boundaries, and in \cite{Durham-Zalloum22} for the sublinearly Morse boundaries. Additionally, in the context of mapping class groups of finite type surfaces with genus $g \geq 3$, the isometry group of $\calC S$ is isomorphic to the underlying mapping class group; serving as another evidence of the special significance held by the omniscient $\calC S$.

In the previous section, we discussed how one can import various constructions from the associated hyperbolic spaces to $X$. This section is a continuation of this philosophy, but the constructions and statements we present here shall only make use on the omniscient $\calC S \in \calU$. Such examples also serve as another justification of its special status.

\vspace{2mm}

 \noindent (1) \textbf{Morse local-to-global}. Gromov hyperbolic spaces with constant $E$ are characterized by the property that sufficiently long $(\lambda, \epsilon)$-quasi-geodesics are global $(\lambda',\epsilon')$-quasi-geodesics where $\lambda', \epsilon'$ depend only on $\lambda, \epsilon$ and $E.$ They are also characterized by the property that every quasi-geodesic is \emph{$M$-Morse}; recall that a quasi-geodesic $\alpha$ is said to be $M$-Morse if there is a map $M: \mathbb{R}^+ \times \mathbb{R}^+ \rightarrow \mathbb{R}^+$ such that every $(\lambda, \epsilon)$-quasi-geodesic $\beta$ with end points on $\alpha$ remains in the $M(\lambda, \epsilon)$-neighborhood of $\alpha.$ Since Morse geodesics characterize hyperbolic spaces, their presence in other geodesics metric spaces can be viewed as a sign of hyperbolicity, and it's reasonable to question whether one can export statements from Gromov hyperbolic spaces to geodesic metric spaces containing such Morse geodesics. This perspective turns out to be indeed fruitful, and many properties of Gromov hyperbolic spaces extend to spaces that contain infinite Morse geodesics, see \cite{Cordes2017ASO}. However, various properties of Gromov hyperbolic spaces do not only rely on Morseness of their geodesics, but also on the local-to-global property met by such geodesics, for instance, regarding the regularity of the geodesic language and rationality of the growth function \cite{Cannon1984}. This motivated Russell, Spriano and Tran \cite{Morse-local-to-global} to introduce the class of \emph{Morse local-to-global} spaces which is a class of spaces where Morse quasi-geodesics enjoy not only the Morse property, but also a similar local-to-global property to the one present in Gromov hyperbolic spaces. They have shown that numerous examples of geodesic metric spaces have this property. In particular, they have shown that if the collection of all Morse quasi-geodesics in a metric space $X$ quasi-isometrically embeds in a certain Gromov hyperbolic space $Y$ (with appropriate quantifiers, see \cite{Morse-local-to-global} for more details), then $X$ has the Morse local-to-global property. The idea of the proof of this theorem is to simply import the local-to-global property from the hyperbolic space $Y$ where Morse quasi-geodesics embed back to the original space $X.$ Since Morse quasi-geodesics in HHSes all quasi-isometrically embed in the omniscient $\calC S$ (again, with certain quantifiers), all HHSes have the Morse local-to-global property. Spaces and Cayley graphs of groups satisfying the Morse local-to-global property have been shown to enjoy various desirable properties similar to those in hyperbolic spaces and groups, for instance, regarding the presence of infinite order elements \cite{Morse-local-to-global}, regularity of Morse geodesics and rationality of the growth of \emph{stable} subgroups, see \cite{Cordes2022} and \cite{Huges2022}.

\vspace{2mm}

\noindent (2) \textbf{Establishing acylindrical hyperbolicity}. Motivated by work of Bowditch \cite{bowditch:tight}, in \cite{osin:acylindrically}, Osin introduced the class of \emph{acylindrically hyperbolic groups} and showed that such a class enjoy various useful algebraic and geometric properties reminiscent to that of Gromov hyperbolic groups. The defining terms of such a class are not important at this point, what is important is that the property of acylindrical hyperbolicity is a very desirable with numerous powerful consequences. In the class of HHGs (groups admitting appropriate actions on HHSes), acylindrical hyperbolicity is characterized purely by the omniscient, namely, an HHG is acylindrically  hyperbolic if and only if the omniscient $\calC S$ has an infinite diameter.

\vspace{2mm}

\noindent (3) \textbf{Characterizing Morse elements and stable subgroups}. An element $g \in G$ is said to be Morse if it admits quasi-geodesic orbit $\langle g \rangle . p$ that is $M$-Morse for some $M$. More generally, a subgroup $H<G$ is said to be $(M,K)$-stable \cite{DurTay15} if every pair of points $x,y \in \text{Cay}(G,A) \cap H$ are connected by an $M$-Morse geodesic that remains $K$-close to $H$ in $\text{Cay}(G,A)$, see also \cite{tran:onstrongly}. It's easy to see that a cyclic subgroup is stable if and only if it is generated by a Morse element $g \in G.$ Stable subgroups in HHGs are precisely subgroups $H$ whose orbit maps in $X$ quasi-isometrically embed in $\calC S.$

\vspace{2mm}
\noindent (4) \textbf{Seeing the entire Morse boundaries}. As we mentioned, both the Morse \cite{ABD} and sublinearly Morse boundaries \cite{Durham-Zalloum22} admit continuous injections into the Gromov boundary of the omniscient. This allows one to export certain desirable topological properties from the Gromov boundary of the omniscient back to the Morse and sublinearly Morse boundaries of HHSes. This philosophy was recently successfully implemented by Abbott and Incerti-Medici \cite{Abbott-Medici-23} as we shall describe below.

Let $G$ be a group admitting a geometric action on a CAT(0) cube complex $X$. There are various topologies that the Morse and sublinearly Morse boundaries of $G$ can be equipped with, each of which has its pros and cones. The first and most natural choice that comes to mind when trying to equip these boundaries with a topology is the subspace topology of the \emph{visual boundary} of $X$ (recall that the visual boundary of a CAT(0) space consists of all geodesic rays emanating from a fixed base point $p$ and two rays are close in this topology if they fellow travel for a long time). However, there are examples where such a topology is not invariant under quasi-isometries \cite{Cashen2016}. On the other hand, while the direct limit topology given in \cite{charneysultan:contracting} by Charney-Sultan is a quasi-isometry invariant, it is not generally metrizable \cite{murray:topology}. Finally, the topologies given in \cite{Cashen2019} by Cashen-Mackay and in \cite{QRT19} by Qing-Rafi-Tiozzo (which agree on the Morse boundary by work of He in \cite{He2022}) are metrizable quasi-isometry invariant topologies; however, understanding when two elements are close with respect to this topology is generally challenging. If the CAT(0) cube complex $X$ on which $G$ acts geometrically is also an HHS, Abbott and Incerti-Medici \cite{Abbott-Medici-23} recently showed that our first natural choice of the subspace topology of the visual boundary does the job. Namely, they show that this topology is metrizable, independent of the space $X$ on which $G$ acts geometrically, and finally, understanding when two elements $\alpha, \beta$ in such a boundary are close is quite easy. To establish their theorem, the authors rely on the fact that both Morse and sublinearly Morse boundaries of HHSes continuously inject in the Gromov boundary of the omniscient $\calC S$, namely, they export the topology of the Gromov boundary via such an injection and they show that it coincides with the subspace topology of the visual boundary of the cube complex $X.$

\vspace{2mm}

\noindent (5) \textbf{Ivanov's Theorem} The primary and motivating examples for the class of HHSes is the class of mapping class groups. The omniscient hyperbolic space in this case is the \emph{curve graph} $\calC S$ of the ambient surface. In this context, the exceptional utility of the omniscient is evident by the celebrated \emph{Ivanov's theorem} stating the following. For any finite type surface of genus at least $3$, the mapping class group of that surface is isomorphic to the isometry group of the curve graph. In other words, the omniscient here sees all there is to see about the underlying mapping class group of that surface.

    \vspace{2mm}

  \subsection{Examples of hierarchically hyperbolic spaces} The class of HHSes is very broad and it includes mapping class groups \cite{MM99,MM00, HHS1} and Teichm\"uller spaces of finite-type surfaces \cite{Raf07, Dur16}, CAT(0) cube complexes with factor systems \cite{HagenSusse2020}, extra-large type Artin groups \cite{HMS21}, the genus two handlebody group \cite{miller2020stable}, surface group extensions of lattice Veech groups \cite{DDLS1,DDLS2} and multicurve stabilizers \cite{russell2021extensions}, and the fundamental groups of 3–manifolds without Nil or Sol components \cite{HHS2}, as well as various combinations and quotients of these objects \cite{BHMScombo, BR_combo}.

    \subsection{Summary} We now summarize the main points of this section and compare it to the situation of CAT(0) cube complexes.

    \subsubsection{Cube complexes via their hyperplanes}

    As we discussed in Subsection \ref{subsec:Summary_CCC},  if $\{X_i\}_{i \in I}$ is the collection of all countable CAT(0) cube complexes, and $\{\calH_i\}_{i \in I} $ is the collection of all bi-partitions of $\mathbb{N}$ satisfying certain consistency conditions, then we have the following 1-1 correspondence:

    $$\{\calH_i\}_{i \in I} \longleftrightarrow \{X_i\}_{i \in I},$$ and the geometry of an underlying CAT(0) cube complex $X$ is fully described by its simple building blocks which are the associated collection of hyperplanes: $$\calH \longleftrightarrow X.$$

    \subsubsection{HHSes via their hyperbolic spaces}

 Let $\calB$ denote the collection of all hyperbolic spaces. It turns out that there is a similar 1-1 correspondence between subsets $\calU \subset \calB$ satisfying certain consistency conditions and HHSes \cite{HHS2}. 

    $$\{\calU_i\}_{i \in I} \longleftrightarrow \{X_i\}_{i \in I},$$ and the coarse geometry of a given HHS $X$ is completely captured by the associated collection of hyperbolic spaces $$\calU \longleftrightarrow  X.$$

The goal of the next section is to show that CAT(0) cube complexes building blocks --hyperplanes-- exist in HHSes and capture many elements of its geometry.

\section{From HHSes to curtains}\label{sec:from_HHS_to_curtains} What motivated the introduction and formalism of the class of HHSes is Behrstock-Hagen-Sisto's discovery that a broad class of CAT(0) cube complexes enjoy a rich hierarchical structure similar to that of mapping class groups. Namely, there are standard hyperbolic spaces that one can associate to a given CAT(0) cube complex, and since hyperplanes encode the entire geometry of a CAT(0) cube complex, it should not be surprising that such hyperbolic spaces are themselves built via hyperplanes. In a large class of CAT(0) cube complexes, such hyperbolic spaces do satisfy the HHS consistency conditions mentioned above turning the underlying CAT(0) cube complex $X$ into an HHS $X$. See the diagram in Figure \ref{fig:Cube_complex_to_HHS}.

\begin{figure}
    \centering
   $$ X_{\text{CAT(0) cube complex }} \rightarrow	 \calH_{\text{hyperplanes }} \rightarrow 
\calU_{\text{hyperbolic spaces }} \rightarrow X_{\text{ HHS }}.$$    \caption{CAT(0) cube complexes $X$ come with their hyperplanes $\calH$, and these hyperplanes are used to build hyperbolic spaces $\calU.$ For a broad class of CAT(0) cube complexes, the hyperbolic spaces $\calU$ satisfy the consistency conditions turning the cube complex $X$ into an HHS.}
    \label{fig:Cube_complex_to_HHS}
\end{figure}

%\begin{figure}
%\begin{tikzcd}
  %\text{A CAT(0) cube complex } X \arrow[r, shorten=3mm]
  %  &\text{Its hyperplanes } \calH  \arrow[r, shorten=3mm]
%& \text{ Hyperbolic spaces }\calU  \arrow[r, shorten=3mm]
%\arrow[r, shorten=3mm] 
       % & \text{An HHS}
%\end{tikzcd}
%\end{figure}

\vspace{2mm}
%To summarize, the hyperplane structure $\calH$ enjoyed by a CCC can be used to produce a collection of hyperbolic spaces $\calU$ turning the CAT(0) cube complex $X$ into an HHS $X$ and thus the diagram $\calH \Longrightarrow \calU \Longrightarrow X.$

A natural question to ask at this point is the converse, namely, to which extent can the arrows of the diagram in Figure \ref{fig:Cube_complex_to_HHS} be reversed? That is, can the hyperbolic collection $\calU$ associated to an HHS always be interpreted as a hyperbolic collection resulting from a set of hyperplanes coming from the underlying HHS? Surprisingly, the answer turns out to be positive, and in a very strong and equivariant sense. However, before we proceed, we shall remark that the hyperplane structure enjoyed by an HHS does \textbf{not} satisfy the same consistency conditions enjoyed by hyperplanes in a CAT(0) cube complex, as otherwise, we would be showing that every HHG is cubulated which we know is incorrect \cite{Kapovich1996}. We now recall the definition of an HHS \emph{curtain} from the introduction.

\begin{definition}(curtains) Let $X$ be an HHS and let $E$ be the HHS constant. A \emph{curtain}, is defined to be a median convex set $h$ such that $X-h \subseteq h^+ \sqcup h^-$ where $h^+, h^-$ are disjoint, median-convex sets with $\dist(h^+, h^-) \geq 10E.$ The sets $h^+,h^-$ are called \emph{half spaces} for $h.$
\end{definition}

A priori, it's possible that the above definition is vacuous in HHSes, but as we shall see, the building blocks of an HHS, i.e. the collection of hyperbolic spaces $\calU=\{U_i\}_{i \in I}$ shall provide us with an abundance of curtains. Let $\calU$ be the collection of hyperbolic spaces  associated to an HHS $X$ and let $U \in \calU$ be an $E$-hyperbolic space. Fix an interval $I$ of length $10E$ along geodesic $b \subset U$. There is a natural map from $X$ to $b$ defined by first mapping $X$ to the hyperbolic space $U$ containing $b$ via $\pi_U:X \rightarrow U$, and then applying the nearest point projection map $\pi_b:U \rightarrow b$. Under the composition of the above two maps, each point $x \in X$ must map either inside $I$, to the left or right part of $I$. Hence, the set $\pi_U^{-1}\pi_b^{-1}(I)$ splits $X$ into a left, middle and right part in the obvious way, and in order to assure that the set is median convex, one considers $\hull(\pi_U^{-1}\pi_b^{-1}(I))$. Namely, its immediate to check that for each $U \in \calU$ a geodesic $b \subset U$ and an interval $I \subset b$, the set $$h_{U,b,I}:=\hull(\pi_U^{-1}\pi_b^{-1}(I))$$ is indeed a curtain. In fact, as a reflection of the fact that $U=\cal\{U_i\}_{i \in I}$ describe the entire coarse geometry of an HHS, these curtains shall do the same.

For instance, notice that if $\alpha$ is a path making a large distance in a hyperbolic space $U$ along a geodesic $\alpha_U \subset U$, then $\alpha$ must be crossing many curtains of the form $h_{U,\alpha_U,I}$. Similarly, backtracking along a geodesic $\alpha_U \subset U$ is the same as crossing and then uncrossing many hyperplanes of the form $h_{U,{\alpha_U},I}$. Now let $\calH$ denotes the collection of \textbf{all} curtains and let $$\calH^\calU:=\{h_{U,b,I}| U \in \calU, b \subset U \text{ is a geodesic and } I \subset b \text{ is an interval of length }10E\}.$$ By definition, we have $\calH^\calU \subset \calH$ and in fact, $\calH^\calU$ is an invariant sub-collection (under $\Aut(X)$ which we haven't defined, but that's not important for now). To be clear, an arbitrary curtain is merely an element $h$ of $\calH$ which might and might not be in $\calH^\calU \subset \calH$. The primary use of the collection $\calH^\calU$ is to establish that $\calH$ is a sufficiently rich collection; this in turns will imply that the collection $\calH$ captures many non-trivial elements of the coarse geometry of $X$.

\begin{figure}
    \centering
$$ \calH_{\text{curtains}} \leftarrow  \calU_{\text{hyperbolic spaces }}\leftarrow X_{\text{HHS} }.$$ 
    \caption{An HHS $X$ comes with its collection of hyperbolic spaces $\calU=\{U_i\}_{i \in I}$, and the latter allow us to show that curtains in $X$ are abundant. Namely, by consider geodesics $b \subset U$, intervals $I \subset b$ and analyzing where a given point $x \in X$ project on $b$ with respect to $I.$}
    \label{fig:reversing_HHS_Arrows}
\end{figure}

\vspace{2mm}

%The collection of such ``hyperplanes" are called \emph{curtains}, and as we mentioned, these curtains describe a great deal of the geometry of the HHS. For instance, distances, hierarchy paths, hierarchichally quasi-convex sets, hierarchical hulls, and isometry types can all be described purely in terms of these curtains (See section bla for the precise).

\begin{lemma}[Petyt-Spriano-Zalloum]\label{lem:sample_body} Let $X$ be an HHS. Up to increasing the HHS constant $E$ by a uniform amount, the following hold.

\begin{enumerate}

    \item $d(x,y)$ coarsely coincides with the cardinality of a maximal chain of curtains separating $x,y.$

    \item A quasi-geodesic $\alpha$ is a median path (equivalently, a hierarchy path) if and only if it never crosses the same curtain twice.

    \item For any three points $x,y,z \in X$, their coarse median $\mu(x,y,z)$ coarsely agrees with the intersection of the $E$-neighborhoods of all half spaces $h^+$ containing the majority of $x,y,z.$

    \item Let $Y$ be a median convex set in $X$ and let $x \in X$. The gate $P_Y(x)$ is coarsely the unique point in $Y$ with the property that whenever $c$ is a chain of curtains separating $x,P_Y(x),$ then all but $E$-many curtains of $c$ separate $x,Y.$

    \item For any set $A$, the median hull of $A$ is coarsely the intersection of the $E$-neighborhood of all half spaces properly containing $A.$
\end{enumerate}
    
\end{lemma}

Combining the above lemma with the discussion in Subsection \ref{subsec:fundemental_HHSes}, we see that the fundamental geometric notions of an HHS which were described via the hyperbolic collection $\calU=\{U_i\}_{i \in I}$ can also be characterized in terms of curtains.

Further, notice that the curtain's descriptions in the above lemma are exactly the ``coarsening" of the respective hyperplane statements in CAT(0) cube complexes, see Subsection \ref{subsec:fundemental_CCC}.

\subsection{The omniscient's curtains} As we discussed in Subsection \ref{subsec:omniscient}, each HHS comes with a special hyperbolic space $\calC S \in \calU$ that we called the omniscient. The collection of curtains $\calH^\calU$ consists of curtains that come from geodesics $b \subset U$ with $U \in \calU$. Given the special status of the omniscient $\calC S \in \calU$, it is natural to wonder how curtains coming from geodesics $b$ living in the omniscient $\calC S$ differ from other curtains in $\calH^\calU.$ The following lemma addresses this question (for those who are familiar with HHSes, the following lemma is a characterization in the context of the ABD structure given in \cite{ABD}, but that's not relevant for now). First, following \cite{PSZCAT}, a pair of HHS curtains $h_1,h_2$ are said to be \emph{$L$-separated} if every chain of curtains $c$ meeting both has cardinality $|c| \leq L.$

\begin{lemma}[Petyt-Spriano-Zalloum]\label{lem:omniscients_curtains}
 For $i \in \{1,2\}$, let $U=\calC S$, $b_i \subset U$ be a geodesic, $I_i\subset b_i$ be an interval and let $h_i:=h_{U,b_i, I_i}\subset X$. There are constants $E', L$, depending only on $E,$ such that if $\dist(\pi_U(h_1),\pi_U(h_2))>E',$ then $h_1,h_2$ are $L$-separated.
\end{lemma}

The key ingredient for the proof of the above lemma is Lemma 7.10 in \cite{Durham-Zalloum22} which states that for any two sets $A,B$ projecting far in $\calC S$, the gate map of the median hull to the other is uniformly bounded, see Appendix A for the detailed proof. While Lemma \ref{lem:omniscients_curtains} concerns the omniscient's curtains, the following lemma applies to any pair of curtains, the proof will also be provided in Appendix A.

\begin{lemma} [Petyt-Spriano-Zalloum]\label{lem:points_versus_their_gates} Let $h_1,h_2$ be two disjoint curtains and let $x,y \in h_1$. If $c$ is a chain of curtains separating $x,y$, then all but at most $E'$-many curtains of $c$ separate $P_{h_2}(x), P_{h_2}(y)$, where $E'$ is a constant depending only on $E.$
    
\end{lemma}

The following lemma follows immediately by combining Lemma 7.10 and Corollary 7.11 in \cite{Durham-Zalloum22} (also, compare the statement with Proposition 14 in \cite{Genevois16} and Lemma 7.12 in \cite{Durham-Zalloum22}), see Appendix A for more details.

\begin{lemma}[Petyt-Spriano-Zalloum]\label{lem:separated_iff_gate} Let $X$ be an HHS. Two disjoint curtains $h_1,h_2 \subset X$ are $L$-separated if and only if $\diam (P_{h_1}(h_2)) \leq L'$, where $L,L'$ determine each other.
\end{lemma}

   \section{CAT(0) spaces, their curtains and curtain model}\label{sec:CAT(0)_curtains}

   \subsection{A comparison with HHSes}\label{subsec:a_comparsion_with_HHSes}

One similarity between CAT(0) spaces and HHSes is that in the same way an HHS comes with a collection of hyperbolic spaces $\calU=\{U_i\}_{i \in I}$ and projections to these spaces $\{\pi_U: X \rightarrow U_i\}_{i \in I}$, CAT(0) spaces do also come with projections to hyperbolic spaces. Namely, an almost characteristic property of CAT(0) spaces is the existence and uniqueness of \emph{nearest point projections} to closed convex sets, and particularly to geodesics. Hence, if $\calL$ denotes the collection of all geodesic lines in $X,$ then, we have the nearest point projection maps $\pi_l: X \rightarrow l \cong \mathbb{R}$ for each $l \in \calL$ (assuming $X$ has the geodesic extension property, but that's a very technical point and it's not strictly needed).  See Figure \ref{fig:First_comparison} below.

     \begin{figure}[ht]
    \centering
\begin{tikzcd}  
 X_{\text{HHS}} \rightarrow \calU_{\text{hyperbolic spaces}}
\end{tikzcd}%
    \qquad \hspace{10mm}
\begin{tikzcd}  
 X_{\text{CAT(0)}} \rightarrow \calL_{\text{hyperbolic lines}}
\end{tikzcd}    \caption{HHSes come with a collection of hyperbolic spaces $\calU=\{U_i\}_{i \in I}$ and projections to these hyperbolic spaces $\pi_U: X \rightarrow U$, and CAT(0) spaces come with a collection of hyperbolic lines $\calL$ and projections to these hyperbolic lines $\pi_l: X \rightarrow l$.}%
    \label{fig:First_comparison}%
\end{figure}

   %In Subsection \ref{sec:from_HHS_to_curtains}, we discussed how the projection maps $\pi_U:X \rightarrow U$ allow us to define a collection of HHS curtains in $X$ relying on fibers of $\pi_U.$ 

   %The main goal of this section is to attempt and do the same for the projection maps $\pi_l:X \rightarrow l$ from CAT(0) spaces $X$ to geodesic lines $l \subset X,$ but first, we shall highlight some differences between the projections in the HHS setting $\{\pi_U:X \rightarrow U\}_{U \in \calU}$ to that of the CAT(0) setting $\{\pi_l:X \rightarrow l\}_{U \in \calL}$.
   
   %Having emphasized some similarities between CAT(0) spaces and nd , we shall point out some similarities and differences between the HHS projections $\{\pi_U:X \rightarrow U\}_{U \in \calU}$ and the CAT(0) projections $\{\pi_l:X \rightarrow l\}_{U \in \calL}.$

   %\textcolor{brown}{ You can't get all with the same consistency conditions and be invariant under the entire isometry group. Well, how about getting only one? Which one would you want to get? Of course, the omniscient.}

   A major distinction between the two settings, however, is the set of rules dictating how the hyperbolic spaces in $\calL$ versus those in $\calU$ interact with one another. Namely, in the HHS settings, the hyperbolic spaces $\calU=\{U_i\}_{i \in I}$ and maps $\pi_{U_i}: X \rightarrow U_i$ satisfy various \textbf{finiteness} conditions that are completely absent in the CAT(0) case. For instance, each pair of points $x,y$ in an HHS have large projections only to finitely many hyperbolic spaces $U \in \calU,$ however, and as can be seen from simple CAT(0) spaces like $X=\mathbb{R}^2$, a pair of points $x,y$ can have large projections on uncountably many lines $l \in \calL.$

\subsection{Too many curtains, a drawback or an advantage?}
   
   The CAT(0) example of $\mathbb{R}^2$ is indeed an HHS, but not with respect to the collection $\calL$ described above. To see that $\mathbb{R}^2$ is an HHS, one considers the two obvious horizontal and vertical lines (the $x$-axis and $y$-axis), denoted $l_h,l_v$ with the nearest point projection maps $\pi_h:\mathbb{R}^2 \rightarrow l_h$, $\pi_h:\mathbb{R}^2 \rightarrow l_v.$ These two lines $\calU=\{l_h, l_v\}$ suffice for understanding the coarse geometry of $\mathbb{R}^2$; for instance, notice that the Euclidean distance between a pair of points $x,y \in \mathbb{R}^2$ coarsely agrees with the sum $\dist (\pi_h(x), \pi_h(y))+\dist(\pi_v(x),\pi_v(y)).$

   %Further, notice that the curtains we defined in the previous section here consists of two sets of curtains, the vertical ones which are the fibers of the nearest point projection $\pi_{l_v}:\mathbb{R}^2 \rightarrow l_h$ to the $x$-axis $l_h,$ as well as the horizontal ones occurring as fibers of $\pi_{l_v}:\mathbb{R}^2 \rightarrow l_h$. 

   Importantly, however, the hyperbolic spaces $\calU=\{l_v, l_h\}$ are not invariant under $\isom(\mathbb{R}^2)$ and neither are the HHS curtains they define (all of which are either parallel to $l_v$ or $l_h$). Since we will make use of this observation later, we shall record it as a separate remark (in fact, the remark below is slightly more general that what we just discussed, but the proof follows similar lines).

   \begin{remark}\label{rmk:no_HHS_structure} The CAT(0) space $\mathbb{R}^2$, with respect to the Euclidean distance, admits no HHS structure which is invariant under $\isom(\mathbb{R}^2)$.
   \end{remark}

   Similarly, if we put aside the HHS perspective and only consider $\mathbb{R}^2$ as a CAT(0) cube complex with its usual collection of hyperplanes $\calH$, the collection $\calH$ is also not invariant under the entire isometry group $\isom(\mathbb{R}^2)$, for instance, the rotation by $\frac{\pi}{4}$ fixing the origin does not map $\calH$ to $\calH.$ Therefore, if one is interested in studying the entire isometry group $\isom(\mathbb{R}^2)$ from a hyperplane perspective, it's necessary to consider a larger class of hyperplanes. This is what Petyt, Spriano and the author do in \cite{PSZCAT}. Namely:

   \begin{definition}[CAT(0) curtains] Let $X$ be a CAT(0) space. A CAT(0) \emph{curtain} --or simply, a curtain-- is defined to be $\pi_l^{-1}(I)$ where $l$ is a geodesic, $I \subset l$ is an interval of length one, and $\pi_l:X \rightarrow l$ is the nearest point projection map.

   %Given a geodesic $l$ (not necessarily a line) and a interval $I \subset l$ of length $1$ not containing the end point, a \emph{curtain dual to $l$ at $I$} is defined to be $h_{l,I}:=\pi_l^{-1}(I)$. A \emph{curtain} (or simply, a curtain) $h$ then is defined to a curtain dual to $l$ at $I$ for some geodesic $l \subset X$ and interval $I \subset l$ not containing the end point. Finally, we shall say that two curtains $h_1,h_2$ are $L$-separated if every chain of curtains meeting both has cardinality at most $L.$

      \end{definition}

   It is immediate from the definition that the collection of all curtains $\calH$ is invariant under the entire isometry group of a CAT(0) space $X.$ In fact, the following statement shows that curtains are rather a canonical extension of the notion of a hyperplane.

\vspace{2mm}

\noindent \underline{Observation}: Let $G$ be a group acting on a CAT(0) space $X$ by isometries and let $\calH$ denote the collection of all curtains in $X$. The space $X$ admits a $G$-invariant a cubical structure if and only if there is a discrete $G$-invariant subcollection of curtains $\calH' \subset \calH.$ 
\vspace{2mm}

To avoid confusion, we shall quickly remind the reader of how we defined curtains in the HHS setting compared to the CAT(0) setting. In the HHS setting, a curtain was merely an $E$-median-convex set $h \subset X$ such that $X-h \subset h^+ \sqcup h^-$ where $h^+,h^-$ are disjoint $E$-median-convex sets, and we denoted the collection of all such curtains $\calH.$ The main source of curtains in this setting came from the maps $\pi_U:X \rightarrow U$; namely, by fixing a geodesic line $l \subset U$, an interval $I \subset l$ of length $10E$ and considering the median hull of the fiber $(\pi_l \circ \pi_U)^{-1}(I).$ For $U \in \calU$, let's denote the collection of curtains coming from lines in $U$ by $\calH^U$. Recall that $\calH^\calU$ was defined as $\calH^\calU:= \cup_{U \in \calU}\calH^U.$ In particular, the collection of all HHS curtains $\calH$ includes an especially well-understood sub-collection $\calH^\calU$, nonetheless, a ``curtain" is merely an element of $\calH$ and not necessarily $\calH^\calU.$ Finally, we defined two curtains $h_1,h_2$ to be $L$-separated if every chain of curtains meeting both has cardinality at most $L.$

   \subsection{Towards an analogue of the HHS omniscient}
%\textcolor{brown}{diagram factoring through}

   Coming back to the diagram in Figure \ref{fig:First_comparison}, the main similarity between HHSes and CAT(0) spaces is that each comes with a collection of projections to hyperbolic spaces; the collection $\calU$ for HHSes, which doesn't generally consist of lines, and the collection of lines $\calL$ for CAT(0) spaces.

   As we discussed in Section \ref{sec:from_HHS_to_curtains}, for each $U \in \calU$, the projection $\pi_U: X \rightarrow U$ can be ``broken apart" to a collection of projections to geodesic lines $l \subset U$ via nearest point projections. That is, each map $\pi_U: X \rightarrow U$ induces a (usually infinite) collection of maps $\{\pi_l\circ \pi_U: X \rightarrow l \mid l \text{ is a geodesic line in } U\}_{l \subset U}.$ Let's denote by $\calL^U$ the collection of all geodesic lines in $U,$ and in particular, $\calL^{\calC S}$ is the collection of geodesic lines where the hyperbolic space $U$ is the omniscient $U=\calC S$.

   This provides a slightly stronger analogy between HHSes and CAT(0) spaces: each comes with a collection of projections to geodesic lines; the collection of all lines $\calL$ in the CAT(0) case, and the collection $\calL^\calU:=\cup_{U \in \calU} \calL^U$ in the HHS case. The main distinction between the two situations, however, is that unlike the CAT(0) case, the collection of lines $\calL^\calU$ in the HHS case are grouped into a ``coherent" collection of hyperbolic spaces $\{U_i
   \}_{i \in I}$ satisfying some strong finiteness conditions. A natural question to ask at this point is the following, which is described in Figure \ref{fig:Try} below.

\vspace{4mm}

\noindent\underline{Question 1:} Is it possible to ``organize", or ``assemble" the lines $\calL$ coming from the CAT(0) space $X$ into a collection of hyperbolic spaces $\calU$ so that $(X, \calU)$ is an $\isom(X)$-invariant HHS? 

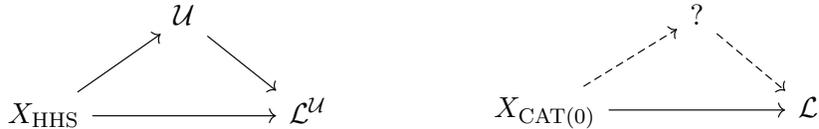
\begin{figure}[ht]
    \centering
 \begin{tikzcd}
& \calU \arrow{rd}& \\
  X_{\text{HHS}} \arrow{rr}\arrow{ur} &                         & \calL^\calU
\end{tikzcd}%
    \qquad \hspace{10mm}
 \begin{tikzcd}
& ? \arrow[rd, dashrightarrow] & \\
  X_{\text{CAT(0)}} \arrow{rr}\arrow[ur, dashrightarrow] &                         & \calL
\end{tikzcd}    \caption{Each HHS admits projection maps to geodesic lines which factor through their projections to hyperbolic spaces. Can we factor CAT(0) projections to their geodesic lines $\pi_l:X \rightarrow l$ through some hyperbolic spaces?}
    \label{fig:Try}
\end{figure}
%\begin{tikzcd}  
 %X_{\text{HHS}} \rightarrow \calU_{\text{hyperbolic spaces}} \rightarrow \calL^\calU
%\end{tikzcd}%
 %   \qquad \hspace{10mm}
%\begin{tikzcd}  
 %X_{\text{CAT(0)}} \rightarrow \calL_{\text{hyperbolic lines}}
%\end{tikzcd} 

\vspace{2mm}

 The answer to the above question is simply no, especially if the goal is to obtain a space that $\isom(X)$-invariant as discussed in the previous section. A natural follow-up question then is,  how about the omniscient? Namely:

\noindent\underline{Question 2:} Given the exceptional utility of the HHS omniscient $\calC S$, since we are unable to assemble all the CAT(0) lines in $\calL$ into a collection of hyperbolic spaces $\calU$ so that $(X,\calU)$ is an $\isom(X)$-invariant HHS, can we at least assemble \textbf{some} of the lines $\calL' \subsetneq \calL$ to obtain a single $\isom(X)$-invariant hyperbolic space that is similar to the HHS omniscient? See Figure \ref{fig:How_about_omniscient}.

 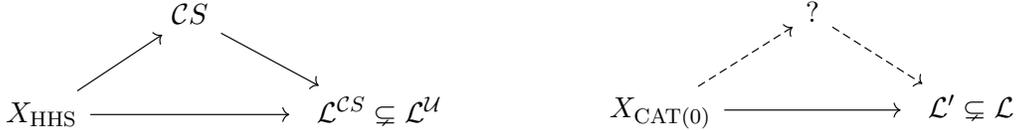
\begin{figure}[ht]
    \centering
 \begin{tikzcd}
& \calC S \arrow{rd}& \\
  X_{\text{HHS}} \arrow{rr}\arrow{ur} &                         & \,\,\,\calL^{\calC S}\subsetneq \calL^\calU
\end{tikzcd}%
    \qquad \hspace{10mm}
 \begin{tikzcd}
& ? \arrow[rd, dashrightarrow] & \\
  X_{\text{CAT(0)}} \arrow{rr}\arrow[ur, dashrightarrow] &                         & \,\,\,\calL'\subsetneq \calL
\end{tikzcd}    \caption{Can the projections to a subfamily of CAT(0) lines $\calL' \subset \calL$ factor through a hyperbolic space similar to the HHS omniscient?}
    \label{fig:How_about_omniscient}
\end{figure}

\vspace{2mm}

If we are to use a subcollection of lines from $\calL$ in a CAT(0) space $X$ to build an analogue of the omniscient, we first need to pint point that subcollection. In order to do that, it's reasonable to start with an HHS $X$, analyze the subcollection of lines $l$ that come from the omniscient $\calC S$, identify the properties they satisfy and look for lines $l \in \calL$ in the CAT(0) space $X$ with similar properties. Let  $\calL^{\calC S} \subset \calL^\calU$ denote the collection of such lines. Each such $l \subset \calC S$ comes with the projection map $\pi_l \circ \pi_{\calC S}:X \rightarrow l$, and by Lemma \ref{lem:omniscients_curtains}, after taking the median hulls, fibers of $\pi_l\circ \pi_{\calC S}$ are $L$-separated (provided they project to sufficiently far sets in $\calC S$, but that is technical) for some $L$ depending only on $X$. Hence, in order to identify lines $l \in \calL$ in a CAT(0) space $X$ that can be assembled together to obtain an analogue of the omniscient, it's wise to look for lines $l \in  \calL$ in our CAT(0) space $X$ where fibers of $\pi_l:X \rightarrow l$ are $L$-separated. This idea indeed does the trick. Namely, our approach will be to define a distance $\dist_L$ such that $\dist_L(x,y)$ is the cardinality of a maximal chain of $L$-separated curtains crossed by $[x,y]$. More precisely, define $X_L$ to be $X$ as a set with a distance 
$$\dist_L(x,y)=1+ \sup\{|c| \text{ where } c \text{ is a chain whose elements are pair-wise $L$-separated}\}.$$

Each space $(X_L,\dist_L)$ is $\delta$-hyperbolic, with $\delta$ at the order of $L^2$, and it's immediate from the definition that $\dist_L \leq \dist_{L+1} \leq 1+\dist$ for each $L$. Further, as shown in \cite{PSZCAT}, the collection $\{(X_L, \dist_L)\}_{L \geq 0}$ witnesses the entire hyperbolic aspects of the CAT(0) space $X$, for instance, every rank-one element acts loxodromically on some $X_L$. 

Finally, an infinite geodesic line $l \subset X$ has an infinite projection in $X_L$ if and only if infinitely many curtains coming from $l$ are pair-wise $L$-separated \cite{PSZCAT}, let's denote the collection of such lines by $\calL^L \subset \calL$. In particular, the space $X_L$ can be thought of as the space obtained by assembling the lines $\calL^L$. This provides a --somewhat unsatisfactory-- answer to Question 2 above, namely, the question mark and the collection $\calL'$ in Figure \ref{fig:How_about_omniscient} can be filled in as follows:

 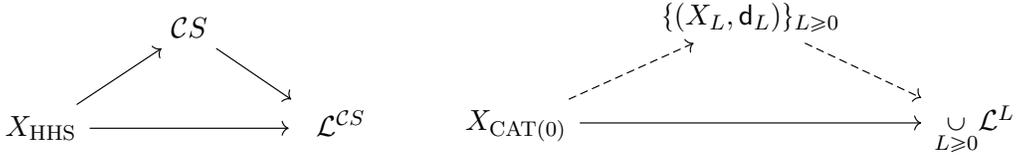
\begin{figure}[ht]
    \centering
 \begin{tikzcd}
& \calC S \arrow{rd}& \\
  X_{\text{HHS}} \arrow{rr}\arrow{ur} &                         & \,\,\,\calL^{\calC S}
\end{tikzcd}%
    \qquad \hspace{1mm}
 \begin{tikzcd}
& \{(X_L, \dist_L)\}_{L \geq 0} \arrow[rd, dashrightarrow] & \\
  X_{\text{CAT(0)}} \arrow{rr}\arrow[ur, dashrightarrow] &                         &  \underset{L \geq 0}{\cup} \calL^L
\end{tikzcd}    \caption{Filling out the previous figure}
    \label{fig:First_step}
\end{figure}

%Hence, this $X_L$ is not yet an analogue of the HHS omniscient simply because the omniscient $\calC S$ is a single space witnessing the entire hyperbolic aspects 

The reason the answer provided by Figure \ref{fig:First_step} is not fully satisfactory is simply the fact that in the HHS setting, there is a single hyperbolic space $\calC S$ witnessing the entire hyperbolic aspects of the underlying HHS. In the above context however, one needs to consider an infinite family of hyperbolic spaces in order to capture all the hyperbolic aspects of the CAT(0) space. The goal of the next subsection is to combine the hyperbolic spaces $\{(X_L, \dist_L)\}_{L \geq 0}$ into a singleton; the \emph{curtain model}.

  \begin{figure}[ht]
   \includegraphics[width=12cm, trim = 1cm 8cm 1cm 4cm]{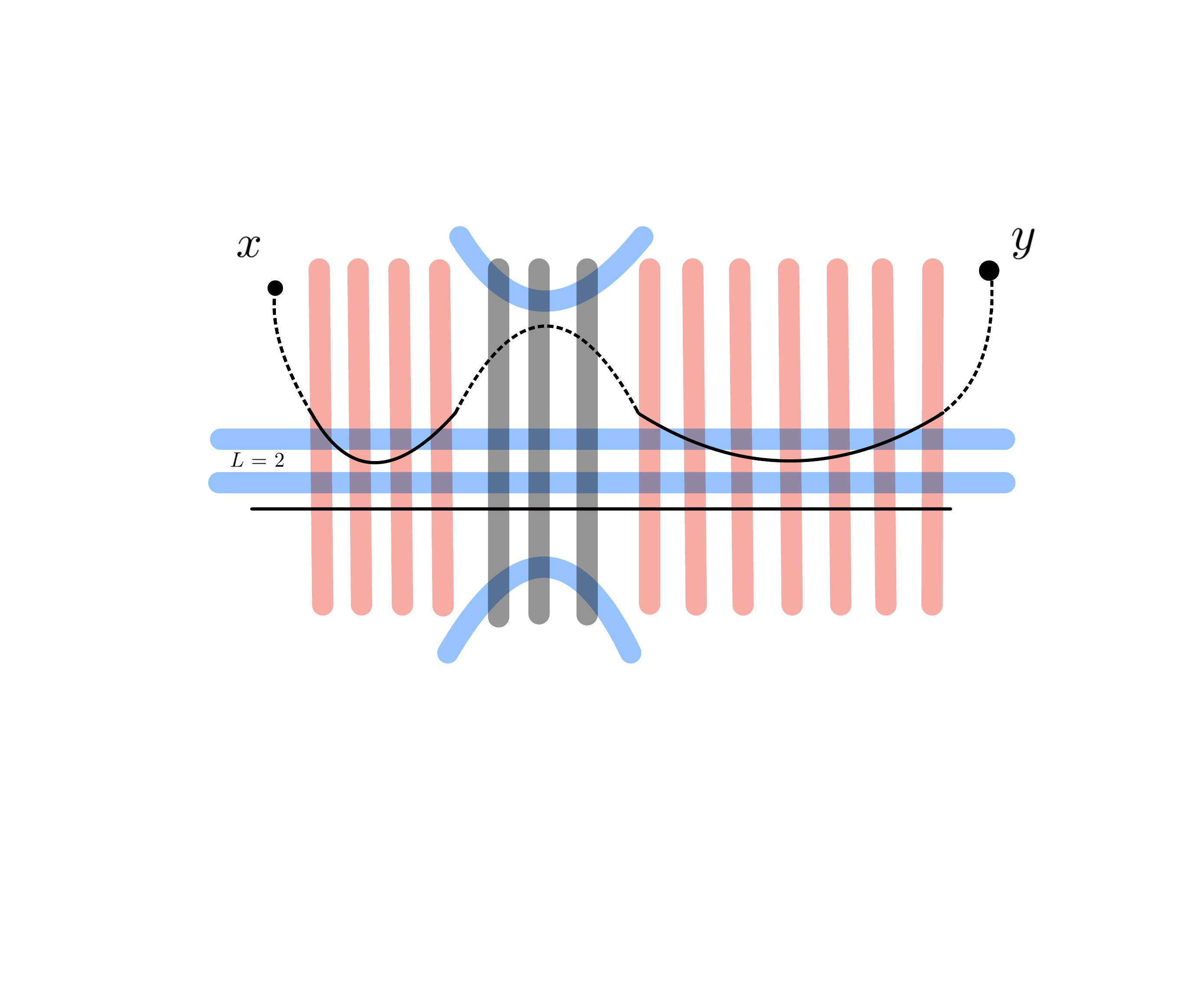}\centering
\caption{Time spent in an $L$-bottleneck} 
\label{fig:uniform_botleneck_new}
\end{figure}

\subsection{The curtain model; a counter-part of the HHS omniscient} Given the hyperbolic space $(X_L, \dist_L)$ associated to a CAT(0) space $X$, the integer $L \geq 0$ is meant to quantify the width of the bottleneck a geodesic $[x,y]$ goes through, and $\dist_L(x,y)$ measures the time spent in such a bottleneck by $[x,y]$. For instance, the distance $\dist_2(x,y)$ in Figure \ref{fig:uniform_botleneck_new} is 11 which is the time spent in 2-bottlenecks by the geodesic connecting $x,y$.

The bottleneck behaviour in a general CAT(0) space is not uniformly controlled, namely, it's easy to construct examples such that the CAT(0) space $X$ contains bottlenecks of arbitrarily large width, for instance, see Figure \ref{fig:increasing_bottleneck}. This is a reflection of the fact that the curvature of a CAT(0) space $X$ is only bounded above by zero, in particular, it's possible for $X$ to contain regions $Y_L$ whose curvatures $\kappa_L$ get arbitrarily close to zero. Surprisingly, such an arbitrary increase in the bottleneck thickness, as in Figure \ref{fig:increasing_bottleneck}, can also be found in proper cocompact CAT(0) spaces, in fact, CAT(0) cube complexes, see \cite{Shepard2022}. Namely, Shepard provided an example of a proper cocompact CAT(0) square complex $X$ with infinitely many pairs of disjoint hyperplanes $\{(h_i,k_i)\}_{i \in \mathbb{N}}$ such that $h_i,k_i$ are crossed by a chain of cardinality exactly $i$.

    \begin{figure}[ht]
   \includegraphics[width=14cm, trim = 1cm 8cm 1cm 7cm]{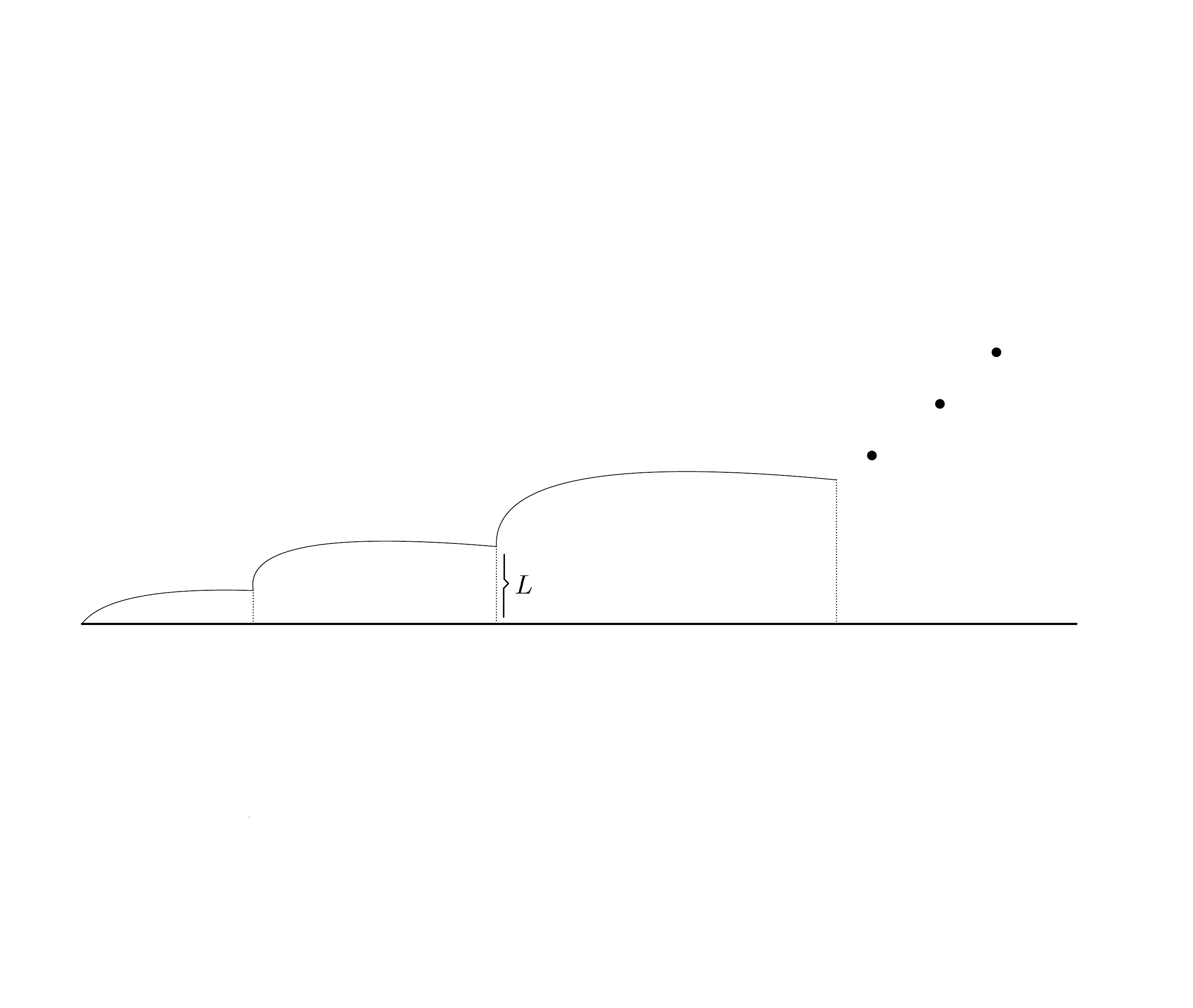}\centering
\caption{Increasing bottleneck widths $L$} \label{fig:increasing_bottleneck}
\end{figure}

In particular, it's possible for rank-one elements (which are semi-simple isometries with an axis $l$ that doesn't bound a half flat) to bound rectangular strips $l \times [0,L]$ of an arbitrarily large width $L.$ 

In light of the above, obtaining a single hyperbolic space that collectively witnesses all the bottleneck behaviours of the underlying CAT(0) space seems hopeless: if we want an analogue of the omniscient $\calC S$ witnessing all the hyperbolic aspects of $X$, on one hand, we are compelled to include arbitrarily large bottleneck widths $L$ to assure that all rank-one elements act loxodromically. On the other hand, including arbitrarily large bottlenecks $L$ disrupts hyperbolicity; namely, if $g_L$ is a rank-one element acting along a rectangular strip $[0,L] \times \mathbb{R}$, then, including the bottleneck $[0,L] \times \mathbb{R}$ in the desired hyperbolic space contradicts hyperbolicity when $L$ is too large.

A natural attempt to handle the decline in the bottleneck behaviour (i.e., increase in $L$) is to \emph{compress} each $L$-bottleneck region by a factor of $\frac{1}{L}$ or smaller. Namely, one can multiply the distance $\dist_L(x,y)$ --which recall measures the total time spent by $[x,y]$ in $L$-bottleneks-- by a quantity smaller than or equal to $\frac{1}{L}$, and the hope is that the hyperbolicity constant for these $X_L's$ will now be uniform. We will do exactly that. Namely, let's compress each $L$-bottleneck region by a factor of $\frac{1}{L^4}$ by adjusting the distance $\dist_L$ on $X_L$ to rather be $\frac{\dist_L}{L^4}$ (the choice of $\frac{1}{L^4}$ is not necessarily optimal, however, from a geometric and analytic point of view, we will see that it's necessary to multiply by $\frac{1}{L^p}$ for some $p>1$. In principle, $\frac{1}{L^{1.1}}$ could work).

Having dealt with the arbitrary increase in the hyperbolicity constant, we now need to define a new distance on $X$, denote it $\Dist$, where $\Dist(x,y)$ counts the time spent by $[x,y]$ not only in a single $L$-bottlenecks, but in \emph{any} $L$-bottleneck for any integer $L \geq 0$. In order to do that, it seems reasonable to add up the (modified) distances in $X_L,$ namely, define the distance $$\Dist(x,y)= \sum_{L \geq 1} \frac{\dist_L(x,y)}{L^4}.$$

Let's use $\X$ to denote $(X,\Dist),$ we have the following.

\begin{theorem}[{\cite{PSZCAT}}] \label{thm:Main_CAT(0)_thm} There exists some $\delta$ such that for any CAT(0) space $X$, we have:

\begin{enumerate}
    \item The space $\X$ is $\delta$-hyperbolic,
    \item Each CAT(0) geodesic in $X$ is an unparameterized $(1, \delta)$-quasi-geodesic in $\X,$

    \item $\isom(X)<\isom(\X)$, 

    \item A semi-simple isometry of $X$ is rank-one if and only if it is loxodromic on $\X,$ and 

    \item If $G<\isom(X)$ acts properly on $X$, then all rank-one elements of $G$ act as W.P.D isometries on $\X.$
    
\end{enumerate}

\end{theorem}

The space $\X$ is called the curtain model for $X,$ and this hyperbolic space is indeed an analogue of the HHS omniscient.

The claim that the curtain model is ``similar" to the HHS omniscient needs to be justified. Namely, what does it mean for a hyperbolic space $X_{\Dist}$ to be ``similar" to the HHS omniscient, does any $\isom(X)$-invariant hyperbolic space qualify? Well, of course not. 

The quality of being ``similar to the HHS omniscient" should be judged by comparing properties held by $X_{\Dist}$ against those held by the omniscient. Namely, as we discussed in Subsection \ref{subsec:omniscient}, we know that what makes the omniscient the omniscient is the list of strong properties it satisfies making it witness all hyperbolic aspects of the underlying HHS.

Witnessing all the hyperbolic aspects of a space is rather a philosophical statement, its two biggest indications however are containment of the entire \emph{Morse boundary} --which recall is a space that records the hyperbolic directions of a proper geodesic metric space $X$-- and having infinite orbits for all the Morse elements of $\isom(X)$. 

All of the statements and constructions described in items 1-5 of Subsection \ref{subsec:omniscient} utilizing the HHS omniscient are also enjoyed by the curtain model $\X$ of a CAT(0) space $X$. Namely, the curtain model along with its Gromov boundary contain the entire Morse boundary of $X$, characterize Morse elements and stable subgroups, records the isometry group of $X$ via an Ivanov-style rigidity theorem and finally, it can be used to establish both the Morse local-to-global property for $X$ and acylindrical hyperbolicity for a group $G$ acting geometrically on $X,$ see \cite{PSZCAT}. It is worth noting however that the action on the curtain model is expected to not be acylindrical in general (although it can be used to establish acylindrical hyperbolicity via WPD elements). The example provided in \cite{Shepard2022} likely not act acylindrically on it's curtain model.

\subsection{What does it mean to make distance in the curtain model?} There are two closely related aspects for understanding the geometric meaning of the distance recorded in the curtain model, the first of which is the following.

\vspace{2mm}

     \noindent 1) \underline{Time spent near thinner bottlenecks is more worthy}. Based on the formula for the distance in the curtain model $$\Dist(x,y)= \dist_1(x,y)+ \frac{\dist_2(x,y)}{2^4}+\frac{\dist_3(x,y)}{3^4}+ \cdots,$$ one can see that spending time $t$ near thinner bottlenecks contributes much more to $\Dist(x,y)$ than spending the same amount of time $t$ near wider bottlenecks. Namely, if a geodesic $[x,y]$ crosses $t$ curtains that are 1-separated (i.e., it spends at least $t$-time in $1$-bottlenecks), then $\Dist(x,y)$ is at least $t$ in $\X.$ However, if $[x,y]$ spends $t$-time in say $5$-bottlenecks, then the amount these $5$-bottlenecks contribute to the total $\Dist(x,y)$ is only $\frac{t}{5^4}$. In short, for a fixed $L \geq 1$, in order for the $L$-bottlenecks to contribute distance 1 to $\Dist(x,y)$, the geodesic $[x,y]$ must spend at least $1.L^4$-time in these $L$-bottlenecks.

    \vspace{2mm}

\noindent 2) \underline{Hyperbolicity decay versus time spent in bottlenecks}: The decay in the hyperbolicity (i.e., the increase in the bottleneck width $L$) must be made up for by spending a very long time --at least $L^4$-- in such $L$-bottlenecks, or otherwise, that total time is virtually a pause in $\X$. Namely, the bottleneck behaviour and time spent in these bottlenecks must balance each other out; in order to cancel out the increase of the bottleneck width, the time spent in that new bigger bottleneck must be very substantial. Yet another way of phrasing this, the increase in bottleneck width must be a \emph{sublinear} function in the time spent inside bottlenecks of that fixed width.

For instance, the geodesic ray in the bottom of Figure \ref{fig:not_making_distance_in_omniscient_CAT} projects to a finite set in $\X.$ This is because, although the geodesic ray goes through infinitely many bottlenecks, the time it spends in each bottleneck $L$ is about $L$ which is not substantial enough to make up for the decay in bottlenecks (also, intuitively, we really do not want such a geodesic ray to have an infinite projection in the curtain model, as there is nothing hyperbolic about it). On the other hand, the geodesic ray at the bottom of Figure \ref{fig:making_distance_in_omniscient_CAT} projects to an infinite quasi-ray in the curtain model.

   \begin{figure}[ht]
   \includegraphics[width=12cm, trim = 3cm 8cm 1cm 6cm]{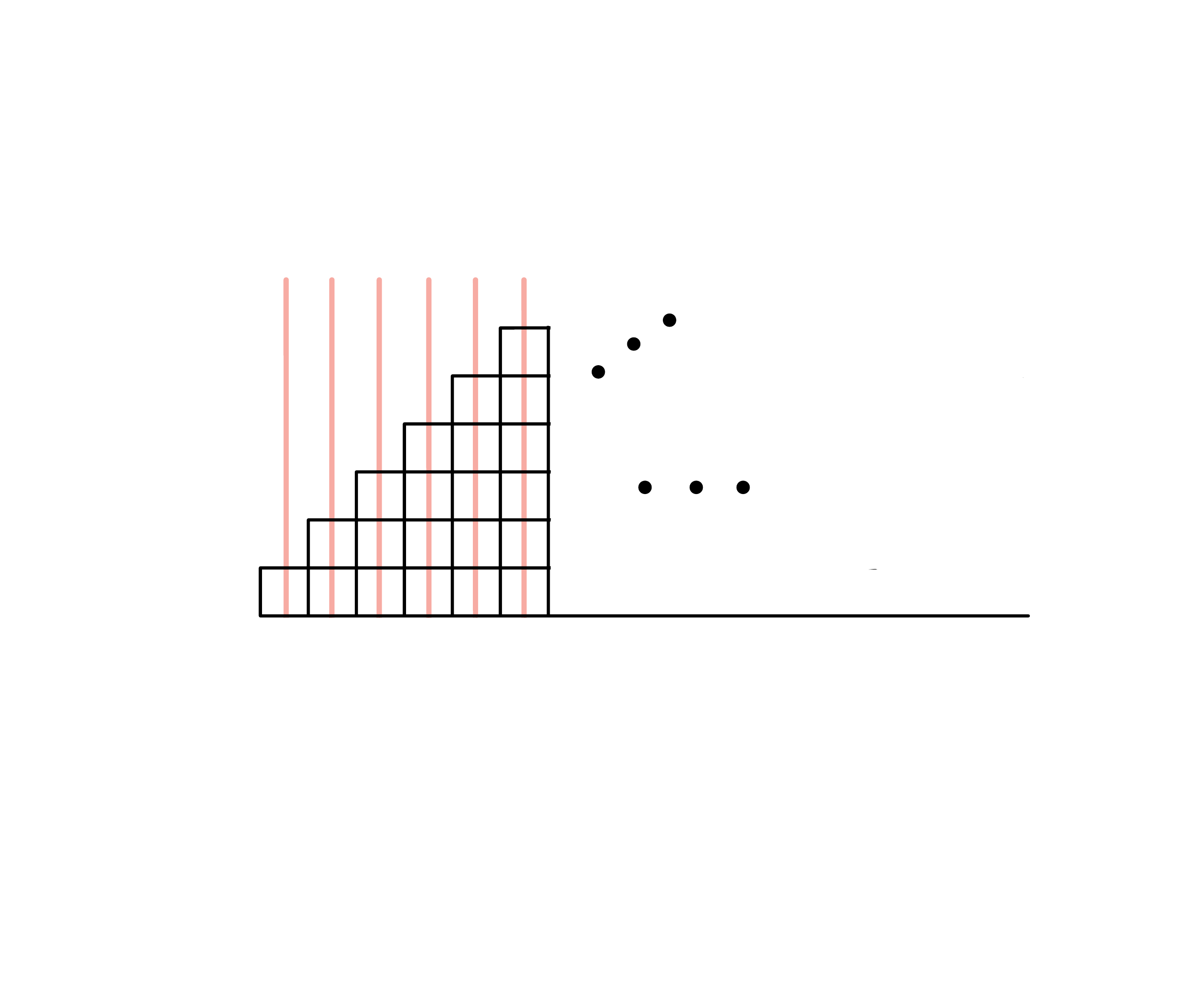}\centering
\caption{The decay of hyperbolicity is witnessed by the increase in the bottleneck constant $L$ which takes every value in $\mathbb{N}$ here. The bottom geodesic ray crosses all the red curtains, but it doesn't spend a substantial time $(\geq L^4)$ in any $L$-bottleneck region, hence, it dies out in the curtain model.} \label{fig:not_making_distance_in_omniscient_CAT}
\end{figure}

    \begin{figure}[ht]
   \includegraphics[width=14cm, trim = 3cm 7cm 1cm 5cm]{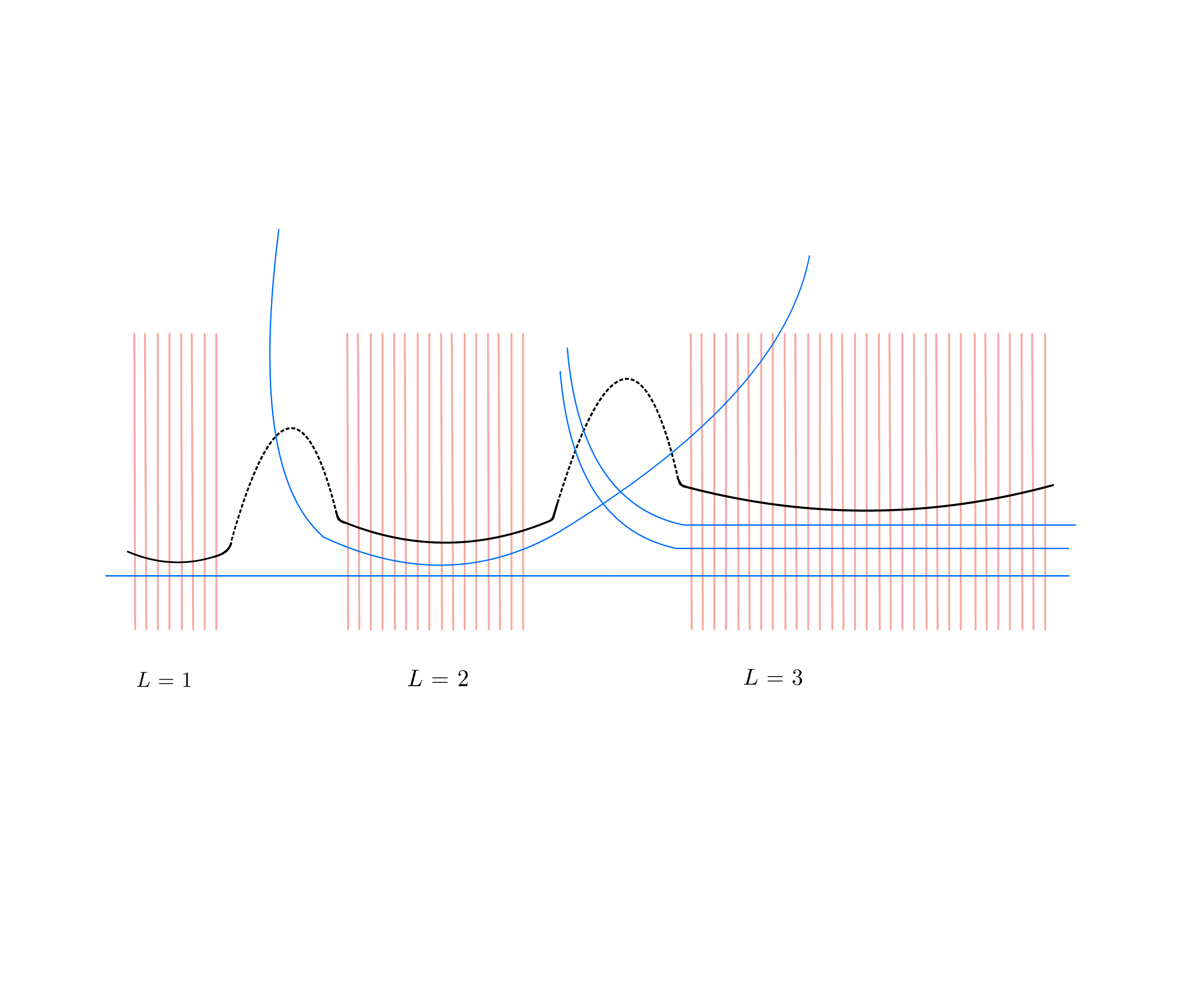}\centering
\caption{The decay in hyperbolicity is witnessed by an increase in the bottleneck constant $L$. However, such a decay is made up for by spending an incredibly long time $(\text{at least } L^4)$ in such bottlenecks leading the black geodesic ray to survive in the curtain model defining a point in its Gromov boundary.} \label{fig:making_distance_in_omniscient_CAT}
\end{figure}

\subsection{A comparison with sublinear Morseness}
The sublinear increase in the bottleneck behaviour of example in Figure \ref{fig:making_distance_in_omniscient_CAT} suggests a relationship between $\X$ and the \emph{sublinearly Morse boundary} $\partial_\kappa X$ of Qing-Rafi-Tiozzo \cite{QRT19}, \cite{QRT20}, and in fact, such a relationship is present. Recall that a quasi-geodesic ray $\gamma$ is said to be $\kappa$-sublinearly Morse if there is a sublinear function $\kappa$, a map $m:\mathbb{R}^+\times \mathbb{R}^+ \rightarrow \mathbb{R}^+$ such that every  $(q,Q)$-quasi-geodesic $\beta:[0,A] \rightarrow X$ with end points on $\gamma$ satisfies $\dist(\beta(t), \gamma) \leq m(q,Q) \kappa(d(\gamma(0), \beta(t)))$ for any $t \in [0,A].$ The $\kappa$-boundary $\partial_\kappa X$ is the collection of all $\kappa$-Morse quasi-geodesic rays with an appropriate equivalence relation and topology.

Forthcoming work of Vest \cite{Vest:curtain} performs thorough analysis regarding the interaction between the $\kappa$-boundary of a CAT(0) space $\partial_\kappa X$ and the Gromov boundary of its curtain model $\partial \X.$ In particular, Vest shows that for most sublinear functions $\kappa$, sublinearly Morse geodesic rays project to infinite diameter quasi-geodesic rays in $\X$.

The discussion in the previous subsection enhanced by Figures \ref{fig:not_making_distance_in_omniscient_CAT} and \ref{fig:making_distance_in_omniscient_CAT} seems to suggest that the only geodesic rays of $X$ that project to infinite sets in the curtain model are the sublinearly Morse ones, but this is incorrect. For instance, in the universal cover $\tilde{X}$ of a torus wedge a circle $X=T \vee S^{1}$, a geodesic ray $b$ is sublinearly Morse with sublinear function $\kappa$ if and only if there is a constant $C$ such that the total amount of time spent by $[b(0), b(t)]$ in any flat $\mathbb{R}^2$ is bounded above by $C\kappa(t)$. However, every geodesic ray that visits infinitely many distinct $\mathbb{R}^2$ --independently of the time it spends in a given $\mathbb{R}^2$-- defines an infinite quasi-geodesic ray in $\X$. This is because, although the time spent in flats is uncontrolled, the bottleneck thickness here is uniformly controlled (it's zero).

In short, while it's true that a sublinear bound on the growth of the bottlenecks crossed by a ray $b$ is necessary if $b$ is to make infinite distance in the curtain model $\X$, the time spent ``in between" these bottlenecks need not be controlled; unlike $\kappa$-sublinearly Morse geodesic rays where such a time must also be controlled.

%For simplicity, suppose that $X$ is a CAT(0) space which is also a cube complex; as we discussed, even when $X$ is proper and cocompact, the bottleneck behaviour in $X$ can decay by an arbitrarily large amount \cite{Shepard2022}. 

\section{Outroduction}\label{sec:out} Throughout this survey, we have discussed four competing notions of non-positive curvature, each of which comes with a wealth of combinatorial objects and tool boxes that describe their geometry to a great extent:

\vspace{2mm}

\noindent 1) \underline{CAT(0) cube complexes} whose geometry is fully recorded via their hyperplanes $\calH=\{h\}$ and the way they interact with one another.

    \vspace{2mm}

\noindent 2) \underline{Injective metric spaces} that come with their universal property making their geometry exceptionally accommodating. Given a subspace $Y \subset X,$ if one wishes to associate a point, a quantity or a property to $Y$, the first step would typically be to find some $Y$-overspace $Z$ (meaning $Y \subset Z$) where the desired point, quantity or property is present, and the injective universal property kicks in allowing us to push it down to $X.$

\vspace{2mm}
    \noindent 3) \underline{Hierarchically hyperbolic spaces and groups} whose coarse geometry is entirely captured via a collection of hyperbolic spaces $\calU=\{U_i\}_{i \in I}$ satisfying certain consistency conditions. If we are to assign weight to the importance of these hyperbolic spaces $\calU=\{U_i\}_{i \in I}$, then there is a clear winner, which is the omniscient $\calC S.$ It's exceptional utility is justified by the amount of information it alone captures about the underlying HHS as we thoroughly described in Subsection \ref{subsec:omniscient}. There are also three more modern machines providing a wealth of extra tools for understanding hierarchically hyperbolic spaces and groups, they are summarized below:
    
\vspace{2mm}

    \begin{itemize}
        \item Each HHG admits a proper cobounded action on an injective metric space, in particular, all the tools available for studying injective metric spaces are applicable to HHGs, thanks to \cite{HHP}.

        \item Median hulls of finitely many points -- possibly including ``boundary" points -- in an HHS look like a CAT(0) cube complexes allowing a great deal of cubical tools to be carried over to the HHS setting. This shows that HHSes generalize trees in two orthogonal ways. First, they generalize hyperbolic spaces and the latter generalize trees, second, the cubical approximation theorem states that HHSes generalize CAT(0) cube complexes and the latter also generalize trees. A useful way of thinking about the cubical approximation theorem is perhaps the following. Although the global geometry of an HHS can be very different from that of cube complex, if we fix a finite number of points $F$ in an HHS and consider the various ways of efficiently moving from one of these points to another, i.e., $\hull(F)$, then the latter looks very much like a CAT(0) cube complex $Q_F$. Further, this ``looking-like" a CAT(0) cube complex for $\hull(F)$ is equivariant, namely, each $g \in \Aut(X)$ induces two median preserving isometries $g:Q_F \rightarrow Q_{gF}$ and $g:\hull(F) \rightarrow \hull(gF)$ so that the approximating maps $f_F:Q_F \rightarrow \hull(F)$, $f_{gF}:Q_{gF} \rightarrow \hull(gF)$ satisfy $g \circ f_F=f_{gF} \circ g$.

        \item There is also the hyperplane or curtain perspective on HHSes allowing for even more cubical tools to transfer to HHSes. These curtains are simply median convex sets $h$ cutting the HHS $X$ into two median convex ``half spaces" $h^+,h^-$. Their utility can be seen from the fact that they provide purely combinatorial descriptions of many important geometric notions in the underlying HHS, this includes a characterization of the HHS distance as a cardinality of a maximal chain of curtains separating them, a description of the coarse median and median-convex set as intersections of appropriate collection of half spaces, see Lemma \ref{lem:sample_body}.

    \end{itemize}
\vspace{2mm}

    %These are spaces which generalize trees in two ``orthogonal" ways: they extend hyperbolic spaces which on their side 

    \noindent 4) \underline{CAT(0) spaces} which come with two combinatorial objects describing many interesting aspects of their geometry; curtains which are analogues of cubical hyperplanes, and the curtain model which is a counter-part of the HHS omniscient. Given a CAT(0) space $X$, these two combinatorial objects record numerous interesting aspects of its geometry as shown in \cite{PSZCAT}.

    \section{Questions}\label{sec:questions}

If one is to follow the route of Caprace and Sageev \cite{capracesageev:rank} in their proof of rank-rigidity, the questions in items 1-3 below are important:

\begin{enumerate}

\item (Skewering) Let $X$ be a cocompact CAT(0) space with the geodesic extension property and let $h^+,k^+$ be half spaces with $k^+ \subsetneq h^+$. Does there exist an element $g \in \isom(X)$ with $g^nh^+ \subsetneq k^+$ for some $n \in \mathbb{N}$?

\item (Flipping) Let $X$ be a CAT(0) space with the geodesic extension property and let $h^+,k^-$ be disjoint half spaces. Suppose that $G < \isom(X)$ acts cocompactly on $X$ with no fixed point at infinity, does there exist $g \in G$ (or even in $\isom(X)$) with $g^nk^-\subsetneq h^+$ for some $n \in \mathbb{N}$?

\item (Diameter of $X_L$) Let $X$ be a CAT(0) space with the geodesic extension property and let $h,k$ be a pair of $L$-separated curtains for some $L \geq 0.$ If $G$ acts on $X$ co-compactly, does $G$ contain a rank-one element? Proposition 6.1 of \cite{PSZCAT} shows that the presence of a chain $\{h_1,h_2,h_3,h_4 \}$ whose elements are pair-wise $L$-separated implies the existence of a rank-one element for any group $G$ with a cobounded action on $X.$ The previous question asks if the number 4 can be reduced to 2.

\item Let $G$ be a group that acts properly cocompactly on $X$. Suppose that $X$ is any of the following:

\begin{itemize}
    \item a CAT(0) space, 

    \item a CAT(0) cube complex, or
    \item an injective metric space.
    \end{itemize}
    
\noindent If every $g \in G$ is an infinite order Morse element, is $G$ hyperbolic? A natural approach in the context of CAT(0) spaces is to consider the action of $G$ on the curtain model $\X$. Namely, we know that $G$ acts co-boundedly on $\X$ and since every element $g \in G$ is Morse, $g$ must be loxodramic on $\X;$ does this suffice to deduce hyperbolicity of $G$?

\item Does every CAT(0) space $(X, \dist)$ admit a distance $\rho$ so that $\isom (X, \dist)$ acts on $(X, \rho)$ by isometries and the identity map $i:(X,\dist) \rightarrow (X, \rho)$ is a quasi-isometry?

\item Let $X$ be a CAT(0) or an injective metric space with a proper cocompact action of $G$ on $X.$ Suppose further that $G$ contains a Morse element. Does $G$ contain a Morse element whose length is bounded above independently of the chosen finite generating set for $G$?  Such a statement will imply that rank-one CAT(0) groups have \emph{uniform exponential growth}; for the definition and a sample of recent results on uniform exponential growth, see \cite{Fuj08}, \cite{Fujiwara2009SubgroupsGB}, \cite{MangahasRecipie}, \cite{Abbot-Ng-Spriano}, and \cite{Kerr}.

\item For a CAT(0) space $X$, under which natural conditions do the spaces $X_L$ \emph{terminate}? Here, by terminate we mean that there exists an $L_0$ such that $X_L$ is quasi-isometric to $X_{L+1}$ for all $L \geq L_0.$

\item For a CAT(0) space $X$, the construction of the hyperbolic spaces $X_L$ is heavily inspired by an analogous construction of Genevois in the context of CAT(0) cube complexes, and the definition of $X_L$ in that setting is essentially the same, but with hyperplanes replacing curtains \cite{genevois:hyperbolicities}. A natural question to ask is: For a CAT(0) cube complex $X,$ how does the space $X_L$ constructed via hyperplanes relate to the space $X_L'$ constructed via curtains?

\item Let $G$ be a group acting properly cocompactly by cubical isometries on an irreducible CAT(0) cube complex $X$. Assume that there exists an $L$ where each pair of disjoint cubical hyperplanes $h,k$ are either $L$-separated or crossed by infinitely many hyperplanes. Does $X$ admit a factor system? Work of Genevois \cite{genevois:hyperbolicities} shows that the presence of a factor system for a CAT(0) cube complex $X$ implies that there is such a constant $L_0;$ the question above investigates the converse of Genevois statement.

    \item Let $G$ be a mapping class group of a finite type surface. Does $G$ admit a finite generating set $A$ so that $\cay(G,A)$ is coarsely injective?

\item Let $G$ be a Morse local-to-global group containing a Morse element $g.$ If $G$ is not virtually cyclic, does it follow that $G$ is acylindrically hyperbolic?

\item Possibly using curtains in HHSes, can one apply Sageev's theorem to show that every HHS is quasi-isometric to a CAT(0) cube complex?

 \item Let $G$ be a finitely generated group and $X$ be a fixed Cayley graph corresponding to some finite generating set for $G$. If $g \in G$ is Morse and $p \in E(X)$, by Theorem A and B in \cite{Sisto-Zalloum-22}, the quasi-line $\langle g \rangle . p$ is generally not Morse (it is Morse if and only if $g$ has a strongly contracting quasi-axis in $X$). Does the quasi-line $\langle g \rangle . p$ ``retain" any hyperbolicity in $E(X)$?

\end{enumerate}

\appendix

\section*{Appendix: a few proofs}\label{appendix}

In this section, we provide the proof of Lemma \ref{lem:intro_lemma}. For the convenience of the reader, we restate the lemma.

\begin{lemma}[Petyt-Spriano-Zalloum]\label{lem:sample_append} Let $X$ be an HHS. Up to increasing the HHS constant $E$ by a uniform amount, the following hold:

\begin{enumerate}

    \item $d(x,y)$ coarsely coincides with the cardinality of a maximal chain of curtains separating $x,y.$

    \item A quasi-geodesic $\alpha$ is a median path if and only if it never crosses the same curtain twice.

    \item For any three points $x,y,z \in X$, their coarse median $\mu(x,y,z)$ coarsely agrees with the intersection of the $E$-neighborhoods of all half spaces that contain the majority of $x,y,z.$

    \item Let $Y$ be a median convex set in $X$ and let $x \in X$. The gate $P_Y(x)$ is coarsely the unique point in $Y$ with the property that whenever $c$ is a chain of curtains separating $x,P_Y(x),$ then all but $E$-many curtains of $c$ separate $x,Y.$

    \item For any set $A$, the median hull of $A$ is coarsely the intersection of the $E$-neighborhoods of all half spaces properly containing $A.$
\end{enumerate}
    
\end{lemma}

Serving as an evidence to the canonical nature of HHS curtains, the proofs are word by word coarsening of the respective statements in the context of CAT(0) cube complexes. Recall that an HHS curtain is defined to be an $E$-median-convex set $h$ such that $X-h \subseteq h^+ \sqcup h^- $ for some disjoint $E$-median-convex sets $h^+,h^-$ with $\dist(h^+,h^-)>10E.$

\emph{Proof}:
   (1) The first item of Lemma \ref{lem:intro_lemma} is an immediate reflection of the distance formula combined with the Behrstock's inequality. Namely, recall that $\text{Rel}_\theta(x,y)$ denotes the collection of all hyperbolic spaces $U$ with  $\dist(\pi_U(x), \pi_U(y))> \theta.$ By Proposition \cite{RST18}, if $\theta \geq 100E,$ then the set $\text{Rel}_\theta(x,y)$ can be partitioned into $n$ subsets $\{\calU_i\}_{i=1}^n$ with $n$ depending only on $E$ such that each $\calU_i$ consists only of transverse domains. Let $\calU$ be the element in $\{\calU_i\}_{i=1}^n$ so that $\sum_{V \in \calU} \dist_V(x,y)$ is largest among all $\{\calU_i\}_{i=1}^n$. Elements $V_i$ of $\calU$ can be totally ordered by the Behrstock's inequality as $V_1<V_2<\cdots V_{m-1}<V_m.$  Now, let $\alpha_{V_i}$ denote a geodesic connecting $\pi_{V_i}(x)$ to $\pi_{V_i}(y)$ and let $I_i \subset \alpha_{V_i}$ be an interval of length $10E.$ Now, the set $h_{I_i}=\hull( \pi_{V_i} \circ \pi_\alpha ^{-1}(I_{i}))$ is median convex for each $V_i$ and each $I_{i} \subset V_i.$ In particular, each $\alpha_i$ defines a pairwise disjoint collection (i.e., a chain) of curtains $h_{I_i^1}, \cdots h_{I_i^k}$ by choosing some intervals $I_i^j$ that avoids the $E$-neighborhood of $\alpha_i$'s end points and satisfy $\dist (I_i^j, I_i^{j+1})=4E$. Also, $k$ here is coarsely $\dist (\pi_{V_i}(x), \pi_{V_i}(y))$. Let  $c_i=\{h_{I_i^1},\cdots h_{I_i^k}\}$ denote the chain in $X$ defined via $\alpha_i \subset V_i,$ it's immediate by the Behrstock's inequality that $\cup_{i=1}^m c_i$ is a chain that separate $x,y.$ Namely, the total order $V_1<V_2 \cdots V_m$ along with the fact that each $I_i^j$ was chosen to avoid the end points on $\alpha_i$ assures that the curtains defined via $\alpha_i \subset V_i$ are all disjoint from the ones defined via $\alpha_j \subset V_j$ with $i \neq j.$

   \vspace{2mm}
(2) We say that a path $\alpha$ \emph{crosses} $h$ if it meets both $h^+,h^-$. Hence, crossing a curtain $h$ twice means that $\alpha$ contains points $\{\alpha(t_i)\}_{i=1}^3$ with $t_1<t_2<t_3$ and $\alpha(t_1), \alpha (t_3) \in h^+$ and $\alpha(t_2) \in h^-.$ Suppose that $\alpha$ is a median path, recall that this means $\gamma$ is an $E$-median path for the HHS constant $E$. Suppose for the sake of contradiction that $\alpha$ crosses $h^+,h^-,h^+$ at points $\alpha(t_1), \alpha(t_2), \alpha(t_3)$ in that order, then, median convexity of $h^+$ assures that $\alpha(t_2)$ is in the $E$-neighborhood of $h^+$, on the other hand, since $\alpha$ is $E$-median, $\mu(\alpha(t_1), \alpha(t_2), \alpha(t_3))$ is within $E$ of $\alpha(t_2) \in h^-.$ This contradicts our definition that $\dist (h^-, h^+)>10E.$ Conversely, suppose $\alpha$ never crosses a curtain twice and yet $\mu=\mu(\alpha(t_1), \alpha(t_2), \alpha(t_3))$ is far from $\alpha(t_2)$. By item 1, there is a chain of curtains $c=\{h_i\}$ separating $\mu$ from $\alpha(t_2).$ After replacing $c$ by a sub-chain $c' \subseteq c$ of cardinality $ |c'|\geq |c|-2$, we may assume that no element of $c'$ contains $\alpha(t_1)$ or $\alpha(t_2).$ Let $h$ be the curtain in $c'$ closest to $\alpha(t_2)$ and suppose that $h^-$ is the half space containing $\alpha(t_2)$. We claim that $\alpha$ doesn't cross $h$, suppose it does. Since $\alpha$ is assumed not to cross any curtain twice, at least one of $\alpha(t_1), \alpha(t_3)$, say $\alpha(t_3),$ must lie in $h^-$. Since $h^-$ is median convex, $\mu=\mu(\alpha(t_1), \alpha(t_2), \mu(t_3))$ lies in the $E$-neighborhood of $h^-$, which contradicts the fact that $\mu \in h^+$ and $\dist(h^+, h^-)>10E.$

\vspace{2mm}
(3) For item 3, let $\calH$ be the collection of all curtains and $\calH^+_{x,y,z}$ be the collection of all half spaces that contain at least 2 points among $\{x,y,z\}$. It is worth noting that since curtains are thick, $\calH$ doesn't need to contain a half space for \emph{every} curtain $h \in \calH$, namely, the curtains that define $\calH^+_{x,y,z}$ aren't necessarily the collection of all curtains, it could be a smaller sub-collection. Since each $h^+ \in \calH^+_{x,y,z}$ contains at least two of three points, the coarse Helly property applied to $\calH^+_{x,y,z} \cup \{\hull(x,y,z)\}$ implies there is a point $w \in \bigcap_{h^+ \in \calH} N(h^+,E) \cap \hull(x,y,z) $ and in particular $ w \in \bigcap_{h^+ \in \calH} N(h^+,E).$ The reason for considering the set $\hull(x,y,z)$ is that the coarse Helly property requires the existence of a bounded set among the collection under consideration. Notice that $\mu(x,y,z)$ also lies in $\bigcap_{h^+ \in \calH} N(h^+,E)$ using by median convexity of $h^+$ and the fact that $h^+ \in \calH^+_{x,y,z}$. Now, we show that $\diam (\bigcap_{h^+ \in \calH} N(h^+,E))$ is uniformly bounded depending only on $E$. Suppose that $\dist(w,w')$ is large for some $w,w' \in \bigcap_{h^+ \in \calH} N(h^+,E)$, this provides a chain $c=\{h_1,h_2,\cdots h_n\}$ given in this order with $\dist(w,w') \asymp_E n$. Without loss of generality, say that $h_1$ is the closest curtain in $c$ to $w.$ Since $w \in h_1^+,$ the majority of the points $x,y,z$ must lie in $h \cup h^+$ as if they lie in $h^-$, then $w$ would be in $h^-$ by our choice of $w.$ In particular, $h_i^+$ must contain the majority of $x,y,z$, say $x,y,$ for all $1 <i \leq n.$ A large $n$ would contradict the fact that $w' \in \bigcap_{h^+ \in \calH} N(h^+,E)$ concluding the proof.

   \vspace{2mm}

   Recall that for a median convex set $Y$, we have the gate map $P_Y:X \rightarrow Y$ where $P_Y(x)$ is characterized to be coarsely the unique point in $Y$ with $\dist(\mu(x,P_Y(x), y), P_Y(x))<E'$ for all $y \in Y,$ and for some $E'$ depending only on the HHS constant $E,$ see Lemma 2.57 in \cite{Durham-Zalloum22}. When viewed from curtains perspective, this lemma yields an immediate proof of item (4).
\vspace{2mm}

(4) Let $c$ be a chain separating $x, P_Y(x)$ and let $c'=\{h_1,\cdots h_n\} \subset c$ be a subchain all of whose elements $h_i$ meet $Y$. Suppose that $h_1,h_2\cdots h_n$ are ordered based on their proximity to $x,$ in particular, $h_1$ is the curtain in $c'$ closest to $x$. For each $1 \leq i \leq n $, assume that $h_i^+,h_i^-$ are the half spaces containing $x, P_Y(x)$ respectively, by assumption, $h_1$ meets $Y$, and hence there is a point $z \in Y \cap h_1$. Since $x, z \in h_2^+$, median convexity of $h_2^+$ assures that $\mu(x,z,w)$ is coarsely in $h_2^+$ for any $w \in X.$ However, by Lemma 2.57 in \cite{Durham-Zalloum22}, we have $\dist(\mu(x,z,P_Y(x)), P_Y(x))<E'$, a large $n$ would violate the fact that elements of $c'$ separate $x, P_Y(x)$ concluding the proof.

\vspace{2mm}

   (5) For a subset $A$, if $\calH$ is the collection of half spaces properly containing $A$, then $H=\bigcap_{h^+ \in \calH} N(h^+,E)$ is a median convex set that contains $A.$ Since $\hull(A)$ is defined to be the intersection of the $E$-neighborhoods of all the median convex sets containing $A,$ the set $N(H,E)$ contains $\hull(A).$ Now we need to show that it can't be bigger, namely, for $x \in N(H,E)$, we wish to uniformly bound $\dist(x,P_{\hull(A)}(x))$ or equivalently, if $c=\{h_1,\cdots h_k\}$ is a chain separating $x,P_{\hull(A)}(x)$, we wish to uniformly bound $k$. By item 4, all but uniformly many curtains $h_i$ separate $x, \hull(A),$ hence, after disregarding a uniform number of curtains from $c,$ we may assume that $\hull(A) \subset h_i^+$ for each $i.$ Now, by definition of $H,$ the point $x$ must be in the $E$-neighborhood of each $h_i^+$ which uniformly bounds the number $k$ depending only on $E.$

   \vspace{5mm}

We now re-state and prove Lemmas \ref{lem:points_versus_their_gates}, \ref{lem:separated_iff_gate} and Lemma \ref{lem:omniscients_curtains_append}.

\begin{lemma} [Petyt-Spriano-Zalloum]\label{lem:two_points_versus_their_gates_append} Let $h,k$ be two disjoint curtains and let $x,y \in h$. If $c$ is a chain of curtains separating $P_{k}(x), P_{k}(y)$, then all but at most $E'$-many curtains of $c$ separate $x,y$ where $E'$ is a constant depending only on $E.$
    
\end{lemma}

\begin{proof} Let $c$ be be a chain of curtains separating $P_{k}(x), P_{k}(y)$ and let $c'=\{h_1,\cdots h_n\} \subset c$ be the maximal subollection of $c$ that doesn't separate $x,y.$ Item 4 of Lemma \ref{lem:sample_append} gives us that $|c'| \leq E'$ as otherwise we have a chain with more than $E'$ curtains separating $x, P_{k}(x)$ (or $y, P_{k}(y)$) which doesn't separate $x$ (or $y$) from $k.$
\end{proof}

In the context of CAT(0) cube complexes, the following lemma is exactly Proposition 14 in \cite{Genevois16}, compare this with Lemma 7.12 in \cite{Durham-Zalloum22}.

\begin{lemma}[Petyt-Spriano-Zalloum]\label{lem:separated_iff_gate_append} Let $X$ be an HHS. Two disjoint curtains $h,k \subset X$ are $L$-separated if and only if $\diam (P_{k}(h)) \leq L'$, where $L,L'$ determine each other.
\end{lemma}

\begin{proof} 
The proof of this lemma is a line-by-line coarsening of Proposition 14 in \cite{Genevois16}, for completeness, we provide a proof. If $h,k$ are $L$-separated, then, by the previous lemma, for any $x,y \in h $, if $\{h_1, \cdots h_n\}$ is a chain separating $P_k(x), P_k(y)$, then $\{h_1,\cdots h_n\}$ separate $x,y$, in particular, they all cross $k$ but this number is bounded above by $L$ by the $L$-separation assumption. Conversely, if $c$ is a chain of curtains crossing both $h,k,$ then we claim that all but uniformly finitely many curtains from $c$ must meet $P_k(h)$ from which the statement follows. Suppose that $c'=\{h_1, \cdots h_n\} \subset c$ is a chain neither of whose elements meet $P_k(h)$, and suppose that $h_1$ is the furthest curtain from $P_k(h)$ among curtains in $c'.$ Let $x \in h_1 \cap h$, notice that the curtains $\{h_2, \cdots h_n\}$ separate $x$ from the set $P_k(x) \subset P_k(h)$ but meet $k$, hence, by item 4 of Lemma \ref{lem:sample_append}, we have $|c'|-1 \leq E'$ concluding the proof.
\end{proof}

Recall that a special type of curtains we considered were ones of the form $$h_{U,b,I}:=\hull(\pi_U^{-1}\pi_b^{-1}(I)).$$

\begin{lemma}[Petyt-Spriano-Zalloum]\label{lem:omniscients_curtains_append}
 For $i \in \{1,2\}$, let $U=\calC S$, $b_i \subset U$ be a geodesic, $I_i\subset b_i$ be an interval and let $h_i:=h_{U,b_i, I_i}\subset X$. There are constants $E', L$, depending only on $E,$ such that if $\dist(\pi_U(h_1),\pi_U(h_2))>E',$ then $h_1,h_2$ are $L$-separated.
\end{lemma}

\begin{proof}
The proof of the lemma follows immediately by combining Lemma 7.10, Corollary 7.11 in \cite{Durham-Zalloum22} with Lemma \ref{lem:separated_iff_gate_append} above.
    
\end{proof}

\bibliography{bio}{}
\bibliographystyle{alpha}
\end{document}